\documentclass[11pt]{amsart}

\usepackage{amssymb}
\usepackage{mathtools}
\usepackage{color}

\newtheorem{theo}{Theorem} 

\newtheorem{exttheo}{Theorem} 

\newtheorem{lemma}{Lemma}[section]
\newtheorem{prop}[lemma]{Proposition}
\newtheorem{corol}[lemma]{Corollary}

\newtheorem{claim}[lemma]{Claim}

\theoremstyle{remark}
\newtheorem{remark}[lemma]{Remark}

\newtheorem{notation}[lemma]{Notation}

\theoremstyle{definition}
\newtheorem{defi}[lemma]{Definition}

\newcommand{\lin}{\textsc{l}}

\newcommand{\NN}{\mathbb{N}}
\newcommand{\RR}{\mathbb{R}}

\newcommand{\eps}{\varepsilon}

\newcommand{\AAA}{\mathcal{A}}

\newcommand{\OOO}{\mathcal{O}}
\newcommand{\PPP}{\mathcal{P}}
\newcommand{\RRR}{\mathcal{R}}
\newcommand{\SSS}{\mathcal{S}}

\newcommand{\lf}{\left}
\newcommand{\rg}{\right}

\newcommand{\out}{\rm out}

\newcommand{\tx}{\widetilde{x}}
\newcommand{\Stt}{\texttt{S}}

\newcommand{\tu}{\tilde{u}}

\newcommand{\tf}{\widetilde{f}}

\newcommand{\oveps}{\overline{\eps}}

\newcommand{\ovx}{\overline{x}}

\newcommand{\loc}{\rm loc}
\newcommand{\hdot}{\dot{H}^1}
\newcommand{\Hdot}{\dot{H}^1(\RR^N)}

\newcommand{\EMPH}[1]{\medskip\noindent\textit{#1}.}

\DeclareMathOperator{\supp}{supp}

\DeclareMathOperator{\vect}{span}

\DeclareMathOperator{\Div}{div}

\numberwithin{equation}{section} 

\title[Blow-up for energy critical wave]{Universality of the blow-up profile for small type II blow-up solutions of the energy-critical wave equation: the non-radial case}
\author[T.~Duyckaerts]{Thomas Duyckaerts$^1$}
\email{tduyckae@u-cergy.fr}
\author[C.~Kenig]{Carlos Kenig$^2$}
\email{cek@math.uchicago.edu}
\author[F.~Merle]{Frank Merle$^3$}
\email{Frank.Merle@u-cergy.fr}
\thanks{$^1$Cergy-Pontoise (UMR 8088). Partially supported by ANR Grants ONDNONLIN and ControlFlux}
\thanks{$^2$University of Chicago. Partially supported by NSF Grant DMS-0456583}
\thanks{$^3$Cergy-Pontoise (UMR 8088), IHES, CNRS. Partially supported by ANR Grant ONDNONLIN}

\date{\today}

\begin{document}
\begin{abstract}
 Following our previous paper in the radial case, we consider type II blow-up solutions to the energy-critical focusing wave equation. Let $W$ be the unique radial positive stationary solution of the equation. Up to the symmetries of the equation, under an appropriate smallness assumption, any type II blow-up solution is asymptotically a regular solution plus a rescaled Lorentz transform of $W$ concentrating at the origin.
\end{abstract}
\maketitle

\section{Introduction}
Consider the focusing energy-critical wave equation on an interval $I$ ($0\in I$)
\begin{equation}
\label{CP}
\left\{ 
\begin{gathered}
\partial_t^2 u -\Delta u-|u|^{\frac{4}{N-2}}u=0,\quad (t,x)\in I\times \RR^N\\
u_{\restriction t=0}=u_0\in \hdot,\quad \partial_t u_{\restriction t=0}=u_1\in L^2,
\end{gathered}\right.
\end{equation}
where $u$ is real-valued, $N\in \{3,4,5\}$, $L^2=L^2(\RR^N)$ and $\hdot=\Hdot$. 

The Cauchy problem \eqref{CP} is locally well-posed in $\hdot\times L^2$. This space is invariant under the scaling of the equation: if $u$ is a solution to \eqref{CP}, $\lambda>0$ and
$$ u_{\lambda}=\frac{1}{\lambda^{\frac{N-2}{2}}}u\left(\frac{t}{\lambda},\frac{x}{\lambda}\right),$$
then $u_{\lambda}$ is also a solution and $\|u_{\lambda}(0)\|_{\hdot}=\|u_0\|_{\hdot}$, $\|\partial_t u_{\lambda}(0)\|_{L^2}=\|u_1\|_{L^2}$. 

The energy
$$ E(u(t),\partial_t u(t))=\frac{1}{2}\int (\partial_t u(t,x))^2\,dx+\frac{1}{2}\int |\nabla u(t,x)|^2\,dx-\frac{N-2}{2N}\int |u(t,x)|^{\frac{2N}{N-2}}\,dx$$
and the momentum
$$ \int \partial_t u(t,x)\nabla u(t,x)\,dx$$
are independent of $t$ and also invariant under the scaling.

Let $T_+\in(0,+\infty]$ be the maximal positive time of definition for the solution $u$. 
The local well-posedness theory does not rule out type II blow-up, i.e. solutions such that $T_+<\infty$ and
\begin{equation}
\label{BlowUpII}
\sup_{t\in [0,T_+)} \|\partial_t u(t)\|^2_{L^2}+\|\nabla u(t)\|_{L^2}^2<\infty.
\end{equation}
Examples of radial type II blow-up solutions of \eqref{CP} were constructed in space dimension $N=3$ by Krieger, Schlag and Tataru \cite{KrScTa09}. Let
\begin{equation}
\label{defW}
W=\frac{1}{\left(1+\frac{|x|^2}{N(N-2)}\right)^{\frac{N-2}{2}}},
\end{equation}
which is a stationary solution of \eqref{CP}.  From \cite{KeMe08}, if $u$ is radial or $N=3,4$ and
$$ \sup_{t\in [0,T_+)}\|\nabla u(t)\|_{L^2}^2+\|\partial_t u(t)\|_{L^2}^2< \|\nabla W\|_{L^2}^2,$$
then $T_+=+\infty$ and the solution scatters forward in time, and in particular does not blow up (see Corollary \ref{C:global} below). 

The threshold $\|\nabla W\|^2_{L^2}$ is sharp in space dimension $3$. Indeed from \cite{KrScTa09}, for all $\eta_0>0$ there exists a radial type II blow-up solution such that
\begin{equation}
\label{bound_nabla0}
 \sup_{t\in [0,T_+)} \|\nabla u(t)\|_{L^2}^2+\|\partial_t u(t)\|_{L^2}^2\leq \|\nabla W\|_{L^2}^2+\eta_0.
\end{equation} 
In our previous article \cite{DuKeMe09P}, we considered type II blow-up solutions such that \eqref{bound_nabla0} holds. Our main result was the following.

If $N=3$, there exists $\eta_0>0$ such that for any \emph{radial} solution $u$ of \eqref{CP} such that $T_+(u)=T_+<\infty$ that satisfies \eqref{bound_nabla0},
there exist $(v_0,v_1)\in \hdot\times L^2$, a sign $\iota_0\in \{\pm 1\}$, and a smooth positive function $\lambda(t)$ on $(0,T_+)$ such that $\lim_{t\to T_+}\frac{\lambda(t)}{T_+-t}=0$ and, as $t\overset{\scriptscriptstyle{<}}{\to} T_+$,
\begin{equation*}
\lf(u(t),\partial_t u(t)\rg)-(v_0,v_1)-\lf(\frac{\iota_0}{\lambda(t)^{1/2}}W\left(\frac{x}{\lambda(t)}\right),0\rg)\xrightarrow[t\to T_+]{}0\text{ in }\hdot\times L^2.
\end{equation*}
In this work we extend the above result to the non-radial case. To state our main result we need to recall the following family of solutions, obtained as Lorentz transformations of $W$:
\begin{equation}
\label{def_Wl}
 W_{\ell}(t,x)=W\left(\frac{x_1-t \ell}{\sqrt{1-\ell^2}},\overline{x}\right)=\left(1+\frac{(x_1-t\ell)^2}{N(N-2)(1-\ell^2)}+\frac{|\ovx|^2}{N(N-2)}\right)^{-\frac{N-2}{2}},
\end{equation} 
where $\overline{x}=(x_2,\ldots,x_N)$ and $-1<\ell<1$. Denote by $\vec{e}_1$ the unit vector $(1,0,\ldots,0)\in \RR^N$.
Then:
\begin{theo}
\label{T:classification}
Assume that $N=3$ or $N=5$ and let $\eta_0>0$ be a small parameter. Let $u$ be a solution of \eqref{CP} such that $T_+=T_+(u)<\infty$ and 
\begin{equation}
\label{bound_nabla} 
\limsup_{t\in [0,T_+)} \|\nabla u(t)\|_{L^2}^2+\frac {N-2}2\|\partial_t u(t)\|_{L^2}^2\leq \|\nabla W\|_{L^2}^2+\eta_0.
\end{equation} 
Then, after a rotation and a translation of the space $\RR^N$, there exist $(v_0,v_1)\in \hdot\times L^2$, a sign $\iota_0\in\{\pm 1\}$, a small real parameter $\ell$  and smooth functions $x(t)\in \RR^N$, $\lambda(t)>0$ defined for $t\in(0,T_+)$, such that
\begin{equation*}
\lf(u(t),\partial_t u(t)\rg)-(v_0,v_1)-\lf(\frac{\iota_0}{\lambda(t)^{\frac{N}{2}-1}}W_{\ell}\left(0,\frac{\cdot-x(t)}{\lambda(t)}\right),\frac{\iota_0}{\lambda(t)^{\frac{N}{2}}}(\partial_t W_{\ell})\left(0,\frac{\cdot-x(t)}{\lambda(t)}\rg)\rg)
\xrightarrow[t\to T_+]{} 0
\end{equation*}
in $\hdot\times L^2$ and
\begin{equation}
 \label{classif_parameters}
\lim_{t\to T_+} \frac{\lambda(t)}{T_+-t}=0,\quad \lim_{t\to T_+}\frac{x(t)}{T_+-t}=\ell \vec{e}_1,\quad|\ell|\leq C\eta_0^{1/4}.
\end{equation}
\end{theo}
\begin{remark}
 Note that using Lorentz transform and a localization argument on the solutions of \cite{KrScTa09}, it is possible, for any $\ell\in (-1,+1)$, to construct a solution of \eqref{CP} satisfying the conclusion of Theorem \ref{T:classification}.
\end{remark}
\begin{remark}
The restriction to small dimensions in Theorem \ref{T:classification}, due to regularity issues on the local Cauchy problem for \eqref{CP}, can be removed (at least for odd dimensions) using harmonic analysis methods (see \cite{BuCzLiPaZh09P}).

The restriction to odd dimensions is only coming from Proposition \ref{P:linear} on the behaviour of solutions to the linear wave equation. In dimension $4$, our proof shows a weaker result, namely that there exist (after space rotation), a small parameter $\ell$ and sequences $t_n\to T_+$, $\lambda_n\to 0^+$, $x_n\in \RR^4$ such 
\begin{equation*}
\lf(\lambda_n u(t_n,\lambda_n \cdot+x_n),\lambda_n^2 \partial_t u(t_n,\lambda_n \cdot+x_n)\rg)
\xrightharpoonup[n\to\infty]{} \pm \left(W_{\ell}(0),\partial_t W_{\ell}(0)\right),
\end{equation*}
weakly in $\hdot\times L^2$.
\end{remark}
\begin{remark}
 The constant $\frac{N-2}{2}$ in front of $\|\partial_tu\|_{L^2}^2$ in \eqref{bound_nabla} is necessary in nonradial situations (see also Corollary \ref{C:global} below). For radial data it can be replaced by any small positive constant (see Corollary \ref{C:global} and Remark \ref{R:remark_corol} below).
\end{remark}

One important ingredient of the proof of Theorem \ref{T:classification} is the classification of non-radial solutions that are compact up to modulation under an appropriate smallness assumption:
\begin{theo}
\label{T:compact}
Assume $N\in\{3,4,5\}$.
Let $u$ be a nonzero solution of \eqref{CP} with maximal interval of definition $I_{\max}$ such that there exists functions $\lambda(t)$, $x(t)$ defined for $t\in I_{\max}$ such that 
\begin{equation}
\label{def_K}
K=\Big\{\left(\lambda(t)^{\frac{N}{2}-1}u(t,\lambda(t)x+x(t)),\lambda(t)^{\frac{N}{2}}\partial_t u(t,\lambda(t)x+x(t))\right)\;:\; t\in I_{\max}\Big\} 
\end{equation} 
has compact closure in $\hdot\times L^2$. Assume furthermore
\begin{equation}
\label{bound_nabla2W}
\sup_{t\in I_{\max}}\int |\nabla u(t)|^2<\frac{4\sqrt{N-1}}{N}\int |\nabla W|^2.
\end{equation}
Then $I_{\max}=\RR$ and there exist $\ell\in(-1,+1)$, a rotation $\RRR$ of $\RR^N$, $\lambda_0>0$, $X_0\in  \RR^N$ and a sign $\iota_0\in\{\pm 1\}$ such that
\begin{equation} 
\label{conclusion_compactness}
u(t,x)=\frac{\iota_0}{\lambda_0^{\frac{N-2}{2}}}W_{\ell}\left(\frac{t}{\lambda_0},\frac{\RRR(x)-X_0}{\lambda_0}\right).
\end{equation}
\end{theo}
\begin{remark}
Note that the constant $\frac{4\sqrt{N-1}}{N}$ in \eqref{bound_nabla2W} is always greater than $1$ if $N\in\{3,4,5\}$. Indeed, this constant is equal to $\frac{4\sqrt{2}}{3}\approx 1.89$ if $N=3$, to $\sqrt{3}\approx 1.73$ if $N=4$ and to $1.6$ if $N=5$.
\end{remark}
\begin{remark}
\label{R:compactness}
The parameter $\ell$ and the rotation $\RRR$ in \eqref{conclusion_compactness} are given by the energy and the conserved momentum of $u$. Namely, under the asumptions of Theorem \ref{T:compact}, $E(u_0,u_1)\geq E(W,0)$, $|\ell|=\left|\int \nabla u_0 u_1\right|/E(u_0,u_1)$, and 
$$ u(t,x)=\frac{\iota_0}{\lambda_0^{\frac{N-2}{2}}}W_{\ell}\left(\frac{t}{\lambda_0},\frac{x-X_0}{\lambda_0}\right)$$ 
after a space rotation around the origin chosen so that
\begin{equation}
\label{def_ell}
\ell \vec{e}_1=-\frac{\int \nabla u_0 u_1}{E(u_0,u_1)}.
\end{equation} 
\end{remark}
We next give a Corollary to Theorem \ref{T:compact}, which corrects \cite[Corollary 7.4]{KeMe08} (stated without a proof in \cite{KeMe08}) for nonradial solutions. For $N=5$, in the non-radial case, the solutions $W_{\ell}$ for small $\ell\neq 0$ give a counterexample to \cite[Corollary 7.4]{KeMe08}, as can be seen using the first line of Claim \ref{C:value_W}
below.
\begin{corol}
 \label{C:global}
Assume $N\in\{3,4,5\}$.
Let $u$ be a solution of \eqref{CP} which satisfies
\begin{equation}
\label{bound_nabla_corol}
\limsup_{t\to T_+(u)} \left[\|\nabla u(t)\|_{L^2}^2+\frac{N-2}{2}\|\partial_tu(t)\|_{L^2}^2\right]<\|\nabla W\|_{L^2}^2. 
\end{equation} 
Then $T_+(u)=+\infty$ and $u$ scatters forward in time. If $u$ is radial, \eqref{bound_nabla_corol} can be replaced by the following bound
\begin{equation}
\label{bound_nabla_corol_radial}
\limsup_{t\to T_+(u)} \|\nabla u(t)\|_{L^2}^2<\|\nabla W\|_{L^2}^2. 
\end{equation} 
\end{corol}
A more general version of Corollary \ref{C:global} is given in Corollary \ref{C:below_threshold} below.
\begin{remark}
\label{R:generalization}
 Note that because of the variational estimate \eqref{variational} below, Corollary \ref{C:global} is in fact a generalization of \cite[Theorem 1.1, i)]{KeMe08}. Note also that for $N=3$, the statement is stronger than the one of 
\cite[Corollary 7.4]{KeMe08}.
\end{remark}
Let us also a give a correct version of the second statement in \cite[Corollary 7.5]{KeMe08}.
\begin{corol}
\label{C:blowup}
Let $u$ be a solution of \eqref{CP} such that $T_+(u)<\infty$ and
$$\limsup_{t\to T_+(u)}\|\nabla u(t)\|_{L^2}^2+\|\partial_tu(t)\|_{L^2}^2 <\infty.$$
Then there exist sequences $x_n\in \RR^3$, $t_n\to T_+(u)$ such that for all $R$, 
$$ \lim_{n\to \infty}\int_{|x-x_n|\leq R} \left(|\nabla u(t_n,x)|^2+\frac{N-2}{2}|\partial_tu(t_n,x)|^2\right)dx\geq \int |\nabla W|^2\,dx.$$
\end{corol}
Corollary \ref{C:blowup} follows from the arguments in \cite[section 3]{DuKeMe09P} (see also the beginning of Section \ref{S:universality} below) and we will omit its proof.

For more comments about results of the type of Theorem \ref{T:classification}, we refer to the introduction of \cite{DuKeMe09P}. 
Theorem \ref{T:classification} is an analogue for the energy-critical wave equation of the result of \cite{MeRa05} about the mass-critical nonlinear Schr\"odinger equation. We next list other previous related works that are also discussed in the introduction of \cite{DuKeMe09P}: for works about nonlinear wave maps see e.g. \cite{ChTZ93, ShTZ97,Struwe02,Struwe03,RoSt06P,KrScTa08,StTa09P,KrSc09P,RaRo09P}; for articles about classification of solutions for other equations we refer for example  to \cite{MaMe00,MaMe01,MaMe03,MeRa04,CaFr86,MeZa07,MeZa08P}. 

Let us give a short sketch of the proof of Theorem \ref{T:classification}.  This proof is based on a new strategy which allows us to treat the non-radial case, and also simplifies the proof of the radial case in \cite{DuKeMe09P}. 

In a first step (see Subsection \ref{SS:compactness}), looking at a minimal element among the non-scattering profiles associated to sequences $(u(t_n'),\partial_t u(t_n'))$ (where $t_n'\to T_+$), we get a sequence $t_n\to T_+$ such that for some parameters $\lambda_n$, $x_n$,
\begin{equation}
\label{weak_CV_tn}
\lf(\lambda_n^{\frac N2-1} u(t_n,\lambda_n \cdot+x_n),\lambda_n^{\frac{N}{2}} \partial_t u(t_n,\lambda_n \cdot+x_n)\rg)
\xrightharpoonup[t\to T_+]{} \left(U_0,U_1\right),
\end{equation}
weakly in $\hdot\times L^2$, 
where the solution $U$ of \eqref{CP} with initial condition $(U_0,U_1)$ is compact up to the symmetries of \eqref{CP}, as in Theorem \ref{T:compact}. 

The second step of the proof of Theorem \ref{T:classification} is Theorem \ref{T:compact}, which implies that $U$ must be $W_{\ell}$ up to the symmetries. The proof of Theorem \ref{T:compact}, postponed to Section \ref{S:compact}, is a refinement of the proof of its radial analogue (see \cite{DuKeMe09P}), which was based on techniques developped in \cite{DuMe08}. To treat the non-radial case we introduce new monotonic quantities which are non-symmetric in the space variables. We also prove in \S \ref{SS:global} a more general version of Corollary \ref{C:global} which is also needed in Section \ref{S:universality}. Let us mention that Section \ref{S:compact} is independent of Section \ref{S:universality}. 

In a third step of the proof (see Subsections \ref{SS:strongCV} and \ref{SS:alltimes}), we show that the weak convergence \eqref{weak_CV_tn} is indeed a strong convergence in $\{|x|\leq T_+-t_n\}$. It is here that Proposition \ref{P:linear} on the behavior of solutions to the linear wave equation is used. We then conclude using the minimality of the profile associated to $t_n$ that this strong convergence also holds for all times as $t\to T_+$.

In addition to the parts of the paper mentioned above, Section \ref{S:preliminaries} is devoted to some preliminaries about the Cauchy problem,  profile decomposition, the solution $W_{\ell}$, and Proposition \ref{P:linear} on the localization of the solutions to the linear wave equation. The appendix concerns modulation theory around $W_{\ell}$. 

\subsection*{Notations}
In all the paper, we assume $N\in \{3,4,5\}$ unless otherwise mentioned.
We write $a\lesssim b$ or $a=\OOO(b)$ when the two positive quantities $a$ and $b$ satisfy $a\leq C b$ for some large constant $C>0$, and $a\approx b$ when $a\lesssim b$ and $b\lesssim a$. We also use the notation $a=o(b)$ when $a/b$ goes to $0$.

\subsection*{Acknowledgment}
The authors would like to thank the referee for helpful comments and suggestions, and Pr. Changxing Miao and Dr. Jiqiang Zheng for signaling an error in the statement of Theorem \ref{T:compact} in a previous version of this paper. 

\section{Preliminaries}
\label{S:preliminaries}

\subsection{Cauchy problem}
The Cauchy problem for equation \eqref{CP} was developped in \cite{Pecher84,GiSoVe92,LiSo95,ShSt94,ShSt98,Sogge95,Kapitanski94}.
If $I$ is an interval, we denote by
$$  S(I)=L^{\frac{2(N+1)}{N-2}}\left(I\times\RR^N\right), \;W(I)=L^{\frac{2(N+1)}{N-1}}\left(I\times\RR^N\right).$$
Let $\Stt_{\lin}(t)$ be the one-parameter group associated to the linear wave equation. By definition, if $(v_0,v_1)\in \hdot\times L^2$ and $t\in \RR$, $v(t)=\Stt_{\lin}(t)(v_0,v_1)$ is the solution of
\begin{gather}
 \label{lin_wave}
\partial_t^2 v-\Delta v=0,\\
\label{lin_wave_IC}
v_{\restriction t=0}=v_0,\quad
\partial_t v_{\restriction t=0}=v_1.
\end{gather}
We have
$$ \Stt_{\lin}(t)(v_0,v_1)=\cos(t\sqrt{-\Delta})v_0+\frac{1}{\sqrt{-\Delta}}\sin(t\sqrt{-\Delta})v_1.$$
By Strichartz and Sobolev estimates,
\begin{equation}
\label{strichartz}
\lf\|v\rg\|_{S(\RR)}+\lf\|D_x^{1/2}v\rg\|_{W(\RR)}+\lf\|D_x^{-1/2}\partial_t v\rg\|_{W(\RR)}\leq C_S\left(\|v_0\|_{\hdot}+\|v_1\|_{L^2}\right).
\end{equation} 
A solution of \eqref{CP} on an interval $I$, where $0\in I$, is a function $u\in C^0(I,\hdot)$ such that  $\partial_t u\in C^0(I,L^2)$, 
\begin{equation}
\label{finite_norms}
J\Subset I\Longrightarrow \|u\|_{S(J)}+\|D_x^{1/2}u\|_{W(J)}+\lf\|D_x^{-1/2}\partial_t u\rg\|_{W(J)}<\infty
\end{equation}
satisfying the Duhamel formulation
\begin{equation}
\label{solution}
u(t)=\Stt_{\lin}(t)(u_0,u_1)+\int_0^t \frac{\sin\big((t-s)\sqrt{-\Delta}\big)}{\sqrt{-\Delta}}\left(|u(s)|^{\frac{4}{N-2}}u(s)\right)ds.
\end{equation} 
We recall that for any initial condition $(u_0,u_1)\in \hdot\times L^2$, there is an unique solution $u$, defined on a maximal interval of definition $I_{\max}(u)=(T_-(u),T_+(u))$. Furthermore, $u$ satisfies the blow-up criterion 
\begin{equation}
\label{FBUC}
T_+(u)<\infty\Longrightarrow \|u\|_{S(0,T_+(u))}=+\infty. 
\end{equation}
As a consequence, if 
$\|u\|_{S(0,T_+)}<\infty$, then $T_+=+\infty$. Furthermore in this case, the solution scatters forward in time in $\hdot\times L^2$: there exists a solution $v$ of the linear equation \eqref{lin_wave} such that
$$\lim_{t\to +\infty}\|u(t)-v(t)\|_{\hdot}+\|\partial_t u(t)-\partial_t  v(t)\|_{L^2}=0.$$
 Of course an analoguous statement holds backward in time also. 

If $\left\|\Stt_{\lin}(\cdot)(u_0,u_1)\right\|_{S(I)}=\delta<\delta_1$, for some small $\delta_1$, then $u$ is globally defined and close to the linear solution with initial condition $(u_0,u_1)$ in the following sense: if $A=\left\|D_x^{1/2}\Stt_{\lin}(\cdot)(u_0,u_1)\right\|_{W(I)}$, we have  
\begin{multline}
\label{almost_linear}
\left\|u(\cdot)-\Stt_{\lin}(\cdot)(u_0,u_1)\right\|_{S(I)}\\+\sup_{t\in I}\left(\left\|u(t)-\Stt_{\lin}(t)(u_0,u_1)\right\|_{\hdot}+\left\|\partial_tu(t)-\partial_t(\Stt_{\lin}(t)(u_0,u_1))\right\|_{L^2}\right)
\leq CA\delta^{\frac{4}{N-2}},
\end{multline} 
(see for example \cite{KeMe08}, proof of Theorem 2.7).

We next recall the profile decomposition of H.~Bahouri and P.~G\'erard \cite{BaGe99}. This paper is written in space dimension $N=3$ but the results stated below hold in all dimension $N\geq 3$ (see \cite{Bulut09P}). See also \cite{BrCo85} and \cite{Li85Reb} for the elliptic case and \cite{MeVe98} for the Schr\"odinger equation.

Consider a sequence $(v_{0,n},v_{1,n})_n$ which is bounded in $\hdot\times L^2$. 
Let $(U^j_{\lin})_{j\geq 1}$ be a sequence of solutions of the linear equation \eqref{lin_wave},
with initial data $(U^j_0,U^j_1)\in \hdot\times L^2$, and $(\lambda_{j,n};x_{j,n};t_{j,n})\in (0,+\infty)\times \RR^N\times \RR$, $j\geq 1$, $n\in \NN$, be a family of parameters satisfying the pseudo-orthogonality relation
\begin{equation}
\label{ortho_param}
j\neq k\Longrightarrow \lim_{n\to \infty} \frac{\lambda_{j,n}}{\lambda_{k,n}}+\frac{\lambda_{k,n}}{\lambda_{j,n}}+\frac{|t_{j,n}-t_{k,n}|}{\lambda_{j,n}}+\frac{\left|x_{j,n}-x_{k,n}\right|}{\lambda_{j,n}}=+\infty.
\end{equation} 
We say that $(v_{0,n},v_{1,n})_n$ admits a profile decomposition  $\lf\{U_{\lin}^j\rg\}_j$, $\lf\{\lambda_{j,n};x_{j,n};t_{j,n}\rg\}_{j,n}$ when 
\begin{equation}
\label{decompo_profil}
\left\{\begin{aligned}
 v_{0,n}(x)&=\sum_{j=1}^J \frac{1}{\lambda_{j,n}^{\frac{N-2}{2}}}U_{\lin}^j\left(\frac{-t_{j,n}}{\lambda_{j,n}},\frac{x-x_{j,n}}{\lambda_{j,n}}\right)+w_{0,n}^J(x),\\
v_{1,n}(x)&=\sum_{j=1}^J \frac{1}{\lambda_{j,n}^{\frac{N}{2}}}\partial_t U_{\lin}^j\left(\frac{-t_{j,n}}{\lambda_{j,n}},\frac{x-x_{j,n}}{\lambda_{j,n}}\right)+w_{1,n}^J(x),
\end{aligned}\right.
\end{equation}
with
\begin{equation}
\label{small_w}
\lim_{n\rightarrow+\infty}\limsup_{J\rightarrow+\infty} \left\|w_{n}^J\right\|_{S(\RR)}=0,
\end{equation}
where $w_{n}^J$ is the solution of \eqref{lin_wave} with initial conditions $(w_{0,n}^J,w_{1,n}^J)$. 
Then:
\begin{prop}[\cite{BaGe99,Bulut09P}]
\label{P:profile}
If the sequence $(v_{0,n},v_{1,n})_n$ is bounded in the energy space $\hdot\times L^2$, there always exists a subsequence of $(v_{0,n},v_{1,n})_n$ which admits a profile decomposition. Furthermore, 
\begin{equation}
\label{weak_CV_wJ}
j\leq J\Longrightarrow 
\left(\lambda_{j,n}^{\frac{N-2}{2}} w_n^J\left(t_{j,n},x_{j,n}+\lambda_{j,n}y\right),\lambda_{j,n}^{\frac{N}{2}} \partial_tw_n^J\left(t_{j,n},x_{j,n}+\lambda_{j,n}y\right)\right)\xrightharpoonup[n\to \infty]{}0,
\end{equation}
weakly in $\hdot_y\times L^2_y$, and the following Pythagorean expansions hold for all $J\geq 1$
\begin{gather}
\label{pythagore1a} 
\left\|v_{0,n}\rg\|_{\hdot}^2=\sum_{j=1}^J \left\|U^j_{\lin}\left(\frac{-t_{j,n}} {\lambda_{j,n}}\right)\rg\|_{\hdot}^2+\left\|w_{0,n}^J\rg\|_{\hdot}^2+o_n(1)\\
\label{pythagore1b} 
\lf\|v_{1,n}\right\|^2_{L^2}=\sum_{j=1}^J \lf\|\partial_t U^j_{\lin}\left(\frac{-t_{j,n}} {\lambda_{j,n}}\right)\right\|^2_{L^2}+\lf\|w_{1,n}^J\right\|^2_{L^2}+o_n(1)\\
\label{pythagore2}
E(v_{0,n},v_{1,n})=\sum_{j=1}^J E\left(U^j_{\lin}\left(-\frac{t_{j,n}}{\lambda_{j,n}}\right),\partial_t U^j_{\lin}\left(-\frac{t_{j,n}} {\lambda_{j,n}}\right)\right)+E\left(w_{0,n}^J, w_{1,n}^J\right)+o_n(1).
\end{gather}
 \end{prop}
\begin{notation}
\label{N:nonlinearprofiles}
Consider a profile decomposition with profiles $U^j_{\lin}$ and parameters $\big\{\lambda_{j,n}$;$t_{j,n}$;$x_{j,n}\big\}$, and assume after extraction of a subsequence that $t_{j,n}/\lambda_{j,n}$  has a limit in $\RR\cup\{-\infty,+\infty\}$. 
We will denote  by $\left\{U^j\right\}$ the non-linear profiles associated with $\left(U^j_{\lin},\left\{\frac{-t_{j,n}}{\lambda_{j,n}}\right\}_n\right)$, which are the unique solutions of \eqref{CP} such that for large $n$, $\frac{-t_{j,n}}{\lambda_{j,n}}\in I_{\max}\left(U^j\right)$ and
$$ \lim_{n\rightarrow +\infty} \left\|U^j\left(\frac{-t_{j,n}}{\lambda_{j,n}}\right)-U^j_{\lin}\left(\frac{-t_{j,n}}{\lambda_{j,n}}\right)\right\|_{\hdot}
+\left\|\partial_t U^j\left(\frac{-t_{j,n}}{\lambda_{j,n}}\right)-\partial_t U^j_{\lin}\left(\frac{-t_{j,n}}{\lambda_{j,n}}\right)\right\|_{L^2}=0.$$
The proof of the existence of $U^j$ follows from the local existence for \eqref{CP} if this limit is finite, and from the existence of wave operators for equation \eqref{CP} if $t_{j,n}/\lambda_{j,n}$ tends to $\pm\infty$. 
\end{notation}
By the Strichartz inequalities on the linear problem and the small data Cauchy theory, if 
$\lim_{n\rightarrow +\infty} \frac{-t_{j,n}}{\lambda_{j,n}}=+\infty$, then 
$T_+\left(U^j\right)=+\infty$ and
\begin{equation}
\label{finite_norm2}
 s_0>T_-\left(U^j\right) \Longrightarrow \|U^j\|_{S(s_0,+\infty)}<\infty,
\end{equation} 
an analoguous statement holds in the case $\lim_{n\rightarrow +\infty} \frac{t_{j,n}}{\lambda_{j,n}}=+\infty$.

We will need the following approximation result, which follows from a long time perturbation theory result for \eqref{CP} and is an adaptation to the focusing case of the result of Bahouri-G\'erard (see the Main Theorem p. 135 in \cite{BaGe99}). We refer for \cite{BaGe99} for the proof in the defocusing case and \cite[Proposition 2.8]{DuKeMe09P} for a sketch of proof.

\begin{prop}
\label{P:lin_NL}
 Let $\{(v_{0,n},v_{1,n})\}_n$ be a bounded sequence in $\hdot\times L^2$, which admits the profile decomposition \eqref{decompo_profil}. Let $\theta_n\in (0,+\infty)$. Assume 
\begin{equation}
\label{bounded_strichartz}
\forall j\geq 1, \quad\forall n,\quad\frac{\theta_n-t_{j,n}}{\lambda_{j,n}}<T_+(U^j)\quad \text{and} \quad\limsup_{n\rightarrow +\infty} \left\|U^j\right\|_{S\big(-\frac{t_{j,n}}{\lambda_{j,n}},\frac{\theta_n-t_{j,n}}{\lambda_{j,n}}\big)}<\infty.
\end{equation}
Let $u_n$ be the solution of \eqref{CP} with initial data $(v_{0,n},v_{1,n})$.
Then for large $n$, $u_n$ is defined on $[0,\theta_n)$,
\begin{equation}
\label{NL_bound}
\limsup_{n\rightarrow +\infty}\|u_n\|_{S(0,\theta_n)}<\infty,
\end{equation} 
and
\begin{equation}
\label{NL_profile} 
\forall t\in [0,\theta_n),\quad
u_n(t,x)=\sum_{j=1}^J \frac{1}{\lambda_{j,n}^{\frac{N-2}{2}}}U^j\left(\frac{t-t_{j,n}}{\lambda_{j,n}},\frac{x-x_{j,n}}{\lambda_{j,n}}\right)+
w^J_{n}(t,x)+r^J_n(t,x),
 \end{equation}
where 
\begin{equation}
\label{cond_rJn}
\lim_{n\rightarrow +\infty} \limsup_{J\rightarrow +\infty} \left[\|r^J_n\|_{S(0,\theta_n)}+\sup_{t\in (0,\theta_n)} \left(\|\nabla r^J_n(t)\|_{L^2}+\|\partial_t r^J_n(t)\|_{L^2}\right)\right]=0.
\end{equation}
An analoguous statement holds if $\theta_n<0$.
\end{prop}

\subsection{Elliptic properties of the stationary solution and the solitary wave}
We first recall a variational claim from \cite{KeMe08}:
\begin{claim}
\label{C:variational}
Let $v\in \hdot$. Then
\begin{equation}
 \label{variational}
\|\nabla v\|_{L^2}^2\leq \|\nabla W\|_{L^2}^2\text{ and }E(v,0)\leq E(W,0)\Longrightarrow \|\nabla v\|_{L^2}^2\leq \frac{\|\nabla W\|_{L^2}^2}{E(W,0)}E(v,0)=N E(v,0).
\end{equation}
Furthermore, there is a constant $c>0$ such that if for some small $\eps>0$, $\eps\leq \|\nabla v\|_{L^2}^2\leq \left(\frac{N}{N-2}\right)^{\frac{N-2}{2}}\|\nabla W\|_{L^2}^2-\eps$, then $E(v,0)\geq c\eps$.
\end{claim}
\begin{proof}
 The first part of the Claim is shown in \cite{KeMe08}. 
For the second part, write
$$ E(v,0)=\frac{1}{2}\int |\nabla v|^2-\frac{N-2}{2N}\int |v|^{\frac{2N}{N-2}}\geq \frac{1}{2}\int |\nabla v|^2-\frac{N-2}{2N}C_N^{\frac{2N}{N-2}} \left(\int |\nabla v|^2\right)^{\frac{N}{N-2}},$$
where $C_N=\left(\int |\nabla W|^2\right)^{-1/N}$ is the best constant in the Sobolev inequality $\|v\|_{\frac{2N}{N-2}}\leq C_N\|\nabla v\|_{L^2}$. Let $y=\int|\nabla v|^2$. Then
$$E(v,0)\geq \frac{1}{2} y-\frac{N-2}{2N} C_N^{\frac{2N}{N-2}}y^{\frac{N}{N-2}}=f(y).$$
The equation $f(y)=0$ has two solutions, $y=0$ and $y^*=\left(\frac{N}{N-2}\right)^{\frac{N-2}{2}}\int|\nabla W|^2$, and the statement follows from the fact that $f'(0)\neq 0$ and $f'(y^*)\neq 0$.
\end{proof}

In the following, we will consider the solitary wave solutions of \eqref{CP}, which are obtained from $W$ by a Lorentz transform
$$W_{\ell}(t,x)=W\left(\frac{x_1-t \ell}{\sqrt{1-\ell^2}},\overline{x}\right)=\left(1+\frac{(x_1-t\ell)^2}{N(N-2)(1-\ell^2)}+\frac{|\ovx|^2}{N(N-2)}\right)^{\frac{N-2}{2}},$$
where $\ell\in (-1,1)$.
We have:
\begin{claim}
 \label{C:value_W}
\begin{gather*}
\forall t,\quad \int |\nabla W_{\ell}(t)|^2=\frac{N+(1-N)\ell^2}{N\sqrt{1-\ell^2}}\int |\nabla W|^2\quad \text{and}\quad\int (\partial_t W_{\ell}(t))^2=\frac{\ell^2}{N\sqrt{1-\ell^2}}\int |\nabla W|^2\\
\forall t,\quad \int |\nabla W_{\ell}(t)|^2+\frac{N-2}{2} \int (\partial_tW_{\ell}(t))^2\geq \left(1+\frac{\ell^4}{8}\right)\int |\nabla W|^2\\
 E\lf(W_{\ell}(0),\partial_t W_{\ell}(0)\rg)=\frac{1}{\sqrt{1-\ell^2}}E(W,0)\\
\int \nabla W_{\ell}(0)\partial_t W_{\ell}(0)=-\frac{\ell}{\sqrt{1-\ell^2}}E(W,0)\vec{e}_1=-\ell E(W_{\ell}(0),\partial_tW_{\ell}(0))\vec{e}_1.
\end{gather*}
\end{claim}
\begin{proof}[Sketch of proof]
All statements follow from explicit computations. To get the second line, notice that by the first line,
$$ \int |\nabla W_{\ell}(t)|^2+\frac{N-2}{2} \int (\partial_tW_{\ell}(t))^2-\int|\nabla W|^2=\frac{1}{\sqrt{1-\ell^2}}\left[1-\sqrt{1-\ell^2}-\frac{1}{2}\ell^2\right]\int |\nabla W|^2,$$
and use the standard inequality $\sqrt{1-x}\leq 1-\frac{1}{2}x-\frac{1}{8}x^2$ for $0\leq x<1$.
\end{proof}

We next state an uniqueness result for an asymmetric elliptic equation:
\begin{lemma}
\label{L:asym_ell}
Let $f\in \hdot(\RR^N)\setminus\{0\}$ and $\ell\in\RR$. Assume
\begin{equation}
\label{asym_ell}
(1-\ell^2)\partial_{x_1}^2 f+\sum_{j=2}^N \partial_{x_j}^2f+|f|^{\frac{4}{N-2}}f=0,
\end{equation} 
and
\begin{equation}
\label{bound_ell}
\int |\nabla f|^2<\frac{4\sqrt{N-1}}{N}\int |\nabla W|^2.
\end{equation} 
Then $\ell^2<1$ and there exist $\lambda>0$, $X\in \RR^N$ and a sign $\pm$ such that
$$ f(x)=\pm \frac{1}{\lambda^{\frac{N}{2}-1}}W_{\ell}\left(0,\frac{x-X}{\lambda}\right).$$
\end{lemma}
\begin{proof}
\EMPH{Case $\ell^2=1$}
In this case $f$ solves the equation $\Delta_{\overline{x}}f+|f|^{\frac{4}{N-2}}f=0$, where $\overline{x}=(x_2,\ldots,x_N)$ and we have (for almost every $x_1$) that $f(x_1,\ldots)\in \hdot\lf(\RR^{N-1}\rg)$, $f(x_1,\ldots)\in L^{2^*}(\RR^{N-1})$, $2^*=\frac{2N}{N-2}$. Fix such an $x_1$ and let $F(x_2,\ldots,x_N)=f(x_1,x_2,\ldots,x_N)$. We will show that $F=0$, using the Pohozaev identity in dimension $N-1$.

Until the end of this step we write $x=(x_2,\ldots,x_N)$ and $n=N-1$ to simplify notation. 
By elliptic regularity $F\in C^{2}(\RR^{n})$. Furthermore,
\begin{equation*}
 \Div\lf(x|\nabla F|^2\rg)=n|\nabla F|^2+2\sum_{i,j} x_i \frac{\partial^2F}{\partial{x_i}\partial{x_j}}\frac{\partial F}{\partial{x_j}},
\end{equation*}
and
\begin{multline*}
 2\Div\lf((x\cdot\nabla F)\nabla F\rg)=2(x\cdot \nabla F)\Delta F +2\nabla (x\cdot \nabla F)\cdot \nabla F\\
=-2(x\cdot \nabla F)|F|^{\frac{4}{N-2}}F+2\sum_{i,j}x_i\frac{\partial^2 F}{\partial{x_i}\partial{x_j}}\frac{\partial F}{\partial{x_j} }+2|\nabla F|^2.
\end{multline*}
Hence
\begin{equation*}
 \Div\lf(x|\nabla F|^2\rg)-2\Div\lf((x\cdot\nabla F)\nabla F\rg)=(n-2)|\nabla F|^2+2x\cdot\nabla\left(\frac{|F|^{2^*}}{2^*}\right).
\end{equation*}
Let $\varphi\in C_0^{\infty}(\RR^n)$, such that $\varphi(x)=1$ if $|x|\leq 1$ and $\varphi(x)=0$ if $|x|\geq 2$. Let $
\varphi_R(x)=\varphi(x/R)$. Then
\begin{multline*}
 \Div\lf(x\varphi_R|\nabla F|^2\rg)-2\Div\lf((x\cdot\nabla F)\nabla F\varphi_R\rg)\\
=(n-2)\varphi_R|\nabla F|^2+2\varphi_R x\cdot\nabla\left(\frac{|F|^{2^*}}{2^*}\right)+x\cdot\nabla\varphi_R\,|\nabla F|^2-2(\nabla F\cdot\nabla\varphi_R)(x\cdot \nabla F).
\end{multline*}
Next,
\begin{equation*}
 2\Div\left(x\varphi_R \frac{|F|^{2^*}}{2^*}\rg)=2x\cdot\nabla \varphi_R \frac{|F|^{2^*}}{2^*}+2n\varphi_R \frac{|F|^{2^*}}{2^*}+2\varphi_R\, x\cdot \nabla \left(\frac{|F|^{2^*}}{2^*}\rg).
\end{equation*}
Thus,
\begin{equation*}
 2\varphi_R\, x\cdot \nabla \left(\frac{|F|^{2^*}}{2^*}\rg)=2\Div\left(x\varphi_R \frac{|F|^{2^*}}{2^*}\rg)-2x\cdot\nabla \varphi_R \frac{|F|^{2^*}}{2^*}-2n\varphi_R \frac{|F|^{2^*}}{2^*}.
\end{equation*}
Note that $|x|\,|\nabla \varphi_R|$ is bounded independently of $R$, and when we integrate in $x$, the corresponding terms go to $0$ as $R\to +\infty$ by our assumption on $f$. When we integrate the divergence terms we get $0$. Thus, we conclude
$$ (n-2)\int |\nabla F|^2=\frac{2n}{2^*}\int|F|^{2^*}.$$
If $n=2$ we deduce that $F=0$. Otherwise, using Hardy's inequality and a cut-off, and multiplying the equation $\Delta F+|F|^{\frac{4}{N-2}}F=0$ by $F$, we see that $\int |\nabla F|^2=\int |F|^{2^*}$, so that
$$ \left(\frac{2n}{2^*}-(n-2)\right)\int|F|^{2^*}=0,$$
which gives again $F\equiv 0$. We have shown that $f(x_1,\cdot)=0$ for almost every $x_1$, which shows that $f=0$, contradicting our assumption on $f$.

\EMPH{Case $\ell^2>1$}
Assume for example $\ell>1$. Consider the function 
$$ u(t,x)=f(x_1+\ell t,x_2,\ldots, x_N),$$
which solves \eqref{CP} for all time. Note that $\nabla u(0,x)=\nabla f(x)$ and that $\partial_t u(0,x)=\ell\partial_{x_1}f(x)$, so this is a global in time solution to \eqref{CP} in the energy space. Let $\eps>0$ be given. Find $M$ so large that
$$ \int_{|x|\geq M}\left( |\nabla u(0,x)|^2+(\partial_t u(0,x))^2+\frac{|u(0,x)|^2}{|x|^2}\right)dx\leq \eps.$$
By Proposition 2.17 in \cite{KeMe08}, we have for all $t$
$$ \int_{|x|\geq \frac{3}{2}M+|t|}\left( |\nabla_xu(t,x)|^2+|\partial_t u(t,x)|^2\right)dx\leq C\eps.$$ 
Let $K$ be a compact set in $(x_2,\ldots,x_N)$ and $a<b$. If $t>0$ is large, then 
$$ x_1\in (a-\ell t,b-\ell t)\text{ and }(x_2,\ldots,x_N)\in K\Longrightarrow |x|\geq \ell t-A,$$
where $A$ is a fixed constant depending on $K$ and $(a,b)$. Pick $t$ so large that $\ell t\geq \frac{3}{2}M+t+A$, which is possible since $\ell>1$. 
Then
$$ \int_{K}\int_{a-\ell t}^{b-\ell t}|\nabla u(t,x)|^2\,dx\leq C\eps$$
while $\nabla_xu(t,x)=\nabla f(x_1+\ell t,x_2,\ldots,x_N)$, so the integral equals 
$\int_{K}\int_{a}^{b}|\nabla f(x)|^2$, which shows, since $\eps>0$ is arbitrary, that $f\equiv 0$, contradicting again our assumptions.

\EMPH{Case $\ell^2<1$}
Let
$$ g(x)=f\left(\sqrt{1-\ell^2}x_1,x_2,\ldots,x_N\right).$$

\EMPH{Step 1}
We first prove
\begin{equation}
\label{bound_g}
\int |\nabla g|^2=\frac{N\sqrt{1-\ell^2}}{1+(N-1)(1-\ell^2)}\int |\nabla f|^2\leq\frac{N}{2\sqrt{N-1}}\int |\nabla f|^2. 
\end{equation} 
Indeed, by \eqref{asym_ell}
 \begin{equation}
  \label{eqg}
  \Delta g=-|g|^{\frac{4}{N-2}}g.
 \end{equation} 
 Multiplying \eqref{eqg} by $x_j\partial_{x_j} g$, integrating over $\RR^N$, integrating by parts and using $\int |\nabla g|^2=\int |g|^{\frac{2N}{N-2}}$ we get
 \begin{equation}
 \label{symetric_g}
 \int |\partial_{x_j}g|^2=\frac{1}{N}\int |\nabla g|^2,\quad j=1\ldots N.
 \end{equation} 
 Thus
 \begin{multline*}
\int |\nabla f|^2=\frac{1}{\sqrt{1-\ell^2}}\int (\partial_{x_1}g)^2+\sum_{j=2}^N \sqrt{1-\ell^2}\int (\partial_{x_j}^2g)\\
=\left(\frac{1}{N\sqrt{1-\ell^2}}+\frac{N-1}{N}\sqrt{1-\ell^2}\right)\int |\nabla g|^2,  
 \end{multline*}
which yields \eqref{bound_g} (the second inequality follows from:
$$\max_{\ell\in (-1,1)}\frac{N\sqrt{1-\ell^2}}{1+(N-1)(1-\ell^2)}=\frac{N}{2\sqrt{N-1}},$$
where the max is attained when $1-\ell^2=\frac{1}{N-1}$).

\EMPH{Step 2}
By Step 1 and assumption \eqref{bound_ell}, $\int|\nabla g|^2<2\int |\nabla W|^2$. 
By elliptic estimates, one gets that $g$ is $C^2$. Define
$$ g_+=\max(g,0),\quad g_-=-\min(g,0)=g-g_+.$$
Then by Kato's inequality, in the sense of distribution,
$$\Delta g_++|g_+|^{\frac{4}{N-2}}g_+\geq 0.$$
As a consequence
\begin{equation}
\label{bound_nabla+}
\int |\nabla g_+|^2\leq \int \left|g_+\right|^{\frac{2N}{N-2}}. 
\end{equation} 
Similarly
\begin{equation}
\label{bound_nabla-}
\int |\nabla g_-|^2\leq \int \left|g_-\right|^{\frac{2N}{N-2}}. 
\end{equation} 
Using that 
$$ \int |\nabla g_+|^2+\int |\nabla g_-|^2=\int |\nabla g|^{2}<2\int |\nabla W|^2,$$
we get that $\int |\nabla g_{\pm}|^2<\int |\nabla W|^2$ for at least one of the signs $+$ or $-$. To fix ideas, assume that it is $-$. The bound \eqref{bound_nabla-} and Sobolev inequality implies that $g_-=0$. Indeed,
\begin{equation*}
 \int |\nabla g_-|^2\leq \int \left|g_-\right|^{\frac{2N}{N-2}}\leq \frac{\int W^{\frac{2N}{N-2}}}{\left(\int |\nabla W|^2\right)^{\frac{N}{N-2}}} \left(\int |\nabla g_-|^2\right)^{\frac{N}{N-2}}.
\end{equation*} 
Using that by the equation $\Delta W=-W^{\frac{2N}{N-2}}$, $\int W^{\frac{2N}{N-2}}=\int |\nabla W|^2$, we get that $g_-=0$ or $\int |\nabla W|^2\leq \int |\nabla g_-|^2$, and the second possibility is ruled out by our assumption on $g_-$.

This shows that $g=g_+$ is a nonnegative solution of 
$$\Delta g+|g|^{\frac{4}{N-2}}g=0,$$
and by \cite{GiNiNi81}, there exist $\lambda>0$, $X\in\RR^N$ such that
$$g(x)=\frac{1}{\lambda^{\frac{N-2}{2}}}W\lf(\frac{x-X}{\lambda}\rg).$$
Coming back to $f$, we get
$$f(x)=\frac{1}{\lambda^{\frac{N-2}{2}}}W\lf(\frac{x_1-X_1}{\lambda\sqrt{1-\ell^2}},\frac{x_2-X_2}{\lambda},\ldots,\frac{x_N-X_N}{\lambda}\rg)=
\frac{1}{\lambda^{\frac{N-2}{2}}}W_{\ell}\left(0,\frac{x-X}{\lambda}\right).$$
\end{proof}
\subsection{Linear behaviour}
\begin{prop}
 \label{P:linear}
Assume that $N\geq 3$ is odd. Let $u_0\in \hdot(\RR^N)$, $u_1\in L^2(\RR^N)$ and $u^{\lin}$ be the solution to 
\begin{gather}
\label{linear_wave}
\partial_t^2 u^{\lin}-\Delta u^{\lin}=0\\
\label{linear_data}
u^{\lin}_{\restriction  t=0}=u_0,\quad \partial_t u^{\lin}_{\restriction t=0}=u_1.
\end{gather}
Then one of the following holds
\begin{equation}
\label{out_positive}
\forall t\geq 0,\quad \int_{|x|\geq t} \lf(\lf|\nabla u^{\lin}(t,x)\rg|^2+ \lf(\partial_t u^{\lin}(t,x)\rg)^2\rg)dx\geq \frac 12 \int \lf(\lf|\nabla u_0(x)\rg|^2+ u_1(x)^2\right)dx
\end{equation}
or
\begin{equation}
\label{out_negative}
\forall t\leq 0,\quad \int_{|x|\geq -t} \lf(\lf|\nabla u^{\lin}(t,x)\rg|^2+ \lf(\partial_t u^{\lin}(t,x)\rg)^2\rg)dx\geq \frac 12 \int \lf(\lf|\nabla u_0(x)\rg|^2+ u_1(x)^2\right)dx.
\end{equation}
\end{prop}
Recall that
$ \frac{1}{2}\int_{|x|\geq |t|} \left(|\nabla u^{\lin}(t,x)|^2+(\partial_t u^{\lin}(t,x))^2\right)dx$
is a non-increasing function of $t$ for $t\geq 0$ and a non-decreasing function of $t$ for $t\leq 0$ (see e.g. \cite[p.12]{ShSt98}). Thus the following limits exist:
\begin{equation*}
E_{\pm\infty}^{\out}(u_0,u_1)=\lim_{t\to \pm \infty}\frac{1}{2}\int_{|x|\geq |t|} \left(|\nabla u^{\lin}(t,x)|^2+(\partial_t u^{\lin}(t,x))^2\right)dx.
\end{equation*}
Then Proposition \ref{P:linear} will be a consequence of the following proposition:
\begin{prop}
\label{P:linear'} 
Let $u^{\lin}$ be as in Proposition \ref{P:linear}. Then
$$ E_{+\infty}^{\out}(u_0,u_1)+E_{-\infty}^{\out}(u_0,u_1)=\frac{1}{2}\int |\nabla u_0|^2\,dx+\frac 12\int u_1^2\,dx.$$
\end{prop}
We next prove Proposition \ref{P:linear'}.
First note that we can assume by density that
\begin{equation}
\label{C_infty_0}
 (u_0,u_1)\in C_0^{\infty}(\RR^N),
\end{equation} 
and then by scaling that
\begin{equation}
\label{support_u}
\supp (u_0,u_1)\subset \{|x|\leq 1\}. 
\end{equation} 
Let us reduce the problem further, assuming \eqref{C_infty_0} and \eqref{support_u}. Let $z_1$ (respectively $z_2$) be the solution to \eqref{linear_wave} with initial condition $(u_0,0)$ (respectively $(0,u_1)$). Then
$$ z_1(-t)=z_1(t),\quad z_2(-t)=-z_2(t).$$
We deduce
$$ \int_{|x|\geq |t|} \nabla z_1(t,x)\cdot \nabla z_2(t,x)\,dx+\int_{|x|\geq |t|} \nabla z_1(-t,x)\cdot \nabla z_2(-t,x)\,dx=0$$
and similarly
$$ \int_{|x|\geq |t|} \partial_t z_1(t,x)\partial_t z_2(t,x)\,dx+\int_{|x|\geq |t|} \partial_t z_1(-t,x)\partial_t z_2(-t,x)\,dx=0$$
Developping the equality $u^{\lin}=z_1+z_2$ we get, for $t\geq 0$,
\begin{multline*}
\frac{1}{2}\int_{|x|\geq t} \left(\lf|\nabla u^{\lin}(t,x)\rg|^2+\lf(\partial_t u^{\lin}(t,x)\rg)^2\right)dx+\frac{1}{2}\int_{|x|\geq t} \left(\lf|\nabla u^{\lin}(-t,x)\rg|^2+\lf(\partial_t u^{\lin}(-t,x)\rg)^2\right)dx\\
=\int_{|x|\geq t} \left(|\nabla z_1(t,x)|^2+(\partial_t z_1(t,x))^2\right)dx+\int_{|x|\geq t} \left(|\nabla z_2(t,x)|^2+(\partial_t z_2(t,x))^2\right)dx,
\end{multline*}
and thus, letting $t\to+\infty$,
\begin{equation*}
E_{+\infty}^{\out}(u_0,u_1)+E_{-\infty}^{\out}(u_0,u_1)=2E_{+\infty}^{\out}(u_0,0)+2E_{+\infty}^{\out}(0,u_1).
\end{equation*}
The conclusion of Proposition \ref{P:linear'} will then follow from the Lemma:
\begin{lemma}
\label{L:linear}
 Let $(u_0,u_1)\in C_0^{\infty}(\RR^N)$ with $\supp (u_0,u_1)\subset \{|x|\leq 1\}$. Then
\begin{align*}
E_{+\infty}^{\out}(u_0,0)=E_{-\infty}^{\out}(u_0,0)=\frac 14 \int |\nabla u_0|^2\\
E_{+\infty}^{\out}(0,u_1)=E_{-\infty}^{\out}(0,u_1)=\frac 14 \int u_1^2.
  \end{align*} 
\end{lemma}
 We need a preliminary calculus lemma:
\begin{lemma}
 \label{L:calculus}
Let $f\in C_0^{\infty}(\RR^N)$, $t>0$ ($t$ large), $\omega_0\in \RR^N$ with $|\omega_0|=1$ and $s_0\in (0,1)$. Then
\begin{multline}
\label{calculus}
 \int_{S^{N-1}\cap \lf\{ |\omega+\omega_0|\leq \frac 2t\rg\}}f\big((t+s_0)\omega_0+t\omega\big)t^{N-1}\,d\omega\\
=\int_{S^{N-1}\cap \lf\{ |\omega-\omega_0|\leq \frac 2t\rg\}}f\big(-(t-s_0)\omega_0+t\omega\big)t^{N-1}\,d\omega+\OOO\lf(\frac{1}{t}\rg),
\end{multline}
where $\OOO$ is uniform in $\omega_0$, $s_0$.
\end{lemma}
\begin{proof}
We do an expansion of the left hand side of \eqref{calculus}, by chosing coordinates so that the origin is $s_0\omega_0$ and $\omega_0=\vec{e}_N=(0,\ldots,0,1)$. Then the set $(t+s_0)\omega_0+t\omega$, where $\omega\in S^{N-1}\cap \lf\{|\omega+\omega_0|\leq \frac 2t\rg\}$ is the set of $(y_1,\ldots ,y_N)$ (in the new coordinates) so that
\begin{equation*}
 y_N=t-\sqrt{t^2-y_1^2-\ldots-y_{N-1}^2}\text{ and } \sqrt{y_1^2+\ldots+y_N^2}\leq 2.
\end{equation*}
In particular in this set, $|y_N|\leq \frac{C}{t}$.
Using these coordinates to express the surface integral and replacing by $y_N=0$, asymptotically, and doing the corresponding argument for the integral on the right hand side,  we obtain the desired result.
\end{proof}
It remains to prove Lemma \ref{L:linear} to conclude the proof of Proposition \ref{P:linear'}.
\begin{proof}[Proof of Lemma \ref{L:linear}]
We prove the first statement, the proof of the second one is similar. By a well-known formula (see \cite[p.43]{ShSt98} for instance), the solution $z$ to \eqref{linear_wave} with data $(u_0,0)$ is given by
\begin{equation}
\label{expz1}
 z(t,x_0)=A_N \frac{\partial}{\partial t} \left(\frac 1t\frac{\partial}{\partial_t}\right)^{\frac{N-3}{2}}\left(t^{N-2}\int_{S^{N-1}} u_0(x_0+t\omega)\,d\omega\right),
\end{equation}
where $A_N$ is a constant depending on $N$.
Recalling  that $u_0\in C_0^{\infty}\lf(\{|x|<1\}\rg)$, we get (by the Huygens principle) that $\supp z(t,x_0)\subset \big\{t-1\leq |x_0|\leq t+1\big\}$. For $(t,x_0)$ in the support of $z$, write $x_0=(t+s_0)\omega_0$, $|\omega_0|=1$ and $-1<s_0<1$. From the condition on the support of $u_0$, we get that the preceding surface integrals take place on $|\omega+\omega_0|\leq\frac{2}{t}$, and thus the surface of integration is lesser than $C/t^{N-1}$ for large $t$. From \eqref{expz1}, we get the bound 
$\left|\left(\nabla z,\partial_t z\rg)\rg|\leq \frac{C}{t^{\frac{N-1}{2}}}$, for large $t$, and from the condition $|\omega+\omega_0|\leq 2/t$, 
\begin{align}
\label{surface1}
 \nabla_{x_0}z(t,x_0)&=A_N t^{\frac{N-1}{2}}\int_{S^{N-1}} \nabla\left((\omega_0\cdot \nabla )^{\frac{N-1}{2}}u_0\right)(x_0+t\omega)\,d\omega+\OOO\lf(t^{-\frac{N+1}{2}}\rg),\\
\label{surface2}
\partial_t z(t,x_0)&=A_N t^{\frac{N-1}{2}}\int_{S^{N-1}} (\omega_0\cdot \nabla )^{\frac{N+1}{2}}u_0(x_0+t\omega)\,d\omega+\OOO\lf(t^{-\frac{N+1}{2}}\rg),
\end{align}
where 
$$\lf(\omega\cdot\nabla\rg)^m u_0=\sum_{j\in \{1,\ldots,N\}^m} \omega_{j_1}\ldots\omega_{j_m}\partial_{x_{j_1}}\ldots\partial_{x_{j_m}} u_0.$$ 
(See also \cite{Christodoulou86,Klainerman86}.)
By Lemma \ref{L:calculus}, if $0<s_0<1$, 
\begin{align*}
 \nabla_{x}z\big(t,(t+s_0)\omega_0\big)&=(-1)^{\frac{N-1}{2}}\nabla_{x}z\big(t,(t-s_0)(-\omega_0)\big)+\OOO\lf(t^{-\frac{N+1}{2}}\rg)\\
\partial_t z\big(t,(t+s_0)\omega_0\big)&=(-1)^{\frac{N+1}{2}}\partial_t z\big(t,(t-s_0)(-\omega_0)\big)+\OOO\lf(t^{-\frac{N+1}{2}}\rg).
\end{align*}
Integrating, we get, for some constant $C_N$,
\begin{multline*}
 \int_{t<|x_0|<1+t} |\nabla_{x}z(t,x_0)|^2\,dx_0=C_N\int_{0\leq s_0\leq 1}\int_{S^{N-1}} \lf|\nabla_{x}z(t,(t+s_0)\omega_0)\rg|^2(t+s_0)^{N-1}\,ds_0\,d\omega_0\\
=C_Nt^{N-1}\int_{0\leq s_0\leq 1} \int_{S^{N-1}} \lf|\nabla_{x}z(t,(t+s_0)\omega_0)\rg|^2\,ds_0\,d\omega_0+\OOO\lf(\frac 1t\rg)\\
=C_Nt^{N-1}\int_{-1\leq s_0\leq 0} \int_{S^{N-1}} \lf|\nabla_{x}z(t,(t+s_0)\omega_0)\rg|^2\,ds_0\,d\omega_0+\OOO\lf(\frac 1t\rg)\\
=\int_{t-1\leq |x_0|\leq t} \int_{S^{N-1}} \lf|\nabla_{x}z(t,x_0)\rg|^2\,dx_0+\OOO\lf(\frac 1t\rg).
\end{multline*}
Arguing similarly for $\partial_t z$, we then obtain
\begin{multline*}
\int_{t-1<|x_0|<t} |\nabla_{x}z(t,x_0)|^2\,dx_0+\int_{t-1<|x_0|<t} |\partial_t z(t,x_0)|^2\,dx_0\\
=
\int_{t<|x_0|<1+t} |\nabla_{x}z(t,x_0)|^2\,dx_0+\int_{t<|x_0|<1+t} |\partial_t z(t,x_0)|^2\,dx_0+\OOO\lf(\frac{1}{t}\rg).
\end{multline*}
Letting $t\to +\infty$ and using the conservation of the energy $\frac{1}{2}\int |\nabla u_0|^2$ of $z$, we get
$$ \frac{1}{2}\int |\nabla u_0|^2-E^{\out}_{+\infty}=E^{\out}_{+\infty},$$
which concludes the proof of the first statement of the lemma.
\end{proof}
\subsection{A few identities}
We conclude this section by gathering some useful identities for solutions of \eqref{CP}, which follow from straightforward integration by parts. We define the density of energy by
\begin{equation}
\label{defeu}
e(u)(t,x)=\frac{1}{2}|\nabla u(t,x)|^2+\frac{1}{2}(\partial_t u(t,x))^2-\frac{N-2}{2N}|u|^{\frac{2N}{N-2}}.
\end{equation}
\begin{claim}
\label{C:identities}
 Let $u$ be a solution of \eqref{CP}, $k\in\{1,\ldots,N\}$, $\varphi\in C^{1}(\RR_t\times \RR^N_x,\RR)$ and $\Phi\in C^{1}(\RR_t\times \RR^N_x,\RR^N)$, both compactly supported in the space variable. Then:
\begin{align}
\label{identity1}
\frac{d}{dt}\int \varphi \,u\,\partial_t u&=\int \left((\partial_tu)^2-|\nabla u|^2+|u|^{\frac{2N}{N-2}}\right)\varphi-\int u\,\nabla u\cdot\nabla \varphi+\int u\,\partial_t u\,\partial_t\varphi\\
\label{identity2}
\frac{d}{dt}\int \varphi \,\partial_{x_k}u\,\partial_t u
&=\frac{1}{2}\int\left(-(\partial_t u)^2+|\nabla u|^2-\frac{N-2}{N}|u|^{\frac{2N}{N-2}}\right)\partial_{x_k}\varphi\\
\notag
&\qquad-\sum_{j=1}^N \int\partial_{x_k}u\,\partial_{x_j}u\,\partial_{x_j}\varphi+\int\partial_{x_k}u\,\partial_tu\,\partial_t\varphi\\
\label{identity3}
\frac{d}{dt}\int \Phi\cdot \nabla u\,\partial_t u&= \frac{1}{2}\int\left(-(\partial_t u)^2+|\nabla u|^2-\frac{N-2}{N}|u|^{\frac{2N}{N-2}}\right)\mathrm{div}\,\Phi\\
\notag
&\qquad -\sum_{j,k=1}^N \int \partial_{x_k}u\,\partial_{x_j}u\,\partial_{x_j}\Phi_k+\sum_{k=1}^N\int \partial_{x_k}u\,\partial_{t}u\,\partial_t\Phi_k\\
\label{identity4}
\frac{d}{dt}\int \varphi e(u)&=-\int \nabla\varphi\cdot\nabla u\,\partial_tu+\int \partial_t\varphi \,e(u),
\end{align}
where, $\Phi=(\Phi_1,\ldots,\Phi_N)$, $\mathrm{div}\,\Phi=\sum_k \partial_{x_k}\Phi_k$ and
all the integrals are taken over $\RR^N$ with respect to the measure $dx$.
\end{claim}
\begin{claim}
\label{C:identity2}
 Let $u$ be a solution of \eqref{CP} which has compact support in $x$. Then
\begin{align}
\label{identity1'}
\frac{d}{dt}\int u\,\partial_t u&=\int \left((\partial_tu)^2-|\nabla u|^2+|u|^{\frac{2N}{N-2}}\right)\\
\label{identity2'}
\frac{d}{dt}\int x_k\,\partial_{x_k}u\,\partial_t u
&=\frac{1}{2}\int\left(-(\partial_t u)^2+|\nabla u|^2-\frac{N-2}{N}|u|^{\frac{2N}{N-2}}\right)-\int (\partial_{x_k}u)^2\\
\label{identity3'}
\frac{d}{dt}\int x\cdot \nabla u\,\partial_t u&= -\frac{N}{2}\int (\partial_t u)^2+\frac{N-2}{2}\left(\int |\nabla u|^2-|u|^{\frac{2N}{N-2}}\right)\\
\label{identity4'}
\frac{d}{dt}\int x e(u)&=-\int \nabla u\,\partial_tu\\
\label{identity5}
\frac{d^2}{dt^2}\int u^2&=\frac{4}{N-2}\int |\nabla u|^2+\frac{4(N-1)}{N-2}\int (\partial_tu)^2-\frac{4N}{N-2}E(u_0,u_1).
\end{align}

\end{claim}

\section{Universality of the blow-up profile}
\label{S:universality}
In this section we assume Theorem \ref{T:compact} and Corollary \ref{C:below_threshold} and prove Theorem \ref{T:classification}.
We assume $N\in\{3,4,5\}$ in \S \ref{SS:compactness} and \S \ref{SS:estimates}, and $N\in\{3,5\}$ in \S \ref{SS:strongCV} and \S \ref{SS:alltimes}.
Let $u$ be a solution of \eqref{CP} that blows up in finite time and which satisfies \eqref{bound_nabla}. To simplify notations we will assume
$$ T_+=1.$$
From \cite{DuKeMe09P}, there exists a non-empty finite set $S\subset \RR^N$, called the set of singular points, such that the solution $(u,\partial_tu)$ has a strong limit in $H^{1}_{\loc}(\RR^N\setminus S)\times L^2_{\loc}(\RR^N\setminus S)$ as $t\to 1$. Furthermore, adapting the proof of \cite[Prop 3.9]{DuKeMe09P} in view of Corollary \ref{C:global}, we get
$$ \forall m\in S, \forall \eps>0,\quad \limsup_{t\to 1}\int_{|x-m|\leq \eps}|\nabla u(t)|^2+\frac{N-2}{2}|\partial_tu(t)|^2\geq \int |\nabla W|^2.$$
By \eqref{bound_nabla}, there can be only one singular point. We will assume that this singular point is $0$. Denote by $(v_0,v_1)$ the weak limit, as $t\to 1$ of $(u(t),\partial_t u)$ in $\hdot \times L^2$. Note that this limit is strong away from $x=0$. Let $v$ be the solution of \eqref{CP} such that $(v,\partial_t v)_{\restriction t=1}=(v_0,v_1)$.
Let
$$ a(t,x)=u(t,x)-v(t,x).$$
By finite speed of propagation
$$\supp a\subset \big\{(t,x)\in (T_-,1)\times \RR^N\;:\; |x|\leq 1-t\big\}.$$
Recall also that the following limits exist:
\begin{align}
\label{limE0}
E_0&=\lim_{t\to 1} E(a(t),\partial_t a(t))=E(u_0,u_1)-E(v_0,v_1)\\ 
\label{limd0}
d_0&=\lim_{t\to 1} \int_{\RR^N} \nabla a(t)\partial_ta(t)=\int_{\RR^N} \nabla u_0u_1-\int_{\RR^N} \nabla v_0v_1.
\end{align}

\subsection{Compactness of a minimal element}
\label{SS:compactness}
We define the set of \emph{large profiles} $\AAA\subset \hdot\times L^2$ as follows: $(U_0,U_1)$ is in $\AAA$ if and only if the following conditions are both satisfied
\begin{enumerate}
 \item there exist sequences $\{t_n\}_n$, $\{x_n\}_n$, $\{\lambda_n\}_n$,  with $t_n\in (0,1)$, $t_n\to 1$, $x_n\in \RR^N$, $\lambda_n\in (0,+\infty)$ such that
$$ \left(\lambda_n^{\frac{N}{2}-1}a(t_n,\lambda_n x+x_n),\lambda_n^{\frac{N}{2}}\partial_t a(t_n,\lambda_n x+x_n)\right)\xrightharpoonup[n\to \infty]{} (U_0,U_1)$$
weakly in $\hdot\times L^2$.
\item the solution $U$ of \eqref{CP} with initial condition $(U_0,U_1)$ does not scatter in either time direction, that is
$$ \|U\|_{L^{\frac{2(N+1)}{N-2}} (0,T_+)}=\|U\|_{L^{\frac{2(N+1)}{N-2}} (T_-,0)}=\infty.$$
\end{enumerate}

Let us prove:
\begin{prop}
 \label{P:compactness}
Let $u$ be as in Theorem \ref{T:classification}.
There exists $(V_0,V_1)\in \AAA$ which is minimal for the energy, that is
$$ \forall (U_0,U_1)\in \AAA,\quad E(V_0,V_1)\leq E(U_0,U_1).$$
Moreover, the solution $V$ of \eqref{CP} with initial condition $(V_0,V_1)$ is compact up to modulation.
\end{prop}
\begin{proof}
 \EMPH{Step 1} Let us show that $\AAA$ is not empty. Indeed, we will show that for any sequence $\{t_n\}_n\in (0,1)^{\NN}$ such that $t_n\to 1$, there exists a subsequence of $\{t_n\}$ and sequences $\{\lambda_n\}$, $\{x_n\}$ such that 
\begin{equation*}
\left(\lambda_n^{\frac{N}{2}-1}a(t_n,\lambda_n x+x_n),\lambda_n^{\frac{N}{2}}\partial_t a(t_n,\lambda_n x+x_n)\right)\xrightharpoonup[n\to \infty]{} (U_0,U_1)\in \AAA.
\end{equation*} 
Extracting subsequences if necessary, we may assume that the sequence $(a(t_n),\partial_t a(t_n))$ has a profile decomposition  $\lf\{U_{\lin}^j\rg\}_j$, $\lf\{\lambda_{j,n};x_{j,n};t_{j,n}\rg\}_{j,n}$. Consider the nonlinear profiles $U^j$ associated to this profile decomposition. We will show that exactly one of these nonlinear profiles does not scatter in any of the time directions, and that all others scatter in both time directions.

We can write the profile decomposition
\begin{equation*}
\left\{\begin{aligned}
 u(t_n,x)&=v(t_n,x)+\sum_{j=1}^J \frac{1}{\lambda_{j,n}^{\frac{N-2}{2}}}U_{\lin}^j\left(\frac{-t_{j,n}}{\lambda_{j,n}},\frac{x-x_{j,n}}{\lambda_{j,n}}\right)+w_{0,n}^J(x),\\
u(t_n,x)&=\partial_{t_n}v(t_n,x)+\sum_{j=1}^J \frac{1}{\lambda_{j,n}^{\frac{N}{2}}}\partial_t U_{\lin}^j\left(\frac{-t_{j,n}}{\lambda_{j,n}},\frac{x-x_{j,n}}{\lambda_{j,n}}\right)+w_{1,n}^J(x),
\end{aligned}\right.
\end{equation*}
and consider it as a profile decomposition for the sequence $(u(t_n),\partial_t u(t_n))$, where $(v(t_n),\partial_t v(t_n))$ is (up to an error which is $o(1)$ in $\hdot\times L^2$) interpreted as a profile $U^0_{\lin}$ with initial data $(v_0,v_1)$ and parameters 
$\lambda_{0,n}=1$, $t_{0,n}=0$, $x_{0,n}=0$. Note that as $\lambda_{j,n}\to 0$ for all $j\geq 1$, the sequence of parameters $\big\{\lambda_{0,n},t_{0,n},x_{0,n}\big\}_n$ is pseudo-orthogonal to all sequences $\big\{\lambda_{j,n},t_{j,n},x_{j,n}\big\}_n$, $j\geq 1$ in the sense given by \eqref{ortho_param}.

By Proposition \ref{P:lin_NL}, if all nonlinear profiles scatter forward in time, then $u$ must scatter forward in time, a contradiction. Fix $n$ and let
$$ T_n=\min_{j\geq 1} (\lambda_{j,n}T_+(U^j)+t_{j,n}),$$
where the minimum is taken over all $j$ such that $T_+(U^j)$ is finite.
Consider the quantity
$$ F_n(t)=\max_{j\geq 1}\int_0^{t} \int_{\RR^N} \left|U^j\left(\frac{t-t_{j,n}}{\lambda_{j,n}},\frac{x-x_{j,n}}{\lambda_{j,n}}\right)\right|^{\frac{2(N+1)}{N-2}}\frac{dx\,dt}{\lambda_{j,n}^{N+1}},\quad t\in [0,T_n).$$
The fact that at least one of the profiles does not scatter forward in time shows that $F_n(t)\to +\infty$ as $t\to T_n$. Thus there exists a time $\tau_n\in (0,T_n)$ such that  
\begin{equation}
\label{def_taun}
F_n(\tau_n)=C_{\|\nabla W\|^2_{L^2}-\eta_0},
\end{equation} 
where the constant $C_{\|\nabla W\|_{L^2}^2-\eta_0}$ is given by Corollary \ref{C:below_threshold}. By \eqref{def_taun} and Proposition \ref{P:lin_NL}, $t_n+\tau_n<1$ for large $n$. 
Reordering the profiles, assume that the max in the definition of $F_n(\tau_n)$ is attained for $j=1$. 
By the definition of $C_{\|\nabla W\|^2_{L^2}-\eta_0}$, there exists $s_n\in [0,\tau_n]$ such that
$$ \left\|\nabla U^1\lf(\frac{s_n-t_{1,n}}{\lambda_{1,n}}\rg)\rg\|_{L^2}^2+\frac{N-2}{2}\left\|\partial_t U^1\lf(\frac{s_n-t_{1,n}}{\lambda_{1,n}}\rg)\rg\|_{L^2}^2\geq \lf\|\nabla W\rg\|_{L^2}^2-2\eta_0.$$
By Pythagorean expansion and the bound \eqref{bound_nabla}, all the nonlinear profiles $U^j$, $j\geq 2$, satisfy, for large $n$
$$ \left\|\nabla U^j\lf(\frac{s_n-t_{j,n}}{\lambda_{j,n}}\rg)\rg\|_{L^2}^2+\frac{N-2}{2}\left\|\partial_t U^j\lf(\frac{s_n-t_{j,n}}{\lambda_{j,n}}\rg)\rg\|_{L^2}^2\leq 3\eta_0.$$
Chosing $\eta_0$ small, we get by the small data theory that for $j\geq 2$, $U^j$ scatters in both time directions and satisfies 
$$ \forall t\in \RR,\quad \left\|\nabla U^j\lf(t\rg)\rg\|_{L^2}^2+\left\|\partial_t U^j\lf(t\rg)\rg\|_{L^2}^2\leq C_N\eta_0,$$
for some constant $C_N>0$ depending only on $N$.
We next show that $U^1$ does not scatter either forward or backward in time. Indeed if $U^1$ scatters forward in time, then by Proposition \ref{P:lin_NL}, $u$ scatters forward in time, a contradiction. On the other hand, if $U^1$ scatters backward in time, we can use Proposition \ref{P:lin_NL} again and the orthogonality of the parameters to show that
$$ \int_{0}^{t_n} \int |u|^{\frac{2(N+1)}{N-2}} dx dt=\sum_{j=1}^J \int^{-t_{j,n}/\lambda_{j,n}}_{-(t_{j,n}+t_n)/\lambda_{j,n}} \int \lf|U^{j}\rg|^{\frac{2(N+1)}{N-2}}dx\,dt+\int_0^{t_n} \int\left|w_n^J\right|^{\frac{2(N+1)}{N-2}}dx\,dt+o(1)$$
as $n\to \infty$, and thus $\int_0^1\int_{\RR^N} |u|^{\frac{2(N+1)}{N-2}}$ is finite, a contradiction with the fact that the maximal time of existence of $u$ is $1$. This concludes the proof that $U^1$ does not scatter in any time direction. As a consequence, $-t_{1,n}/\lambda_{1,n}$ is bounded and we can assume (time translating the profile $U^1$ and passing to a subsequence if necessary):
$$t_{1,n}=0.$$
Thus the nonlinear profile $U^1$ is exactly the solution of \eqref{CP} with initial conditions $(U^1_0,U^1_1)$ and it does not scatter in either time direction.
By the definition of $U^1$, 
\begin{equation*}
\left(\lambda_{1,n}^{\frac{N}{2}-1}a(t_{n},\lambda_{1,n} x+x_{1,n}),\lambda_{1,n}^{\frac{N}{2}}\partial_t a(t_{n},\lambda_{1,n}x+x_{1,n})\right)\xrightharpoonup[n\to \infty]{} (U_0^1,U_1^1)
\end{equation*}
weakly in $\hdot\times L^2$, 
which shows that  $(U_0^1,U_1^1)\in \AAA$, concluding Step 1.


\EMPH{Step 2} In this step we show that there exists $(V_0,V_1)\in \AAA$ with minimal energy. We first note that by Claim \ref{C:variational}, the energy of any element of $\AAA$ is non-negative, so that
$$E_{\min}=\inf\big\{ E(U_0,U_1),\; (U_0,U_1)\in \AAA\big\}$$
is a non-negative number. 

Note that any element of $\AAA$ is the only non-scattering profile of a profile decomposition as in Step 1. This shows 
by the Proposition \ref{P:lin_NL} and Pythagorean expansion that the bound \eqref{bound_nabla} extends to $\AAA$. More precisely
\begin{equation}
\label{bound_nablaA}
(U_0,U_1)\in \AAA\Longrightarrow \sup_{t\in I_{\max}(U)} \|\nabla U(t)\|^2_{L^2}+\frac{N-2}{2}\|\partial_t U(t)\|^2_{L^2}\leq \int |\nabla W|^2+\eta_0, 
\end{equation} 
where $U$ is the solution of \eqref{CP} with initial data $(U_0,U_1)$.

Consider a sequence $\big\{(U_{0,n},U_{1,n})\big\}_n$ of elements of $\AAA$ such that
$$ \lim_{n\to \infty}E\left(U_{0,n},U_{1,n}\right)=E_{\min}.$$
After extracting subsequences, one can consider a profile decomposition:
\begin{align}
\label{decompo_1}
U_{0,n}(x)&=\sum_{j=1}^J\frac{1}{\lambda_{j,n}^{\frac{N}{2}-1}}V_{\lin}^j\left(\frac{-t_{j,n}}{\lambda_{j,n}},\frac{x-x_{j,n}}{\lambda_{j,n}}\right)+z_{0n}^J(x)\\
\label{decompo_2}
U_{1,n}(x)&=\sum_{j=1}^J\frac{1}{\lambda_{j,n}^{\frac{N}{2}}}\left(\partial_t V_{\lin}^j\right)\left(\frac{-t_{j,n}}{\lambda_{j,n}},\frac{x-x_{j,n}}{\lambda_{j,n}}\right)+z_{1n}^J(x).
\end{align}
For all $j$ we denote by $V^j$ the nonlinear profile associated to $V^j_{\lin},\left\{-\frac{t_{j,n}}{\lambda_{j,n}}\right\}_n$.
By the definition of $\AAA$, the solution $U_n$ of \eqref{CP} with initial data $(U_{0,n},U_{1,n})$ does not scatter in either time direction and satisfies the bound \eqref{bound_nablaA}. A similar argument to Step 1 shows that there exists only one profile, say $V^1$, which does not scatter in either time direction, that we can assume $t_{1,n}=0$ for all $n$, and that all other profiles $V^j$, $j\geq 2$, scatter in both time directions. 

To simplify notations, denote
$$ V=V^1,\quad V_0=V^1_{\lin}(0),\quad V_1=\partial_t V^1_{\lin}(0).$$
In particular
\begin{equation}
\label{WKCV1}\left(\lambda_{1,n}^{\frac{N}{2}-1}U_{0,n}(\lambda_{1,n}x+x_{1,n}),\lambda_{1,n}^{\frac{N}{2}} U_{1,n}(\lambda_{1,n}x+x_{1,n})\right)\xrightharpoonup[n\to \infty]{} (V_0,V_1).
 \end{equation} 
For all $n$, as $(U_{0,n},U_{1,n})$ is in $\AAA$, there exists sequences $\lf\{\mu_{k,n}\rg\}_k$, $\lf\{y_{k,n}\rg\}_k$, $\lf\{\tau_{k,n}\rg\}_k$ such that 
$$\tau_{k,n}\in (0,1),\quad \lim_{k\to\infty}\tau_{k,n}=1$$ 
and 
\begin{equation}
\label{WKCV2}
\left(\mu_{k,n}^{\frac{N}{2}-1}a(\tau_{k,n},\mu_{k,n}x+y_{k,n}),\mu_{k,n}^{\frac{N}{2}}\partial_t a(\tau_{k,n},\mu_{k,n}x+y_{k,n})\right)\xrightharpoonup[k\to \infty]{} \lf(U_{0,n},U_{1,n}\rg)
\end{equation}
weakly in $\hdot\times L^2$. In view of \eqref{WKCV1} and \eqref{WKCV2}, we can obtain, via a diagonal extraction argument (see Step 1 in the proof of Proposition 7.1 in \cite{DuKeMe09P}), sequences $\lf\{\mu_{n}\rg\}_n$, $\lf\{y_{n}\rg\}_n$, $\lf\{\tau_{n}\rg\}_n$ such that 
$$\tau_{n}\in (0,1),\quad \lim_{n\to\infty}\tau_{n}=1$$ 
and
\begin{equation*}
\left(\mu_{n}^{\frac{N}{2}-1}a(\tau_{n},\mu_{n}x+y_{n}),\mu_{n}^{\frac{N}{2}}\partial_t a(\tau_{n},\mu_{n}x+y_{n})\right)\xrightharpoonup[k\to \infty]{} (V_0,V_1).
\end{equation*}
Thus $(V_0^1,V_1^1)\in \AAA$. By the decomposition \eqref{decompo_1}, \eqref{decompo_2} and the Pythagorean expansion properties of the profiles,
$$ E(U_0^n,U_1^n)=E(V_0,V_1)+\sum_{j=2}^J E(V^j(0),\partial_t V^j(0))+E(w_{0,n}^J(0),w_{1,n}^J(0))+o(1)\text{ as }n\to\infty.$$
Using that by Claim \ref{C:variational} all the profiles have non-negative energy, and that $E(U_0^n,U_1^n)$ tends to $E_{\min}$ as $n$ goes to $\infty$, we obtain
$E_{\min}\geq E(V_0,V_1),$
and thus (as $(V_0,V_1)\in \AAA$),
$$ E(V_0,V_1)=E_{\min}.$$

\EMPH{Step 3} We next show that the solution $V$ of \eqref{CP} with initial data $(V_0,V_1)$ is compact up to modulation. It is sufficient to show that for all sequences $\{t_n\}_n$ in the domain of existence of $V$, there exist a subsequence of $\{t_n\}_n$ and sequences $\{\lambda_n\}_n$, $\{x_n\}_n$ such that
$$ \left(\lambda_n^{\frac{N}{2}-1}V(t_n,\lambda_n x+x_n),\lambda_n^{\frac{N}{2}}\partial_t V(t_n,\lambda_n x+x_n)\right)$$
converges strongly in $\hdot\times L^2$ as $n\to \infty$.

Extracting subsequences, we may assume that the sequence $\big\{(V(t_n),\partial_t V(t_n))\big\}_n$ has a profile decomposition $\lf\{U_{\lin}^j\rg\}_j$, $\lf\{\lambda_{j,n};x_{j,n};t_{j,n}\rg\}_{j,n}$. As before, \eqref{bound_nablaA} and the fact that $V$ does not scatter implies that there is only one nonlinear profile (say $U^1$) that does not scatter, and that we can choose $t_{1,n}=0$. By a diagonal extraction argument and Proposition \ref{P:lin_NL}, we have
$$ (U_0^1,U_1^1)\in \AAA.$$
By the Pythagorean expansion for the energy
\begin{multline*}
E_{\min}=E(V(t_n),\partial_tV(t_n))=E(U_0^1,U_1^1)+\sum_{j=2}^J E\left(U_0^j(-t_{j,n}/\lambda_{j,n}),U_1^j(-t_{j,n}/\lambda_{j,n})\right)\\
+E(w_{0,n}^J,w_{1,n}^J)+o(1)\text{ as }n\to \infty.
 \end{multline*}
Using that $E(U^1_0,U_1^1)\geq E_{\min}$ and that all the energies in the expansion are non-negative, we get by Claim \ref{C:variational} that $U^j=0$ for all $j\geq 2$ and 
$$ \lim_{n\to \infty}\|w_{0,n}^J\|_{\hdot}+\|w_{1,n}^J\|_{L^2}=0.$$
The proof is complete.
\end{proof}
\begin{corol}
\label{C:limit}
Let $u$ be as in Theorem \ref{T:classification}.
 Let $t_n\to 1$ be such that there exists $(V_0,V_1)\in \AAA$ with $E(V_0,V_1)=E_{\min}$ and $\lambda_n'>0$, $x_n'\in \RR^N$ so that
\begin{equation}
\label{deftn}
\left({\lambda_n'}^{\frac{N}{2}-1}a(t_n,\lambda_n' x+x_n'),{\lambda_n'}^{\frac{N}{2}}\partial_t a(t_n,\lambda_n' x+x_n')\right)\xrightharpoonup[n\to \infty]{} (V_0,V_1)\in \AAA.
\end{equation} 
Then rotating the space variable around the origin, and replacing $u$ by $-u$ if necessary, there exist $\lambda_n$, $x_n$ such that
\begin{equation}
\label{V=W}
 \left({\lambda_n}^{\frac{N}{2}-1}a(t_n,\lambda_n x+x_n),{\lambda_n}^{\frac{N}{2}}\partial_t a(t_n,\lambda_n x+x_n)\right)\xrightharpoonup[n\to \infty]{}  \left(W_{\ell}(0,x),\partial_tW_{\ell}(0,x)\right),
\end{equation} 
for some $\ell \in \RR$ with 
\begin{equation}
\label{estim_ell}
\ell^4\|\nabla W\|_{L^2}^2\leq 16\eta_0. 
\end{equation} 
Furthermore for large $n$,
\begin{equation}
\label{estim_eps}
\left\|\lambda_n^{\frac{N}{2}-1}a(t_n,\lambda_n \cdot+x_n)-W_{\ell}(0,\cdot)\right\|^2_{\hdot}+\frac{N-2}{2}\left\|\lambda_n^{\frac{N}{2}}\partial_t a(t_n,\lambda_n \cdot+x_n)-\partial_t W_{\ell}(0,\cdot)\right\|^2_{L^2}\leq 2\eta_0, 
\end{equation} 
 and
\begin{equation}
\label{estim_E}
\left|E_0-E(W,0)\right|+|d_0|\leq C\eta_0^{1/4},
\end{equation} 
where $E_0$ and $d_0$ are the limits of the energy and the momentum of $a$ (see \eqref{limE0}, \eqref{limd0}).
\end{corol}
\begin{proof}
By Proposition \ref{P:compactness},
the solution $V$ with initial condition $(V_0,V_1)$ is compact up to mo\-du\-la\-tion. By Theorem \ref{T:compact}, after a rotation of $\RR^N$ (and possibly changing $u$ into $-u$), there exists $x_0\in \RR^N$ and $\mu_0>0$ such that
$$ (V_0,V_1)=\left(\frac{1}{\mu_0^{\frac{N}{2}-1}}W_{\ell}\left(0,\frac{\cdot-x_0}{\mu_0}\right),\frac{1}{\mu_0^{\frac{N}{2}}}\partial_tW_{\ell}\left(0,\frac{\cdot-x_0}{\mu_0}\right)\right).$$
Taking $\lambda_n=\mu_0\lambda_n'$ and $x_n=x_n'+\lambda_n x_0$ we get \eqref{V=W}.

By \eqref{V=W},
$$ \|\nabla a (t_n)\|_{L^2}^2=\left\|\nabla W_{\ell}(0)-\lambda_n^{\frac{N}{2}}\nabla a(t_n,\lambda_nx +x_n)\right\|_{L^2}^2+\|\nabla W_{\ell}(0)\|_{L^2}^2+o_n(1).$$
Together with the analoguous statement on the time derivative of $a$ and with assumption \eqref{bound_nabla}, we get that for large $n$,
\begin{multline}
\label{dev_Well}
\|\nabla W_{\ell}(0)\|_{L^2}^2+\frac{N-2}{2}\|\partial_t W_{\ell}(0)\|_{L^2}^2+\left\|\nabla W_{\ell}(0)-\lambda_n^{\frac{N}{2}}\nabla a(t_n,\lambda_nx +x_n)\right\|_{L^2}^2\\
+\frac{N-2}{2}\left\|\partial_t W_{\ell}(0)-\lambda_n^{\frac{N}{2}}\partial_t a(t_n,\lambda_nx +x_n)\right\|_{L^2}^2
\leq \|\nabla W\|^2_{L^2}+2\eta_0. 
\end{multline}
By Claim \ref{C:value_W}, 
$$ \|\nabla W_{\ell}(0)\|^2_{L^2}+\frac{N-2}{2}\|\partial_tW_{\ell}(0)\|^2_{L^2}-\int |\nabla W|^2\geq \frac{\ell^4}{8}\int |\nabla W|^2,$$
and thus \eqref{dev_Well} implies $16\eta_0\geq \ell^4\int |\nabla W|^2$, and \eqref{estim_ell}, \eqref{estim_eps} follow. The estimate \eqref{estim_E} follows from \eqref{estim_ell}, \eqref{estim_eps}, and the fact that for small $\ell$,
$\left|E(W,0)-E(W_{\ell}(0),\partial_t W_{\ell}(0))\right|\leq C\ell^2$.
\end{proof}
\subsection{A few estimates}
\label{SS:estimates}
Until the end of the proof, we fix a sequence $t_n$ as in Corollary \ref{C:limit}, and we denote by
\begin{align}
\label{defepsn0}
\tilde{\eps}_{0n}(x)&=a(t_n,x)-\frac{1}{\lambda_n^{\frac{N}{2}-1}}W_{\ell}\left(0,\frac{x-x_n}{\lambda_n}\right)\\
\label{defepsn1}
\tilde{\eps}_{1n}(x)&=\partial_t a(t_n,x)-\frac{1}{\lambda_n^{\frac{N}{2}}}\partial_t W_{\ell}\left(0,\frac{x-x_n}{\lambda_n}\right).
\end{align} 
We have by \eqref{estim_eps}
\begin{equation}
\label{estim_eps2}
\limsup_{n\to \infty} \|\nabla \tilde{\eps}_{0n}\|_{L^2}^2+\frac{N-2}{2}\|\tilde{\eps}_{1n}\|_{L^2}^2\leq 2\eta_0.
 \end{equation} 
\begin{lemma}
\label{L:estimates}
The parameters $x_n$ and $\lambda_n$ satisfy:
\begin{gather}
\label{lambdan_0}
 \lim_{n\to +\infty} \frac{\lambda_n}{1-t_n}=0,\\
\label{xn_small}
\limsup_n \frac{|x_n|}{1-t_n}\leq C\eta_0^{1/4}.
\end{gather}
\end{lemma}
\begin{proof}
Using that $|x|\leq 1-t$ on the support of $a$, we get that $|x_n|\leq C(1-t_n)$ and $|\lambda_n|\leq C(1-t_n)$ (see \cite[p.154-155]{BaGe99}). 

\EMPH{Proof of \eqref{lambdan_0}} We argue by contradiction. Assume (after extraction) that for large $n$,
\begin{equation}
 \label{absurd_lambdan}
\frac{\lambda_n}{1-t_n}\geq c_0>0.
\end{equation} 
Notice that
$$ \lambda_n^{\frac{N}{2}-1}a(t_n,\lambda_n x+x_n)\neq 0\Longrightarrow |x|\leq \frac{1-t_n}{\lambda_n}+\frac{|x_n|}{\lambda_n}\Longrightarrow |x|\leq \frac{1}{c_0}+\frac{C}{c_0},$$
As $W_{\ell}(0)$ is the weak limit of the preceding function, we obtain that $|x|\leq C_0$ on the support of $W_{\ell}(0)$, a contradiction. 

\EMPH{Proof of \eqref{xn_small}} Denote by $e(u)$ the density of energy defined by \eqref{defeu}.
Using that $u$ and $v$ are solutions of \eqref{CP} and that $\supp a\subset \{|x|\leq 1-t\}$, we obtain (see \eqref{identity4} in Claim \ref{C:identities})
\begin{equation}
\label{interm}
    \frac{d}{dt} \int_{\RR^N} x(e(u)-e(v))dx=-\int (\nabla u\partial_t u-\nabla v\partial_t v)=-d_0.
   \end{equation} 
Furthermore,
$$ \left|\int_{\RR^N} x(e(u)-e(v))dx\right|=\left|\int_{|x|\leq (1-t)} x(e(u)-e(v))dx\right|\leq C(1-t),$$
and thus
$$ \lim_{t\to 1}\int_{\RR^N} x(e(u)-e(v))dx=0.$$
Integrating \eqref{interm} between $t_n$ and $1$, we get
\begin{equation}
\label{needed_after}
\int_{\RR^N} x\,(e(u)-e(v))(t_n)dx=d_0(1-t_n),
\end{equation} 
and thus by \eqref{estim_E},
\begin{equation}
\label{bound_uv}
\left|\int_{\RR^N} x(e(u)-e(v))(t_n)dx\right|\leq C\eta_0^{1/4}(1-t_n). 
\end{equation} 
Recall that $\lambda_n^{\frac{N}{2}-1}a(t_n,\lambda_n \cdot+x_n)$ converges weakly to $W_{\ell}(0)$ and that $u(t_n)$ converges weakly to $v(1)$ in $\hdot$ as $n\to\infty$. Thus
\begin{multline*}
\|\nabla W_{\ell}(0)\|_{L^2}^2\leq \limsup_{n\to \infty} \|\nabla a(t_n)\|_{L^2}^2\\
=\limsup_{n\to\infty}\lf(\|\nabla u(t_n)\|^2_{L^2}-2\langle \nabla u(t_n),\nabla v(t_n)\rangle_{L^2}+\|\nabla v(t_n)\|^2_{L^2}\rg)
\\=-\|\nabla v(1)\|_{L^2}^2+\limsup_{n\to\infty}\|\nabla u(t_n)\|^2_{L^2}.
\end{multline*}
Using this together with the analoguous statements on the time derivatives, we see that \eqref{bound_nabla} implies that 
$$ \|\nabla W_{\ell}(0)\|_{L^2}^2+\frac{N-2}{2}\|\partial_t W_{\ell}(0)\|_{L^2}^2+\|\nabla v(1)\|^2+\frac{N-2}{2}\|\partial_t v(1)\|^2\leq \|\nabla W\|^2_{L^2}+\eta_0,$$
and thus for large $n$, using the continuity of $v$ and that, by Claim \ref{C:value_W},
$\|\nabla W\|_{L^2}^2\leq \|\nabla W_{\ell}(0)\|_{L^2}^2+\frac{N-2}{2}\|\partial_t W_{\ell}(0)\|_{L^2}^2$,
\begin{equation}
\label{bound_v} 
\|v(t_n)\|_{\hdot}^2+\frac{N-2}{2}\|\partial_t v(t_n)\|_{L^2}^2\leq 2\eta_0.
\end{equation} 
Thus \eqref{bound_uv} implies
\begin{equation}
\label{forxn1}
\left|\int_{\RR^N} x\,e(a)(t_n)dx\right|\leq C\eta_0^{1/4}(1-t_n).
\end{equation} 
By \eqref{defepsn0}, \eqref{defepsn1} and \eqref{estim_eps2} there exists $A>0$ such that for large $n$,
$$\int_{\frac{|x-x_n|}{\lambda_n}\geq A} |\nabla a|^2+(\partial_t a)^2+|a|^{\frac{2N}{N-2}}\leq C\eta_0\leq C\eta_0^{1/4}.$$ 
As a consequence, for large $n$ (using that on the support of $a$, $|x|\leq 1-t_n$),
\begin{equation}
\label{forxn2}
\left|\int_{|x-x_n|\geq A\lambda_n} xe(a)(t_n)\right|\leq C\eta_0^{1/4}(1-t_n).
\end{equation} 
On the other hand,
\begin{equation}
\label{forxn3}
\int_{|x-x_n|\leq A\lambda_n} xe(a)(t_n)= \int_{|x-x_n|\leq A\lambda_n} (x-x_n)e(a)(t_n)+x_n\int_{|x-x_n|\leq A\lambda_n} e(a)(t_n).
\end{equation} 
By \eqref{lambdan_0},
\begin{equation}
\label{forxn4}
\lim_{n\to \infty}\frac{1}{1-t_n}\left|\int_{|x-x_n|\leq A\lambda_n} (x-x_n)e(a)(t_n)\right|=0. 
\end{equation} 
Furthermore, using that $\eta_0$ is small, we get by \eqref{estim_eps2} that if $A$ is chosen large,
\begin{equation}
\label{forxn5}
\liminf_{n\to\infty}\int_{|x-x_n|\leq A\lambda_n} e(a)(t_n)\geq \frac{1}{2}E(W_{\ell}(0),\partial_tW_{\ell}(0)). 
\end{equation}
Combining \eqref{forxn1},\ldots, \eqref{forxn5} we get the desired estimate \eqref{xn_small}.

\end{proof}
\subsection{Strong convergence to the solitary wave for a sequence of times}
\label{SS:strongCV}
Until the end of Section \ref{S:universality}, we assume $N\in \{3,5\}$. 
\begin{prop}
\label{P:strong_CV}
 Let $\{t_n\}$ be any sequence as in Corollary \ref{C:limit}. Then there exists $\ell\in (-1,1)$ such that (rotating again the space variable around the origin and replacing $u$ by $-u$ if necessary), 
$$\lim_{n\to \infty} \left(\lambda_n^{\frac{N}{2}-1}a(t_n,\lambda_nx +x_n),\lambda_n^{\frac{N}{2}}\partial_ta(t_n,\lambda_nx+x_n)\right)=\left(W_{\ell}(0),\partial_t W_{\ell}(0)\right),$$
\emph{strongly} in $\hdot\times L^2$.
\end{prop}
\begin{proof}
\EMPH{Step 1. Rescaling and application of the linear lemma} We first rescale the solutions. Let
$$ g_n(\tau,y)=(1-t_n)^{\frac{N}{2}-1} u(t_n+(1-t_n)\tau,(1-t_n)y),\quad (g_{0n},g_{1n})=\left(g_n(0),\partial_{\tau}g_n(0)\right),$$
and 
$$ h_n(\tau,y)=(1-t_n)^{\frac{N}{2}-1} v(t_n+(1-t_n)\tau,(1-t_n)y),\quad (h_{0n},h_{1n})=\left(h_n(0),\partial_{\tau}h_n(0)\right).$$
Then for all $n$, $g_n$ is a solution to \eqref{CP} with maximal time of existence $1$, and $h_n$ is a globally defined solution of \eqref{CP}. By \eqref{defepsn0}, \eqref{defepsn1}, \eqref{estim_eps2}, 
\begin{align}
\label{def_h0n}
 g_{0n}(y)&=h_{0n}(y)+\frac{1}{\mu_n^{\frac{N}{2}-1}}W_{\ell}\left(0,\frac{y-y_n}{\mu_n}\right)+ \eps_{0n}(y)\\
\label{def_h1n}
 g_{1n}(y)&=h_{1n}(y)+\frac{1}{\mu_n^{\frac{N}{2}}}\partial_t W_{\ell}\left(0,\frac{y-y_n}{\mu_n}\right)+\eps_{1n}(y),
\end{align}
where
$$ \mu_n=\frac{\lambda_n}{1-t_n}\to 0,\quad y_n=\frac{x_n}{1-t_n},\quad |y_n|\leq C\eta_0^{1/4}$$
and
\begin{equation*}
 \eps_{0n}=\frac{1}{\mu_n^{\frac{N}{2}-1}}\tilde{\eps}_{0n}\left(\frac{y-y_n}{\mu_n}\right),\quad\eps_{1n}=\frac{1}{\mu_n^{\frac{N}{2}}}\tilde{\eps}_{0n}\left(\frac{y-y_n}{\mu_n}\right).
\end{equation*} 
We argue by contradiction. We must show that $\left(\tilde{\eps}_{0n},\tilde{\eps}_{1n}\right)$ tends to $0$ in $\hdot\times L^2$, i.e that $\left(\eps_{0n},\eps_{1n}\right)$ tends to $0$ in $\hdot\times L^2$. Assume (after extraction) that
$$ \lim_{n\to\infty} \|\eps_{0n}\|_{\hdot}^2+\|\eps_{1n}\|^2_{L^2}=\delta_1>0.$$
Using that $|x|\leq 1-t_n$ on the support of $a$, we obtain
\begin{equation}
\label{ext_nul}
 \lim_{n\to \infty}\int_{|y|\geq 1} |\nabla \eps_{0n}(y)|^2+(\eps_{1n}(y))^2=0.
\end{equation} 
We denote by $\eps_n^{\lin}$ (respectively $\eps_n$) the solution to the linear wave equation (respectively the nonlinear wave equation) with initial condition $\left(\eps_{0n},\eps_{1n}\right)$. Applying Proposition \ref{P:linear} to $\eps_n^{\lin}$, we get (in view of \eqref{ext_nul}) that for large $n$, the following holds for all $\tau>0$, or for all $\tau<0$:
\begin{equation}
\label{linear_exit}
\int_{|\tau|\leq |y-y_n|\leq 2+|\tau|} \left|\nabla \eps_n^{\lin}(\tau)\right|^2+\left(\partial_t \eps_n^{\lin}(\tau)\right)^2\geq \frac{\delta_1}{4}.
\end{equation} 
\EMPH{Step 2. Concentration of some energy outside the light-cone}
In step 3 we will show that if \eqref{linear_exit} holds for all $\tau>0$, then for large $n$,
\begin{equation}
 \label{NL_exit+}
\int_{\frac{3}{4}\leq |y-y_n|\leq 3}\lf|\nabla g_n\left(\frac 34\right)\rg|^2+\lf(\partial_t g_n\lf(\frac 34\rg)\rg)^2\geq \frac{\delta_1}{16},
\end{equation} 
and if \eqref{linear_exit} holds for all $\tau<0$, then for a small $r_0>0$ and for large $n$,
\begin{equation}
 \label{NL_exit-}
\int_{|\tau_n|\leq |y-y_n|\leq |\tau_n|+10}|\nabla g_n(\tau_n)|^2+(\partial_t g_n(\tau_n))^2\geq \frac{\delta_1}{16}, \text{ where }\tau_n=-\frac{r_0}{1-t_n}.
 \end{equation} 
In this step we show that \eqref{NL_exit+} or \eqref{NL_exit-} yield a contradiction. If \eqref{NL_exit+} holds, then for large $n$,
$$
\int_{\frac{3}{4}\leq \frac{|x-x_n|}{1-t_n}\leq 3}\lf|\nabla u\lf(\frac{3}{4}+\frac{t_n}{4}\rg)\rg|^2+\lf(\partial_t u\lf(\frac{3}{4}+\frac{t_n}{4}\rg)\rg)^2\geq \frac{\delta_1}{16}.$$
Let $t'_n=\frac 34+\frac{t_n}{4}\to 1$ as $n\to \infty$. Then the preceding inequality implies
\begin{equation}
\label{more_exit}
\int_{2(1-t'_n)\leq |x|\leq 13(1-t_n')}|\nabla u(t_n')|^2+(\partial_t u(t_n'))^2\geq \frac{\delta_1}{16}.
\end{equation} 
Indeed, by \eqref{xn_small}, and using that $1-t'_n=\frac{1-t_n}{4}$, we get for large $n$,
\begin{equation*}
 \frac{3}{4}\leq \frac{|x-x_n|}{1-t_n}\leq 3\Longrightarrow 3\leq \frac{|x-x_n|}{1-t_n'}\leq 12\Longrightarrow 3-C\eta_0^{1/4}\leq \frac{|x|}{1-t_n'}\leq 12+C\eta_0^{1/4},
\end{equation*}
and \eqref{more_exit} follows if $\eta_0$ is small.

If $|x|\geq 1-t_n'$, then $v(t_n',x)=u(t_n',x)$ and we obtain by \eqref{more_exit} that for large $n$,
\begin{equation*}
\int_{2(1-t'_n)\leq |x|\leq 13(1-t_n')}|\nabla v(t_n')|^2+(\partial_t v(t_n'))^2\geq \frac{\delta_1}{16},
\end{equation*} 
a contradiction with the fact that $(v,\partial_t v)\in C^0(\RR,\hdot\times L^2)$ (and thus the preceding integral tends to $0$ as $n$ goes to $\infty$).

In the case where \eqref{NL_exit-} holds, we obtain that for large $n$,
\begin{equation*}
 \int_{\frac{r_0}{1-t_n}\leq \frac{|x-x_n|}{1-t_n}\leq \frac{r_0}{1-t_n}+10}\left|\nabla u(t_n-r_0)\right|^2+(\partial_tu(t_n-r_0))^2dx\geq \frac{\delta_1}{16},
\end{equation*}
which yields a contradiction in a similar manner.

\EMPH{Step 3. Nonlinear approximation}
It remains to prove \eqref{NL_exit+} and \eqref{NL_exit-}. We will focus on the proof of \eqref{NL_exit+}. The proof of \eqref{NL_exit-} is similar and we leave the details to the reader.

Let $A$ be a large positive number to be specified later. Recall that $\eps_n$ is the solution of \eqref{CP} with initial condition $ (\eps_{0n},\eps_{1n})$. In view of \eqref{def_h0n}, \eqref{def_h1n} we get
\begin{align}
\label{gn_Amun} 
g_{n}(A\mu_n,y)&=h_n(A\mu_n,y)+\frac{1}{\mu_n^{\frac{N}{2}-1}}W_{\ell}\left(A,\frac{y-y_n}{\mu_n}\right)+\eps_{n}(A\mu_n,y)+o_n(1)\text{ in }\hdot\\
\label{dtgn_Amun} 
 \partial_t g_{n}(A\mu_n,y)&=\partial_th_n(A\mu_n,y)+\frac{1}{\mu_n^{\frac{N}{2}}}\partial_t W_{\ell}\left(A,\frac{y-y_n}{\mu_n}\right)+\partial_t\eps_{n}(A\mu_n,y)+o_n(1)\text{ in }L^2.
\end{align}
 To show this, write a profile decomposition $\lf\{U_{\lin}^j\rg\}_{j\geq 3}$, $\lf\{\lambda_{j,n},x_{j,n},t_{j,n}\rg\}_{j,n}$ for the sequence $(\eps_{0n},\eps_{1n})$ and notice that the equality 
\begin{align*}
 g_{0n}(y)&=h_{0n}(y)+\frac{1}{\mu_n^{\frac{N}{2}-1}}W_{\ell}\left(0,\frac{y-y_n}{\mu_n}\right)+ \sum_{j=3}^J \frac{1}{\lambda_{j,n}^{\frac{N}{2}-1}}U^j_{\lin}\lf(\frac{-t_{j,n}}{\lambda_{j,n}},\frac{x-{x_{j,n}}}{\lambda_{j,n}}\right)+w_{0n}^J\\
g_{1n}(y)&=h_{1n}(y)+\frac{1}{\mu_n^{\frac{N}{2}}}\partial_t W_{\ell}\left(0,\frac{y-y_n}{\mu_n}\right)+
\sum_{j=3}^J \frac{1}{\lambda_{j,n}^{\frac{N}{2}}}\partial_t U^j_{\lin}\lf(\frac{-t_{j,n}}{\lambda_{j,n}},\frac{x-{x_{j,n}}}{\lambda_{j,n}}\right)+w_{1n}^J
\end{align*}
provides a profile decomposition for the sequence $(g_{0n},g_{1n})$, where two additional profiles $U_{\lin}^1$ and $U_{\lin}^2$ are given by the solutions of the linear wave equation with initial conditions $(v_0,v_1)$ and $(W_{\ell}(0),\partial_t W_{\ell}(0))$ respectively, and $t^1_n=t^2_n=0$, $x^1_n=0$, $x^2_n=y_n$, $\lambda_{1,n}=1-t_n$, $\lambda_{2,n}=\mu_n$. Applying Proposition \ref{P:lin_NL} to both sequences $(\eps_{0n},\eps_{1n})$ and $(g_{0n},g_{1n})$ we get \eqref{gn_Amun}, \eqref{dtgn_Amun}. Note that it is also possible to show directly \eqref{gn_Amun}, \eqref{dtgn_Amun} from a long-time perturbation result, without relying on profile decomposition.

Let $\psi\in C_0^{\infty}(\RR^N)$ be a radial function such that $\psi(x)=1$ for $|x|\leq \frac{1}{3}$ and $\psi(x)=0$ for $|x|\geq \frac{2}{3}$. Write \eqref{gn_Amun}, \eqref{dtgn_Amun} as
\begin{align*}
g_{n}(A\mu_n,y)&=\left(1-\psi\left(\frac{y}{30}\right)\right)h_n(A\mu_n,y)+\psi\left(\frac{y-y_n}{A\mu_n}\right)\frac{1}{\mu_n^{\frac{N}{2}-1}}W_{\ell}\left(A,\frac{y-y_n}{\mu_n}\right)+\overline{\eps}_{0n}(y)\\
 \partial_t g_{n}(A\mu_n,y)&=\left(1-\psi\left(\frac{y}{30}\right)\right)\partial_th_n(A\mu_n,y)+\psi\left(\frac{y-y_n}{A\mu_n}\right)\frac{1}{\mu_n^{\frac{N}{2}}}\partial_t W_{\ell}\left(A,\frac{y-y_n}{\mu_n}\right)+\oveps_{1n}(y),
\end{align*}
 where as $n\to \infty$, in $\hdot\times L^2$,
\begin{align*}
\oveps_{0n}&=\psi\left(\frac{y}{30}\right)h_n(A\mu_n)+\left(1-\psi\left(\frac{y-y_n}{A\mu_n}\right)\rg)\frac{1}{\mu_n^{\frac{N}{2}-1}}W_{\ell}\left(A,\frac{y-y_n}{\mu_n}\right)+\eps_{n}(A\mu_n)+o(1),\\
\oveps_{1n}&=\psi\left(\frac{y}{30}\right)\partial_t h_n(A\mu_n)+\left(1-\psi\left(\frac{y-y_n}{A\mu_n}\right)\rg)\frac{1}{\mu_n^{\frac{N}{2}}}\partial_t W_{\ell}\left(A,\frac{y-y_n}{\mu_n}\right)+\partial_t \eps_{n}(A\mu_n)+o(1).
\end{align*}
Then as $n\to \infty$.
\begin{multline}
\label{Explain1}
\left\|\oveps_{0n}-\eps_n^{\lin}(A\mu_n)\rg\|_{\hdot}
 \lesssim \left\|\eps_n(A\mu_n)-\eps_n^{\lin}(A\mu_n)\rg\|_{\hdot}
\\+\sqrt{\int_{|x|\geq \frac{A}{3}} |\nabla W_{\ell}(A,x)|^2}
+\sqrt{\int_{|x|\leq 20(1-t_n)} |\nabla v(t_n+(1-t_n)A\mu_n,x)|^2}+o(1),
\end{multline}
and similarly
\begin{multline}
\label{Explain2}
\left\|\oveps_{1n}-\partial_t\eps_n^{\lin}(A\mu_n)\rg\|_{L^2}
 \lesssim \left\|\partial_t \eps_n(A\mu_n)-\partial_t \eps_n^{\lin}(A\mu_n)\rg\|_{L^2}+
\\+\sqrt{\int_{|x|\geq \frac{A}{3}} |\partial_t W_{\ell}(A,x)|^2}+\sqrt{\int_{|x|\leq 20(1-t_n)} |\partial_t v(t_n+(1-t_n)A\mu_n,x)|^2}+o(1),
\end{multline}
As $\ell\leq C\eta_0^{1/4}$, we can assume that $\ell$ is small, and thus, by the explicit expression of $W_{\ell}$, if $A$ is chosen large enough,
\begin{equation}
\label{Explain3}
\sqrt{\int_{|x|\geq \frac{A}{3}} |\nabla W_{\ell}(A,x)|^2}
+\sqrt{\int_{|x|\geq \frac{A}{3}} |\partial_t W_{\ell}(A,x)|^2}\leq \frac{\sqrt{\delta_1}}{10000}.
\end{equation} 
Furthermore by the small data theory (see \eqref{almost_linear}), if $n$ is large 
\begin{equation}
\label{Explain4}
\left(\left\|\eps_n(A\mu_n)-\eps_n^{\lin}(A\mu_n)\rg\|_{\hdot}^2+ \left\|\partial_t \eps_n(A\mu_n)-\partial_t \eps_n^{\lin}(A\mu_n)\rg\|_{\hdot}^2\rg)^{1/2}\leq \frac{\sqrt{\delta_1}}{10000}.
\end{equation}
For large $n$, combining \eqref{Explain1}, \eqref{Explain2}, \eqref{Explain3} and \eqref{Explain4}, we get
\begin{equation}
\label{oveps_close}
  \lf(\left\|\oveps_{0n}-\eps_n^{\lin}(A\mu_n)\rg\|_{\hdot}^2+\left\|\oveps_{1n}-\partial_t \eps_n^{\lin}(A\mu_n)\rg\|_{L^2}^2\rg)^{1/2}\leq \frac{\sqrt{\delta_1}}{1000}.
 \end{equation} 
Furthermore, by the definition of $\overline{\eps}_{0n}$ and $\overline{\eps}_{1n}$,
\begin{equation*}
 \frac{y}{10}\leq 1 \text{ and }\left|\frac{y-y_n}{\mu_n A}\right|\geq \frac 23\Longrightarrow g_{n}(A\mu_n)=\overline{\eps}_{0n}\text{ and } \partial_t g_{n}(A\mu_n)=\overline{\eps}_{1n}.
\end{equation*} 
Using again that $\eta_0$ is small, and that $|y_n|\leq C\eta_0^{1/4}$, we get 
\begin{equation}
\label{equality}
 \frac{2}{3}A\mu_n\leq |y-y_n|\leq 9 \Longrightarrow g_{n}(A\mu_n)=\overline{\eps}_{0n}\text{ and } \partial_t g_{n}(A\mu_n)=\overline{\eps}_{1n}.
\end{equation} 
Let $\oveps_n$ (respectively $\overline{\eps}_n^{\lin}$) be the solution to \eqref{CP} (respectively to the linear wave equation) with initial data $(\oveps_{0n},\oveps_{1n})$. By \eqref{oveps_close} and the conservation of the energy for the linear equation,
\begin{equation}
\label{explain1}
\lf( \left\|\oveps_n^{\lin}(\sigma)-\eps_n^{\lin}(\sigma+A\mu_n)\right\|_{\hdot}^2+\left\|\partial_t\oveps_n^{\lin}(\sigma)-\partial_t\eps_n^{\lin}(\sigma+A\mu_n)\right\|_{L^2}^2\rg)^{1/2}\leq \frac{\sqrt{\delta_1}}{1000}.
\end{equation} 
By the small data theory (see \eqref{almost_linear}), using that $\delta_1\leq \eta_0$, and that $\eta_0$ is small, we get
\begin{equation}
 \label{explain2}
\lf(
\left\|\oveps_n(\sigma)-\oveps_n^{\lin}(\sigma)\right\|_{\hdot}^2+\left\|\partial_t\oveps_n(\sigma)-\partial_t \oveps_n^{\lin}(\sigma)\right\|_{L^2}^2\rg)^{1/2}\leq \frac{\sqrt{\delta_1}}{1000}.
\end{equation} 
Combining \eqref{explain1} and \eqref{explain2} with \eqref{linear_exit} we obtain taking $\sigma=3/4-A\mu_n$ (and $\tau=\frac{3}{4}$ in \eqref{linear_exit}),
$$\int_{\frac{3}{4}\leq |y-y_n|\leq 3} \left|\nabla \oveps_n\left(\frac{3}{4}-A\mu_n\right)\right|^2+\left|\partial_t \oveps_n\left(\frac{3}{4}-A\mu_n\right)\right|^2\geq \frac{\delta_1}{10},$$
for large $n$. By \eqref{equality} and the finite speed of propagation, we get
$$ g_n\lf(\frac 34\rg)=\eps_n\lf(\frac{3}{4}-A\mu_n\rg)\text{ for } \frac{3}{4}-\frac{1}{3}A\mu_n \leq |y-y_n|\leq 8,$$
hence \eqref{NL_exit+}. 
\end{proof}
\begin{corol}
\label{C:E0d0}
\begin{gather}
\label{energy_equal}
 E_0=E(W_{\ell},\partial_t W_{\ell})=E_{\min}\\
\label{l_and_momentum}
d_0=-E_0\ell \vec{e}_1,
\end{gather}
where $\vec{e}_1=(1,0,\ldots,0)\in \RR^N$.
 \end{corol}
\begin{proof}
By definition, $E_0=\lim_{t\to 1} E(a(t),\partial_t a(t))$. The fact that $E_0=E_{\min}$ follows from the choice of $t_n$ and the strong convergence of the sequence $(a(t_n),\partial_t a(t_n)$. To complete the proof of \eqref{energy_equal}, observe that
$$ E_0=\lim_{n\to \infty} E(a(t_n),\partial_t a(t_n))=E(W_{\ell},\partial_tW_{\ell}).$$
The equality \eqref{l_and_momentum} follows from
$$ d_0=\lim_{n\to \infty} \int \nabla a(t_n)\partial_t a(t_n)=\int \nabla W_{\ell}(0)\partial_t W_{\ell}(0)=-\ell \,E(W_{\ell},\partial_t W_{\ell})\vec{e}_1.$$
(See Claim \ref{C:value_W}.)
\end{proof}

\subsection{Strong convergence for all times and end of the proof}
\label{SS:alltimes}
\begin{lemma}
\label{L:CV_seq}
 Let $\{t'_n\}\in (0,1)^{\NN}$ be any sequence such that $t'_n\to 1$ as $n\to \infty$. Then there exist $\lambda_n'$, $x_n'$ and a sign $\pm$ such that 
\begin{align*}
\lim_{n\to \infty} {\lambda_n'}^{\frac{N}{2}}a(t_n',\lambda_n'x+x_n')&=\pm W_{\ell}(0),\text{ in }\hdot\\
\lim_{n\to \infty} {\lambda_n'}^{\frac{N}{2}-1}\partial_t a(t_n',\lambda_n'x+x_n')&=\pm \partial_t W_{\ell}(0),\text{ in }L^2.
\end{align*}
where $\ell=-\frac{d_0}{E_0}$.
\end{lemma}
\begin{proof}
 Consider a profile decomposition $\lf\{U_{\lin}^j\rg\}_j$, $\lf\{\lambda_{j,n},x_{j,n},t_{j,n}\rg\}_{j,n}$ associated to the sequence $\lf(a\lf(t_n'\rg),\partial_t a\lf(t_n'\rg)\rg)$. Let $\lf\{U^j\rg\}_j$ be the corresponding non-linear profiles. Reordering the profiles, we can assume as usual that all solutions $U^j$, $j\geq 2$ scatter forward and backward in time, that $t_{1,n}=0$, and that $U^1$ does not scatter in either time direction. By the definition of $\AAA$, we deduce that $U^1\in \AAA$. By the Pythagorean expansion of the energy and the $\hdot\times L^2$ norm we get that for all $J$, as $n\to\infty$,
\begin{align}
\label{pyth_energy}
E\lf(a(t_n'),\partial_ta(t_n')\rg)&=E\lf(U^1_0,U^1_1\rg)+\sum_{j=2}^J E\lf(U^j_0,U^j_1\rg)+E\lf(w_{0,n}^J,w_{1,n}^J\rg)+o(1)\\
\label{pyth_norm}
\lf\| a\left(t_n'\right)\rg\|^2_{\hdot}+\frac{N-2}{2}\left\|\partial_t a\lf(t_n'\rg)\rg\|_{L^2}^2&=\sum_{j=1}^J\left(\lf\|U^j\lf(\frac{-t_{j,n}}{\lambda_{j,n}}\rg)\rg\|^2_{\hdot}+\frac{N-2}{2}\lf\|\partial_t U^j\lf(\frac{-t_{j,n}}{\lambda_{j,n}}\rg)\rg\|^2_{L^2}\right)\\
\notag
&+\lf\|w_{0,n}^J\rg\|_{\hdot}^2+\frac{N-2}{2}\left\|w_{1,n}^J\rg\|^2_{L^2}+o(1).
\end{align} 
By \eqref{bound_nabla}, \eqref{pyth_norm} and Claim \ref{C:variational}, we deduce that all the energies in \eqref{pyth_energy} are positive. By Corollary \ref{C:E0d0}, 
$$\lim_{n\to \infty} E\lf(a(t_n'),\partial_ta(t_n')\rg)=E_{\min}\leq E\lf(U^1,\partial_t U^1\rg).$$
As a consequence, $E\lf(U^1,\partial_t U^1\rg)=E_{\min}$ and for all $J\geq 2$,
$$ \lim_{n\to \infty}\sum_{j=2}^J E\lf(U^j_0,U^j_1\rg)+E\lf(w_{0,n}^J,w_{1,n}^J\rg)=0.$$
By Claim \ref{C:variational} again, this shows there are no other non-zero profile than $U^1$ and that $(w_{0,n}^J,w_{1,n}^J)$, which does not depend on $J\geq 2$, goes to $0$ in $\hdot\times L^2$ as $n\to \infty$.

Using that $E(U^1_0,U^1_1)=E_{\min}$, we can apply Proposition \ref{P:strong_CV} to the sequence $t'_n$, which shows that there exists a rotation $\RRR$ of $\RR^N$ (centered at the origin), $x_0\in \RR^N$, $\lambda_0>0$, $\ell'\in (-1,1)$ and a sign $\pm$ such that
$$ U^1(t,x)=\pm \frac{1}{\lambda_0^{\frac{N}{2}-1}}W_{\ell'}\left(\frac{t}{\lambda_0},\RRR\lf(\frac{x-x_0}{\lambda_0}\rg)\right).$$
By Corollary \ref{C:E0d0}, $\ell'=-\frac{E_0}{d_0}$ and
$$\ell \vec{e}_1=\ell'\RRR(\vec{e}_1),$$
which shows that $\RRR$ is a rotation with axis $(0,\vec{e}_1)$, and that $\ell=\ell'$. As a consequence (using that $W_{\ell}$ if invariant by this type of rotation),
$$ \frac{1}{\lambda_0^{\frac{N}{2}-1}}W_{\ell}\left(\frac{t}{\lambda_0},\frac{x-x_0}{\lambda_0}\right)=\frac{1}{\lambda_0^{\frac{N}{2}-1}}W_{\ell'}\left(\frac{t}{\lambda_0},\RRR\lf(\frac{x-x_0}{\lambda_0}\rg)\right),$$
concluding the proof of Lemma \ref{L:CV_seq}.
\end{proof}
\begin{corol}
\label{C:parameters}
There exist parameters $\lambda(t)$ and $x(t)$, defined for $t\in [0,1)$, such that
\begin{equation}
\label{exist_parameters}
\lim_{t\to 1} \lf(\lambda(t)^{\frac{N}{2}-1}a\lf(t,\lambda(t)y+x(t)\right),\lambda(t)^{\frac{N}{2}}\partial_t a\lf(t,\lambda(t)y+x(t)\right)\rg)=\lf(W_{\ell}(0),\partial_t W_{\ell}(0)\rg).
\end{equation} 
Furthermore,
\begin{equation}
\label{control_parameters}
 \lim_{t\to 0} \frac{\lambda(t)}{1-t}=0,\quad \sup_{t\in [0,1)} \frac{|x(t)|}{1-t}\leq C\eta_0^{1/4}.
\end{equation} 
\end{corol}
\begin{proof}
By Proposition \ref{P:strong_CV}, there exists a sequence $t_n\to 1$ such that 
\begin{equation}
\label{strongCV_known}
\lim_{n\to\infty}\inf_{\substack{\lambda_0>0\\x_0}}\left(\left\|\lambda_0^{\frac{N}{2}-1}a(t_n,\lambda_0y+x_0)-W_{\ell}(0)\right\|_{\hdot}+
\left\|\lambda_0^{\frac{N}{2}}\partial_t a(t_n,\lambda_0y+x_0)-\partial_t W_{\ell}(0)\right\|_{L^2}\right)=0.
\end{equation} 
We show \eqref{exist_parameters} by contradiction. 
Assume that there exist $c_0>0$ and a sequence $\tau_n\to 1$ such that for all $n$,
\begin{equation}
\label{strongCV_contra}
 \inf_{\substack{\lambda_0>0\\x_0}}\left(\left\|\lambda_0^{\frac{N}{2}-1}a(\tau_n,\lambda_0y+x_0)-W_{\ell}(0)\right\|_{\hdot}+
\left\|\lambda_0^{\frac{N}{2}}\partial_t a(\tau_n,\lambda_0y+x_0)-\partial_t W_{\ell}(0)\right\|_{L^2}\right)=c_0.
\end{equation} 
In view of \eqref{strongCV_known}, using the continuity of the $\hdot\times L^2$ valued map $t\mapsto (a(t),\partial_t a(t))$, we can change the sequence $\tau_n$ in \eqref{strongCV_contra} so that $0<c_0\leq \|W_{\ell}(0)\|_{\hdot}+\|\partial_t W_{\ell}(0)\|_{L^2}$. By Lemma \ref{L:CV_seq} we get a contradiction, which shows \eqref{exist_parameters}. The estimates \eqref{control_parameters} follow by Lemma \ref{L:estimates}.
\end{proof}
To complete the proof of Theorem \ref{T:classification}, it remains to show the second equality of \eqref{classif_parameters}, which is done in the next lemma:
\begin{lemma}
The translation parameter $x(t)$ of Corollary \ref{C:parameters} satisfies
\begin{equation}
 \label{limit_xt}
\lim_{t\to 1} \frac{x(t)}{1-t}=-\ell \vec{e}_1.
\end{equation}
\end{lemma}
\begin{proof}
It is sufficient to fix a sequence $\{t_n\}_n$ such that $t_n\to 1$, and show that \eqref{limit_xt} holds along a subsequence of $\{t_n\}_n$. 

From \eqref{needed_after} in the proof of Lemma \ref{L:estimates}, we have
\begin{equation}
\label{needed_after'}
\frac{1}{1-t_n}\int_{\RR^N} x\,(e(u)-e(v))(t_n)dx=d_0=-E_0\ell \vec{e}_1=-E(W_{\ell}(0),\partial_tW_{\ell}(0))\vec{e}_1.
\end{equation}
Using that $(v,\partial_t v)$ is continuous from $\RR$ to $\hdot\times L^2$ and that $a$ is supported in $\{|x|\leq 1-t\}$, we get
$$ \frac{1}{1-t_n} \left(\int_{\RR^N} xe(a)(t_n)-\int_{\RR^N} x\,(e(u)-e(v))(t_n)dx)\right)\underset{n\to \infty}{\longrightarrow}0. $$
Expanding
$$ \int_{\RR^N} xe(a)(t_n)=\int_{\RR^N} (x-x(t_n))e(a)(t_n)+x(t_n)\int_{\RR^N} e(a)(t_n),$$
and using \eqref{exist_parameters}, one can show \eqref{limit_xt}. The proof is similar to the end of the proof of Lemma \ref{L:estimates} and we skip it.
\end{proof}

\section{Classification of compact solutions}
\label{S:compact}
In all this section we assume $N\in\{3,4,5\}$.
\begin{defi}
\label{D:compact}
Let $u$ be a solution of \eqref{CP}. We will say that $u$ is \emph{compact up to modulation} when there exist functions $\lambda(t)$, $x(t)$ on $I_{\max}(u)$ such that $K$ defined by \eqref{def_K} has compact closure in $\hdot\times L^2$. 
\end{defi}
Note that if $\lambda(t)$ and $x(t)$ exist as in Definition \ref{D:compact}, we can always replace them by smooth functions of $t$ (see \cite{KeMe06}).

In this section, we show Theorem \ref{T:compact}, i.e that the only solutions that are compact up to modulation and satisfy the bound \eqref{bound_nabla2W} are (up to the transformations of the equation) the solutions $W_{\ell}$. After a preliminary subsection about modulation parameters around $W_{\ell}$, we show in \S \ref{SS:global} that all compact solutions are globally defined. In \S \ref{SS:subseq} we show that there exists two sequences of times (one going to $+\infty$, the other to $-\infty$) for which the solution converges to $W_{\ell}$ up to a time dependent modulation. In \S \ref{SS:endofproof} we conclude the proof. In \S \ref{SS:bound}, we prove a general version of Corollary \ref{C:global}.
\subsection{Modulation around the solitary wave}
\label{SS:modulation}
We first introduce some modulation parameters around $W_{\ell}$, adapting the modulation around $W$ in \cite{DuMe08} to the more general case of $W_{\ell}$. The proofs, which are very similar to the ones of \cite[Appendix A]{DuMe08}, are sketched in Appendix \ref{A:modulation}. Consider a solution $u$ of \eqref{CP} such that for some $\ell\in(-1,+1)$,
\begin{equation}
\label{cond_energy}
E(u_0,u_1)=E\lf(W_{\ell}(0),\partial_t W_{\ell}(0)\rg)\text{ and }\int \nabla u_0\,u_1=\int \nabla W_{\ell}(0)\,\partial_tW_{\ell}(0).
\end{equation} 
Let $d_{\ell}$ be defined by
\begin{equation}
\label{def_dl}
d_{\ell}(t)=\int |\nabla u(t)|^2\,dx+\int(\partial_t u(t))^2\,dx-\int |\nabla W_{\ell}(0)|^2\,dx-\int (\partial_t W_{\ell}(0))^2\,dx. 
\end{equation} 
As in the case $\ell=0$, we have the following trapping property:
\begin{claim}
\label{C:trapping}
Let $u$ be a solution such that \eqref{cond_energy} holds. 
\begin{itemize}
 \item If $d_{\ell}(0)=0$, there exist $\lambda_0>0$, $x_0\in \RR^N$ and a sign $\pm$ such that
$$u(t,x)= \frac{\pm1}{\lambda_0^{\frac{N-2}{2}}}W_{\ell}\left(\frac{t}{\lambda_0},\frac{x-x_0}{\lambda_0}\right);$$
\item If $d_{\ell}(0)>0$, then for all $t$ in the domain of existence of $u$, $d_{\ell}(t)>0$;
\item If $d_{\ell}(0)<0$, then for all $t$ in the domain of existence of $u$, $d_{\ell}(t)<0$.
\end{itemize}
\end{claim}
We refer to Appendix \ref{A:modulation} for the proof of Claim \ref{C:trapping}. The next proposition, which is again proved in Appendix \ref{A:modulation}, states that, for small $d_{\ell}(t)$ it is possible to modulate $u$ so that it satisfies suitable orthogonality conditions.
\begin{lemma}
 \label{L:modulation}
Assume \eqref{cond_energy}.
There exists a small $\delta_0=\delta_0(\ell)>0$ such that if $|d_{\ell}(t)|<\delta_0$ on a time-interval $I$, then there exist $C^1$ functions $\lambda(t)>0$, $x(t)\in \RR^N$, $\alpha(t)\in \RR$, defined for $t\in I$ and a sign $\pm$ such that
\begin{equation*}
 \lambda(t)^{\frac{N-2}{2}}u(t,\lambda(t)x+x(t))=\pm(1+\alpha(t))W_{\ell}(0,x)+f(t,x),
\end{equation*}
where $\tilde{f}(t,x)=f\lf(t,\sqrt{1-\ell^2}x_1,x_2,\ldots,x_N\rg)$ satisfies
$$\tilde{f}\in \left\{W,\partial_{x_1}W,\ldots,\partial_{x_N}W,\frac{N-2}{2}W+x\cdot\nabla W\rg\}^{\bot} \text{ in }\hdot\lf(\RR^N\rg).$$ 
Furthermore, the following estimates hold for $t\in I$:
\begin{gather}
\label{modul1}
|\alpha(t)|\approx \lf\|\nabla\lf(\alpha(t) W_{\ell}(0)+f(t)\rg)\rg\|_{L^2}\approx \lf\|\nabla f(t)\rg\|_{L^2}+\lf\|\partial_t u(t)+\ell\partial_{x_1}u(t)\rg\|_{L^2}\approx |d_{\ell}(t)|.\\
\label{modul2}
|\lambda'(t)|+|x'(t)-\ell \vec{e}_1|+\lambda|\alpha'(t)|\leq C|d_{\ell}(t)|.
\end{gather}
Here the implicit constants in \eqref{modul1} and the constant $C$ in \eqref{modul2} might depend on $\ell$, but are independent of $u$ and $t$.
\end{lemma}

\subsection{Global existence}
\label{SS:global}
In this subsection, we show that all solutions of \eqref{CP} which are compact up to modulation and satisfy the bound \eqref{bound_nabla2W} are globally defined. We start by showing that solutions that are compact up to modulation have positive energy.
\begin{lemma}
 \label{L:positive_energy}
Let $u$ be a nonzero solution of \eqref{CP} which is compact up to modulation. Then $E(u_0,u_1)>0$.
\end{lemma}
\begin{proof}
 Assume $E(u_0,u_1)\leq 0$. By Claim \ref{C:variational}, $\|\nabla u_0\|_{L^2}^2>\|\nabla W\|_{L^2}^2$. By \cite{KeMe08}, $u$ blows up in finite time in both time directions. We denote $T_{\pm}=T_{\pm}(u)$ the finite times of existence.

The fact that $u$ is compact up to modulation implies that it is bounded in $\hdot\times L^2$. Furthermore, by Lemmas 4.7 and  4.8 in \cite{KeMe08}, $\lambda(t)\to 0$ as $t\to T_{\pm}$ and there exist two blow-up points, $x_+,x_-\in \RR^N$ such that 
$$\supp (u,\partial_t u)\subset \left\{|x-x_+|\leq |T_+-t|\right\}\cap \left\{|x-x_-|\leq |T_--t|\right\}.$$
Let 
$$y(t)=\int_{\RR^N} u^2(t,x)\,dx.$$
Then by \eqref{identity5} in Claim \ref{C:identity2} and the fact that $E(u_0,u_1)\leq 0$, 
\begin{equation}
\label{y''}
y''(t)\geq 0. 
\end{equation} 
Furthermore by Hardy's inequality and the properties of the support of $u$:
\begin{equation}
\label{y0}
\lim_{t\to T_+}y(t)=\lim_{t\to T_-}y(t)=0. 
\end{equation} 
By \eqref{y''} and \eqref{y0}, $y(t)=0$ for all $t$, which shows that $u=0$, contradicting our assumptions.
\end{proof}

The main result of this subsection is the following:
\begin{prop}
\label{P:global}
 Let $u$ be a solution of \eqref{CP} that satisfies \eqref{bound_nabla2W} and such that $E(u_0,u_1)>0$. Assume that there exist $\lambda(t)>0$, $x(t)\in \RR^N$ defined for $t\geq 0$ such that
$$K_+=
\Big\{\left(\lambda(t)^{\frac{N}{2}-1}u(t,\lambda(t)x+x(t)),\lambda(t)^{\frac{N}{2}}\partial_t u(t,\lambda(t)x+x(t))\right)\;:\; t\in [0,T_+(u))\Big\}$$
has compact closure in $\hdot\times L^2$. Then $T_+(u)=+\infty$.
\end{prop}

We argue by contradiction, assuming that $T_+(u)$ is finite. Without loss of generality, one may assume $T_+(u)=1$. As in Remark \ref{R:compactness}, we will assume that $\int \nabla u_0 u_1$ is parallel to $\vec{e}_1=(1,0,\ldots,0)$ and define $\ell$ by \eqref{def_ell}.

As seen in the proof of Lemma \ref{L:positive_energy}, there exists an unique blow-up point (that we will assume to be $x=0$). Moreover, $\lambda(t)\to 0$ as $t\to 1$ and
$$ \supp u(t)\subset\{|x|\leq 1-t\}$$
Furthermore by \cite[p.144-145]{BaGe99},
$$ \lambda(t)+|x(t)|\leq C(1-t).$$
We will need the following result, which is proved in \cite[Section 6]{KeMe08}:
\begin{exttheo}
\label{T:no_ss}
Let $u$ satisfies the assumptions of Proposition \ref{P:global}. Then
there exists a sequence $\{t_n\}\in [0,1)^{\NN}$ such that
\begin{equation}
\label{subseq}
 \lim_{n\to \infty}t_n=1,\quad \lim_{n\to \infty}\frac{\lambda(t_n)}{1-t_n}=0.
\end{equation} 
\end{exttheo}
We divide the proof of Proposition \ref{P:global} into a few lemmas.
\begin{lemma}[Control of the space translation]
\label{L:cpct_space}
Let $u$ be a solution which is compact up to modulation and such that $T_+=1$. Let $\{t_n\}\in [0,1)^{\NN}$ be any sequence that satisfies \eqref{subseq}. Then
$$\lim_{n\to \infty}\frac{x(t_n)}{1-t_n}=-\ell \vec{e}_1.$$
\end{lemma}
\begin{proof}
Let $\Psi(t)=\int xe(u)$, where $e(u)$ is defined by \eqref{defeu}

By \eqref{def_ell} and conservation of momentum and identity \eqref{identity4'} in Claim \ref{C:identity2},
\begin{equation}
\label{Y'}
 \Psi'(t)=-\int \nabla u(t)\partial_t u(t)=\ell E(u_0,u_1)\vec{e}_1.
\end{equation}
Write 
\begin{equation}
\label{intermY}
\Psi(t)=x(t)E(u_0,u_1)+\int_{|x|\leq 1-t} (x-x(t))e(u),
\end{equation} 
where $e(u)(t,x)$ is defined by \eqref{defeu}.
Fix $\eps>0$. Using the compactness of $K_+$, one may find $A_{\eps}>0$ such that
\begin{equation}
\label{def_Aeps}
\forall t\in [0,1),\quad \int_{|x-x(t)|\geq A_{\eps}\lambda(t)}r(u)\leq \eps,
\end{equation} 
where
\begin{equation}
 \label{defru}
r(u)(t,x)=|\nabla u(t,x)|^2+(\partial_t u(t,x))^2+|u|^{\frac{2N}{N-2}}+\frac{1}{|x|^2}|u|^2.
\end{equation}
Then
\begin{equation*}
\left|\int (x-x(t))e(u)\right|=\left|\int_{|x-x(t)|\leq A_{\eps}\lambda(t)} (x-x(t))e(u) +\int_{|x-x(t)|\geq A_{\eps}\lambda(t)} (x-x(t))e(u) \rg|,
\end{equation*} 
and thus, in view of the bound $|x(t)|\leq C(1-t)$, and the fact that $|x|\leq 1-t$ on the support of $u$,
\begin{equation*}
\left|\int_{|x|\leq 1-t} (x-x(t))e(u)\right|\leq CA_{\eps}\lambda(t)+C\,\eps(1-t).
\end{equation*} 
By \eqref{subseq}, and using that $\eps>0$ is arbitrary, we get in view of \eqref{intermY},
$$\lim_{n\to +\infty}\lf|\frac{1}{1-t_n}\left(\Psi(t_n)-x(t_n)E(u_0,u_1)\rg)\rg|=0.$$
Using that, by \eqref{Y'},
$$ \Psi(t_n)=-\vec{e}_1\int_{t_n}^1 \ell E(u_0,u_1)dt=-\ell E(u_0,u_1)(1-t_n)\vec{e}_1,$$
we get the conclusion of the lemma.
\end{proof}
We next show:
\begin{lemma}
\label{L:cpct_subseq}
 Let $u$ and $\{t_n\}_n$ be as in Lemma \ref{L:cpct_space}. Then
$$ \lim_{n\to \infty} \frac{1}{1-t_n}\int_{t_n}^1\int |\partial_t u(t)+\ell\partial_{x_1}u(t)|^2\,dx\,dt=0.$$
\end{lemma}
\begin{proof}
 Let 
\begin{multline*}
Z(t)=(\ell^2-1)\int (x+\ell (1-t)\vec{e}_1)\cdot \nabla u\partial_t u\,dx+\frac{N-2}{2}\lf(\ell^2-1\rg)\int u\partial_t u\,dx\\ -\ell^2\int \lf(x_1+\ell(1-t)\rg)\partial_{x_1}u\partial_t u\,dx.
\end{multline*}
Then by Claim \ref{C:identity2} and using that $\int \nabla u_0 u_1=-\ell E(u_0,u_1)\vec{e}_1$, we get
\begin{equation*}
 Z'(t)=\int \lf(\partial_tu +\ell\partial_{x_1}u\rg)^2\,dx.
\end{equation*} 
Integrating the preceding equality between $t_n$ and $1$, we see that it is sufficient to show:
\begin{equation}
\label{limZ}
\lim_{n\to \infty} \frac{Z(t_n)}{1-t_n}=0. 
\end{equation}
We first show: 
\begin{equation}
\label{limZ1}
\lim_{n\to\infty}\frac{1}{1-t_n} \left|\int u(t_n)\partial_t u(t_n)\,dx\right|=0. 
\end{equation} 
Fix $\eps>0$, and let $A_{\eps}$ satisfying \eqref{def_Aeps}. Then
\begin{multline*}
 \int u(t_n)\partial_t u(t_n)\,dx=\int_{|x-x(t_n)|\geq A_{\eps}\lambda(t_n)}|x-x(t_n)|\frac{1}{|x-x(t_n)|}u(t_n)\partial_t u(t_n)\,dx\\
+\int_{|x-x(t_n)|\leq A_{\eps}\lambda(t_n)}|x-x(t_n)|\frac{1}{|x-x(t_n)|}u(t_n)\partial_t u(t_n)\,dx,
\end{multline*}
and we get, as in the proof of Lemma \ref{L:cpct_space} (and using Hardy's inequality),
$$ \left|\int u(t_n)\partial_t u(t_n)\,dx\right|\leq C\,\eps(1-t_n)+CA_{\eps}\lambda(t_n).$$
Using \eqref{subseq}, and the fact that $\eps$ is arbitrary in the preceding equality, we get \eqref{limZ1}. We next show
\begin{equation}
\label{limZ2}
\lim_{n\to\infty}\frac{1}{1-t_n} \left|\int (x+\ell (1-t_n)\vec{e}_1)\cdot \nabla u(t_n)\partial_t u(t_n)\,dx\right|=0. 
\end{equation} 
Fix again $\eps>0$, and $A_{\eps}$ as in \eqref{def_Aeps}, and divide the integral between the regions $|x-x(t_n)|\leq A_{\eps}\lambda(t_n)$ and $|x-x(t_n)|\geq A_{\eps}\lambda(t_n)$. By \eqref{def_Aeps} and again the fact that $|x|\leq 1-t$ on the support of $u$, 
\begin{equation*}
 \left|\int_{|x-x(t_n)|\geq A_{\eps}\lambda(t_n)} (x+\ell (1-t_n)\vec{e}_1)\cdot \nabla u(t_n)\partial_t u(t_n)\,dx\right|\leq C(1-t_n)\eps.
\end{equation*} 
Furthermore, if $|x-x(t_n)|\leq A_{\eps}\lambda(t_n)$, then
\begin{equation*}
|x+\ell (1-t_n)\vec{e}_1|\leq |x-x(t_n)|+|x(t_n)+\ell (1-t_n)\vec{e}_1|\leq A_{\eps}\lambda(t_n)+|x(t_n)+\ell (1-t_n)\vec{e}_1|,    
 \end{equation*}
which shows by Lemma \ref{L:cpct_space} that
\begin{equation*}
 \frac{1}{1-t_n}\lim_{n\to \infty}\left|\int_{|x-x(t_n)|\leq A_{\eps}\lambda(t_n)} (x+\ell (1-t_n)\vec{e}_1)\cdot \nabla u(t_n)\partial_t u(t_n)\,dx\right|=0.
\end{equation*}  
Combining these estimates and using that $\eps>0$ is arbitrary, we get \eqref{limZ2}. To conclude the proof of \eqref{limZ}, and thus of the lemma, it remains to show 
\begin{equation}
\label{limZ3}
\lim_{n\to\infty}\frac{1}{1-t_n} \left|\int (x_1+\ell (1-t_n))\partial_{x_1}u(t_n)\partial_t u(t_n)\,dx\right|=0. 
\end{equation} 
The proof of \eqref{limZ3} is the same than the one of \eqref{limZ2} and therefore we omit it.
\end{proof}

To show Proposition \ref{P:global} it remains to prove the following proposition:
\begin{prop}
\label{P:end_finite_time}
 There is no function $u$ as in Proposition \ref{P:global} such that $T_+=1$ and for some sequence $t_n\to 1$,
\begin{equation}
\label{hyp_tn}
\lim_{n\to \infty} \frac{1}{1-t_n}\int_{t_n}^1\int |\partial_t u(t)+\ell\partial_{x_1}u(t)|^2\,dt\,dx=0,
\end{equation} 
where $\ell$ is defined by \eqref{def_ell}.
\end{prop}
Let us first show:
\begin{lemma}
\label{L:ene_mom}
Let $u$ be as in Proposition \ref{P:end_finite_time}. Then $\ell\in (-1,+1)$,
\begin{gather*}
E(u_0,u_1)=E(W_{\ell}(0),\partial_tW_{\ell}(0)=\frac{1}{\sqrt{1-\ell^2}}E(W,0),\\
\int \nabla u_0\,u_1=\int \nabla W_{\ell}(0)\partial_t W_{\ell}(0)=-\frac{\ell}{\sqrt{1-\ell^2}}E(W,0)\vec{e}_1.
\end{gather*}
\end{lemma}
\begin{proof}
In view of Lemma \ref{L:cpct_subseq}, one may show, using the argument of the proof of Corollary 5.3 in \cite{DuKeMe09P}, that there exists a sequence $\lf\{t_n'\rg\}_n$ such that in $\hdot\times L^2$
$$ \lim_{n\to +\infty} \lf(\lambda^{\frac{N-2}{2}}(t_n')u(t_n',\lambda(t_n')x+x(t_n')),\lambda^{\frac{N}{2}}(t_n')\partial_t u(t_n',\lambda(t_n')x+x(t_n'))\rg)=(U_0,U_1),$$
and the solution $U$ of \eqref{CP} with initial condition $(U_0,U_1)$ satisfies for some $T\in (0,T_+(U))$:
$$ \int_0^{T} \int_{\RR^N}|\partial_t U+\ell\partial_{x_1}U|^2=0.$$
As a consequence,
\begin{equation}
\label{equation_U}
\partial_t U+\ell\partial_{x_1}U=0\text{ in }(0,T)\times \RR^N.
\end{equation} 
Differentiating with respect to $t$, we get
\begin{equation}
\label{eq_U}
\Delta U+|U|^{\frac{4}{N-2}}U -\ell^2\partial_{x_1}^2U=0\text{ in }(0,T)\times \RR^N.
\end{equation} 
%
Using that $U(0)$ satisfies the equation \eqref{asym_ell}, that by \eqref{bound_nabla2W}, $\|\nabla U(0)\|_{L^2}^2<\frac{4\sqrt{N-1}}{N}\|\nabla W\|_{L^2}^2$, 
and that $U\neq 0$ (the energy of $U$ is positive), we get by Lemma \ref{L:asym_ell} that $\ell^2<1$ and that there exists $\lambda_0>0$, $x_0\in \RR^N$ such that 
$$U_0(x)=\pm \frac{1}{\lambda_0^{\frac{N}{2}-1}}W_{\ell}\left(0,\frac{x-x_0}{\lambda_0}\right).$$
By \eqref{equation_U}, we get
$$U_1(x)=\pm \frac{1}{\lambda_0^{\frac{N}{2}}}\partial_t W_{\ell}\left(0,\frac{x-x_0}{\lambda_0}\right),$$
which shows that
$$ U(t,x)=\pm \frac{1}{\lambda_0^{\frac{N}{2}-1}}W_{\ell}\left(t,\frac{x-x_0}{\lambda_0}\right).$$
The conclusion of the lemma follows by conservation of energy and momentum.
\end{proof}
We are now ready to prove Proposition \ref{P:end_finite_time}. Let us mention that this part of the proof fills a small gap in the paper \cite{DuMe08}. Indeed Proposition 2.7 of this paper is a direct consequence of \cite{KeMe08} only in the case of self-similar blow-up. To show that $T_+(u)=+\infty$ under the general assumption of Proposition 2.7 of \cite{DuMe08}, one must use the Steps 1, 3 and 4 of the proof below (Step 2 is only needed in the case of nonzero momentum).

Recall from \S \ref{SS:modulation} the definition of $d_{\ell}(t)$ and $\delta_0$. By \S \ref{SS:modulation}, if $|d_{\ell}(t)|<\delta_0$, there exist $\lambda(t)>0$, $x(t)\in \RR^N$ and $\alpha(t)$ such that
\begin{gather*}
 \lambda(t)^{\frac{N-2}{2}}u(t,\lambda(t)x+x(t))=(1+\alpha(t))W_{\ell}(0,x)+f(t,x),\\
\|f\|_{\hdot}+|\alpha|+\|\partial_t u+\ell\partial_{x_1}u\|_{L^2}\leq C|d_{\ell}(t)|. 
\end{gather*}
It is easy to see that we can replace the $\lambda(t)$ and $x(t)$ defining $K_+$ by the above $\lambda(t)$ and $x(t)$ for all $t$ such that $|\delta_{\ell}(t)|<\delta_0$, without losing the compactness of $\overline{K}_+$ in $\hdot\times L^2$, which we will do in the remainder of this proof. For these $x(t)$ and $\lambda(t)$ we still have
\begin{equation}
\label{bound_lambda_x}
\forall t\in [0,1), \quad |x(t)|+|\lambda(t)|\leq C(1-t).
\end{equation} 
Let
\begin{equation}
\label{def_Phi}
\Phi(t)=(N-2)\int (x+(1-t)\ell e_1)\cdot \nabla u\partial_tu+\frac{(N-2)(N-1)}{2}\int u\partial_t u.
\end{equation}
By the conservation of momentum, \eqref{identity1'} and \eqref{identity3'} in Claim \ref{C:identity2}, and the fact that $\int \nabla u\partial_tu=-\ell E(u_0,u_1)\vec{e}_1$, we get
\begin{equation}
\label{dif_Phi}
\Phi'(t)=d_{\ell}(t).
\end{equation}
\EMPH{Step 1. Bound on $\lambda(t)$}
Let us show
\begin{equation}
 \label{bound_lambda}
|\lambda(t)|\leq C(1-t)\lf|d_{\ell}(t)\rg|^{\frac{2}{N-2}}.
\end{equation} 
If $|d_{\ell}(t)|\geq \delta_0$, the bound follows from \eqref{bound_lambda_x}. Let us assume that
$|d_{\ell}(t)|\leq \delta_0$. Then by \S \ref{SS:modulation} and the choice of $\lambda(t)$ and $x(t)$, we have
$$ u(t,x)=\frac{1}{\lambda(t)^{\frac{N-2}{2}}}W_{\ell}\left(0,\frac{x-x(t)}{\lambda(t)}\right)+\frac{1}{\lambda(t)^{\frac{N-2}{2}}}\eps\left(t,\frac{x-x(t)}{\lambda(t)}\right),$$
where $\|\eps(t)\|_{\hdot}\leq C|d_{\ell}(t)|$. Using \eqref{bound_lambda_x} and that on the support of $u$, $|x|\leq 1-t$, we obtain that $u(t,x)=0$ if $|x-x(t)|\geq C_1(1-t)$ for some large constant $C_1$. In particular
\begin{multline*}
\int_{|x-x(t)|\geq C_1(1-t)}\frac{1}{\lambda(t)^N}\lf|\nabla W_{\ell}\lf(0,\frac{x-x(t)}{\lambda(t)}\rg)\rg|^2dx\\
=\int_{|x-x(t)|\geq C_1(1-t)}\frac{1}{\lambda(t)^N}\lf|\nabla \eps\lf(t,\frac{x-x(t)}{\lambda(t)}\rg)\rg|^2dx\leq C\lf(d_{\ell}(t)\rg)^2.
\end{multline*}
As a consequence
\begin{equation}
\label{interm_lambda}
 C|d_{\ell}(t)|^2\geq \int_{|y|\geq \frac{C_1(1-t)}{\lambda(t)}}\lf|\nabla W_{\ell}\lf(0,y\rg)\rg|^2dy\geq c\lf(\frac{\lambda(t)}{1-t}\rg)^{N-2},
\end{equation} 
hence \eqref{bound_lambda}.
The last inequality in \eqref{interm_lambda} follows from the expression \eqref{def_Wl} of $W_{\ell}$. Indeed $|\nabla W_{\ell}(0,y)|\approx |y|^{-(N-1)}$ for large $y$  and thus $\int_{|y|\geq A}|\nabla W_{\ell}(0,y)|^2\,dy\approx A^{2-N}$ for large $A>0$. 

\EMPH{Step 2} Let 
$$ y_{\ell}(t)=x(t)+(1-t)\ell \vec{e}_1.$$
In this step we show
\begin{equation}
 \label{bound_y}
\lf|y_{\ell}(t)\rg|\leq C(1-t)\lf|d_{\ell}(t)\rg|^{1+\frac{2}{N}}.
\end{equation} 
We define $S(t)$ by
\begin{equation}
\label{def_S}
S(t)=\int_{\RR^N} \lf(x+(1-t)\ell \vec{e}_1\rg)e(u)dx,
\end{equation}
where $e(u)$ is the density of energy defined in \eqref{defeu}. 
Then using that $u$ is a solution of \eqref{CP} such that, by Lemma \ref{L:ene_mom}, 
$$ E(u_0,u_1)=\frac{1}{\sqrt{1-\ell^2}}E(W,0),\quad \int \nabla u_0\,u_1=-\frac{\ell}{\sqrt{1-\ell^2}}E(W,0)\vec{e}_1.
$$
we get by \eqref{identity4'} in Claim \ref{C:identity2} that $S'(t)=0$. Furthermore, as $|x|\leq 1-t$ on the support of $u$, we get that $S(t)\to 0$ as $t\to 1$, which shows
that $S(t)$ is identically $0$. As a consequence
\begin{equation}
\label{exp_yl}
 y_{\ell}(t)E(u_0,u_1)=-\int(x-x(t))e(u).
\end{equation} 
 It remains to show
\begin{equation}
\label{bound_xe}
 \left|\int(x-x(t))e(u)\rg|\leq C(1-t)|d_{\ell}(t)|^{1+\frac{2}{N}}.
\end{equation} 
If $|d_{\ell}(t)|\geq \delta_0$, where $\delta_0$ is given by Lemma \ref{L:modulation}, the bound follows from the fact that $u$ is supported in the light cone $\{|x|\leq 1-t\}$ and from the bound  on $x(t)$ in \eqref{bound_lambda_x}.

Assume $|d_{\ell}(t)|<\delta_0$. Then by Lemma \ref{L:modulation}, one has
\begin{align}
\label{u=W+eps1}
u(t,x)&=\frac{1}{\lambda(t)^{\frac{N-2}{2}}}W_{\ell}\left(0,\frac{x-x(t)}{\lambda(t)}\right)+\frac{1}{\lambda(t)^{\frac{N-2}{2}}}\eps\left(t,\frac{x-x(t)}{\lambda(t)}\right)\\
\label{u=W+eps2}
\partial_t u(t,x)&=\frac{1}{\lambda(t)^{\frac{N}{2}}}\partial_t W_{\ell}\left(0,\frac{x-x(t)}{\lambda(t)}\right)+\frac{1}{\lambda(t)^{\frac{N}{2}}}\eps_1\left(t,\frac{x-x(t)}{\lambda(t)}\right),
\end{align}
where
\begin{equation}
\label{small_eps}
\|\eps(t)\|_{\hdot}+\|\eps_1\|_{L^2}\leq C\lf|d_{\ell}(t)\rg|
 \end{equation} 
(the bound on $\eps_1$ follows from the bound $\|\partial_t u+\ell\partial_{x_1}u\|_{L^2}\lesssim d_{\ell}(t)$).
Then, developing the density of energy $e(u)$,
\begin{multline}
\label{dev_eu}
  \left|\int(x-x(t))e(u)\rg|= \left|\int_{|x-x(t)|\leq C_1(1-t)} (x-x(t))e(u)\rg|\\
\lesssim \left|\int_{|x-x(t)|\leq C_1(1-t)} (x-x(t))e\left(W_{\ell,\lambda(t),x(t)}(0,x)\right)\rg|+R(t)+(1-t)|d_{\ell}(t)|^2,
\end{multline}
where we have denoted by 
$$W_{\ell,\lambda(t),x(t)}(s,x)=\frac{1}{\lambda(t)^{\frac{N-2}{2}}}W_{\ell}\left(s,\frac{x-x(t)}{\lambda(t)}\right),$$
and
\begin{multline}
\label{def_R}
 R(t)=\int_{|x-x(t)|\leq C_1(1-t)} \frac{|x-x(t)|}{\lambda(t)^N}\left|\nabla_{t,x} W_{\ell}\lf(0,\frac{x-x(t)}{\lambda(t)}\rg)\rg|
 \times \left| \sqrt{\left|\nabla_{x} \eps\right|^2+|\eps_1|^2}\lf(t,\frac{x-x(t)}{\lambda(t)}\rg)\rg|\,dx\\
+\int_{|x-x(t)|\leq C_1(1-t)} \frac{|x-x(t)|}{\lambda(t)^N}\left|W_{\ell}\lf(0,\frac{x-x(t)}{\lambda(t)}\rg)\rg|^{\frac{N+2}{N-2}}\times \left|\eps\lf(t,\frac{x-x(t)}{\lambda(t)}\rg)\rg|\,dx.
\end{multline}
We have used the notation $|\nabla_{t,x}v|^2=|\nabla v|^2+|\partial_t v|^2$.
The first term in the second line of \eqref{dev_eu} is $0$ by the parity of $\lf|W_{\ell}(0)\rg|$ and $\lf|\partial_t W_{\ell}(0)\rg|$. 
Let us show
\begin{equation}
\label{bound_R}
R(t)\leq C|d_{\ell}(t)|^{1+\frac{2}{N}}(1-t),
 \end{equation} 
which would conclude this step. We show the bound \eqref{bound_R} on the first term $R_1$ in \eqref{def_R}, the proof of the bound on the second term is similar. First remark that by the change of variable $y=\frac{|x-x(t)|}{\lambda(t)}$,
\begin{equation*}
 R_1(t)=\lambda(t)\int_{|y|\leq C_1\frac{1-t}{\lambda(t)}} |y|\big|\nabla_{t,x} W_{\ell}\lf(0,y\rg)\big|\,
\sqrt{\lf|\nabla_{x} \eps\lf(t,y\rg)\rg|^2+|\eps_1(t,y)|^2}\,dy.
 \end{equation*}
Let $A=A(t)\geq 1$ be a parameter and divide the preceding integral between the regions $|y|\geq A$ and $|y|\leq A$. By Cauchy-Schwarz and using the explicit decay of $W_{\ell}(0,y)$ as $|y|\to \infty$, we get 
\begin{multline*}
 \lambda(t)\int_{A\leq |y|\leq C_1\frac{1-t}{\lambda(t)}} |y|\,\big|\nabla_{t,x} W_{\ell}\lf(0,y\rg)\big|
\,\big(\sqrt{\lf|\nabla_{x} \eps\lf(t,y\rg)\rg|^2+|\eps_1(t,y)|^2}\,dy\\
\leq C(1-t)|d_{\ell}(t)|\sqrt{\int_{|y|\geq A}\left|\nabla_{t,x} W_{\ell}\lf(0,y\rg)\rg|^2}
\leq C(1-t)|d_{\ell}(t)|A^{1-\frac{N}{2}}.
\end{multline*}
By Cauchy-Schwarz:
\begin{equation*}
 \lambda(t)\int_{|y|\leq \min\{C_1\frac{1-t}{\lambda(t)},A\}} |y|\,\big|\nabla_{t,x} W_{\ell}\lf(0,y\rg)\big|\,
\sqrt{\lf|\nabla_{x} \eps\lf(t,y\rg)\rg|^2+|\eps_1(t,y)|^2}\,dy
\leq \lambda(t)A|d_{\ell}(t)|.
\end{equation*}
Taking $A=C\left(\frac{1-t}{\lambda(t)}\right)^{\frac{2}{N}}$ and combining the two bounds with \eqref{bound_lambda},  we obtain \eqref{bound_R}, which concludes step 2.

\EMPH{Step 3. Bound on $\Phi(t)$}
Let us show
\begin{equation}
\label{bound_Phi}
 |\Phi(t)|\leq C(1-t)|d_{\ell}(t)|^{1+\frac 2N}.
\end{equation} 
As usual, the bound for $|d_{\ell}(t)|\geq \delta_0$ follows from the condition on the support of $u$ and from the bound $|x(t)|\leq C(1-t)$. 
Let us assume that $|d_{\ell}(t)|<\delta_0$. 
Write
\begin{multline}
\label{re_Phi}
\Phi(t)=(N-2)y_{\ell}(t)\cdot\int_{|x-x(t)|\leq C_1(1-t)} \nabla u\partial_tu\\
+(N-2)\int_{|x-x(t)|\leq C_1(1-t)} (x-x(t))\cdot \nabla u\partial_tu+\frac{(N-2)(N-1)}{2}\int_{|x-x(t)|\leq C_1(1-t)} u\partial_t u.
\end{multline}
The first term of \eqref{re_Phi} is bounded by step 2.
To handle the other terms, decompose $u$ as in \eqref{u=W+eps1}, \eqref{u=W+eps2}. Then
\begin{multline*}
 \left|\int_{|x-x(t)|\leq C_1(1-t)}(x-x(t))\nabla u\partial_t u  \right|\leq CR(t)+C(1-t)|d_{\ell}(t)|^2\\+\left|\int_{|x-x(t)|\leq C_1(1-t)}(x-x(t))\nabla W_{\ell}\lf(0,\frac{x-x(t)}{\lambda(t)}\rg)\,\partial_t W_{\ell}\lf(0,\frac{x-x(t)}{\lambda(t)}\rg)\right|,
\end{multline*}
where $R(t)$ is defined by \eqref{def_R}. Noting that the last integral is $0$ by the parity of $W_{\ell}$, and bounding $R(t)$ by \eqref{bound_R}, we get 
$$ \left|\int_{|x-x(t)|\leq C_1(1-t)}(x-x(t))\nabla u\partial_t u  \right|\leq C(1-t)|d_{\ell}(t)|^{1+\frac{2}{N}}.$$
Writing 
$$ \int_{|x-x(t)|\leq C_1(1-t)} u\partial_tu=\int_{|x-x(t)|\leq C_1(1-t)}|x-x(t)|\frac{1}{|x-x(t)|}u\partial_tu,$$
and using the same argument, we get the bound
$$ \left|\int_{|x-x(t)|\leq C_1(1-t)}u\partial_t u  \right|\leq C(1-t)|d_{\ell}(t)|^{1+\frac{2}{N}},$$
which completes step 3.

\EMPH{Step 4. End of the proof} 
By \eqref{bound_Phi}, then \eqref{dif_Phi}, 
\begin{equation}
\label{last_ineg_Phi}
 |\Phi(t)|\leq C(1-t)|d_{\ell}(t)|^{1+\frac 2N}\leq C(1-t)\lf|\Phi'(t)\rg|^{1+\frac 2N}.
\end{equation}
Thus
\begin{equation*}
 \frac{1}{(1-t)^{\frac{1}{1+2/N}}}\leq \frac{C\lf|\Phi'\rg|}{\lf|\Phi\rg|^{\frac{1}{1+2/N}}}.
\end{equation*}
Integrating and using that $\frac{1}{1+2/N}<1$, we obtain
$$(1-t)^{1-\frac{1}{1+2/N}}\leq C |\Phi(t)|^{1-\frac{1}{1+2/N}},$$
and thus
\begin{equation}
\label{self_sim_Phi}
C\frac{|\Phi(t)|}{1-t}\geq 1.
\end{equation}
By the proof of Lemma \ref{L:ene_mom}, there exists a sequence of times $t_n'\to 1$ such that $d_{\ell}(t_n')\to 0$.
Applying the first inequality of \eqref{last_ineg_Phi} to this sequence, we get
\begin{equation*}
 \lim_{n\to \infty} \frac{1}{1-t_n'}|\Phi(t_n')|=0,
\end{equation*}
which contradicts \eqref{self_sim_Phi}. The proof of Proposition \ref{P:end_finite_time} is complete.
\qed

\subsection{Convergence for a sequence of times}
\label{SS:subseq}
\begin{lemma}
\label{L:subseq}
 Let $u$ be a solution which is compact up to modulation, globally defined and satisfies the bound \eqref{bound_nabla2W}. 
Assume after a space rotation around the origin that there exists a $\ell\in\RR$ such that
\begin{equation*}
-\frac{\int\nabla u_0u_1}{E(u_0,u_1)}=\ell \vec{e}_1
 \end{equation*} 
Then $|\ell|<1$, and there exist $t_n\to +\infty$, $\lambda_0>0$, $x_0\in \RR^N$ and a sign $\pm$ such that 
\begin{multline*}
\lim_{n\to \infty} \left(\lambda(t_n)^{\frac{N-2}{2}}u\lf(t_n,\lambda(t_n)x+x(t_n)\rg),\lambda(t_n)^{\frac{N}{2}}\partial_t u\lf(t_n,\lambda(t_n)x+x(t_n)\rg)\rg)\\
=\pm\lf( \frac{1}{\lambda_0^{\frac{N-2}{2}}}W_{\ell}\lf(0,\frac{x-x_0}{\lambda_0}\rg),\frac{1}{\lambda_0^{\frac{N}{2}}}\partial_t W_{\ell}\lf(0,\frac{x-x_0}{\lambda_0}\rg)\rg)
\end{multline*}
in $\hdot\times L^2$.
\end{lemma}
Note that from Lemma \ref{L:positive_energy}, the energy of $u$ is $>0$, which justifies the definition of $\ell$.
\begin{proof}
As usual, we may assume that $x(t)$ and $\lambda(t)$ are continuous functions of $t$.

\EMPH{Step 1}
We show that 
\begin{equation}
 \label{lambda0}
\lim_{t\to +\infty}\frac{\lambda(t)}{t}=0.
\end{equation}
The proof is standard (see \cite{KeMe08}). We argue by contradiction.
By finite speed of propagation, $\lambda(t)/t$ is bounded for $t\geq 1$. If \eqref{lambda0} does not hold, then there exists a sequence $t_n\to +\infty$ and a $\tau_0 \in (0,+\infty)$ such that 
\begin{equation}
\label{lambda_pas0}
 \lim_{n\to\infty}\frac{\lambda(t_n)}{t_n}=\frac{1}{\tau_0}.
\end{equation}
Let 
$$w_n(s,y)=\lambda(t_n)^{\frac{N-2}{2}}u\lf(t_n+\lambda(t_n)s,\lambda(t_n)y+x(t_n)\rg).$$
Then after extraction there exists $(w_0,w_1)\in \hdot\times L^2$ such that
$$ \lim_{n\to \infty} (w_n(0),\partial_t w_n(0))=(w_0,w_1)\text{ in }\hdot\times L^2.$$
Let $w$ be the solution with initial data $(w_0,w_1)$. Let us show that $w$ is globally defined. For this we check that $w$ is compact up to modulation. For $s\in (T_-(w),T_+(w))$, let
\begin{align*}
u_{0n}(y)&=\lambda\Big(t_n+\lambda(t_n)s\Big)^{\frac{N-2}{2}}u\Big[t_n+\lambda(t_n)s,\lambda\Big(t_n+\lambda(t_n)s\Big)y+x\Big(t_n+\lambda(t_n)s\Big)\Big]\\ 
u_{1n}(y)&=\lambda\Big(t_n+\lambda(t_n)s\Big)^{\frac{N}{2}}\partial_t u\Big[t_n+\lambda(t_n)s,\lambda\Big(t_n+\lambda(t_n)s\Big)y+x\Big(t_n+\lambda(t_n)s\Big)\Big].
\end{align*}
Then by the definition of $K$, $(u_{0n},u_{1n})\in K$. Thus after extraction, $(u_{0n},u_{1n})$ has a limit as $n\to\infty$ which is in $\overline{K}$ (and thus, by energy conservation, not identically $0$). Next note that
\begin{align*}
u_{0n}(y)&=\left[\frac{\lambda\big(t_n+\lambda(t_n)s\big)}{\lambda(t_n)}\right]^{\frac{N-2}{2}}w_n\lf[s,\frac{\lambda\big(t_n+\lambda(t_n)s\big)}{\lambda(t_n)}y+\frac{x\big(t_n+\lambda(t_n)s\big)-x(t_n)}{\lambda(t_n)}\rg]\\ 
u_{1n}(y)&=\left[\frac{\lambda\big(t_n+\lambda(t_n)s\big)}{\lambda(t_n)}\right]^{\frac{N}{2}}\partial_s w_n\lf[s,\frac{\lambda\big(t_n+\lambda(t_n)s\big)}{\lambda(t_n)}y+\frac{x\big(t_n+\lambda(t_n)s\big)-x(t_n)}{\lambda(t_n)}\rg].
\end{align*}
Using that by continuity of the flow
$$ \lim_{n\to \infty}\lf(w_n(s),\partial_tw_n(s)\rg)=\lf(w(s),\partial_t w(s)\rg)\neq 0\text{ in }\hdot\times L^2,$$
we get that there exists $C(s)>0$ such that for all $n$,
\begin{equation*}
 \frac{1}{C(s)}\leq \frac{\lambda\big(t_n+\lambda(t_n)s\big)}{\lambda(t_n)}\leq C(s),\quad \lf|\frac{x\big(t_n+\lambda(t_n)s\big)-x(t_n)}{\lambda(t_n)}\rg|\leq C(s).
\end{equation*}
After extraction of a subsequence, this two quantities converge to $\tilde{\lambda}(s)$, $\tilde{x}(s)$. As a consequence, we get that
$$\left(\tilde{\lambda}(s)^{\frac{N-2}{2}}w\lf(s,\tilde{\lambda}(s)y+\tilde{x}(s)\rg),\tilde{\lambda}(s)^{\frac{N}{2}}\partial_sw\lf(s,\tilde{\lambda}(s)y+\tilde{x}(s)\rg)\rg)\in \overline{K}.$$
In particular, $w$ is compact up to modulation and satisfies the bound \eqref{bound_nabla2W}. By Proposition \ref{P:global}, $w$ is globally defined.

Let $s_n=-t_n/\lambda(t_n)$. Then
$$ \lf(w_n(s_n,y),\partial_t w_n(s_n,y)\rg)=\Big(\lambda(t_n)^{\frac{N}{2}-1}u\big(0,\lambda(t_n)y+x(t_n)\big),\lambda(t_n)^{\frac{N}{2}}\partial_t u\big(0,\lambda(t_n)y+x(t_n)\big)\Big),$$
and by \eqref{lambda_pas0}
$$ \lim_{n\to \infty} \big(w_n(s_n,y),\partial_t w_n(s_n,y)\big)=\big(w(-\tau_0,y),\partial_t w(-\tau_0,y)\big)\text{ in }\hdot\times L^2.$$
This shows that $\lambda(t_n)$ is bounded, a contradiction with \eqref{lambda_pas0}. Step 1 is complete.

\EMPH{Step 2}
By finite speed of propagation, there exists a constant $M>0$ such that 
\begin{equation}
\label{finite_speedx}
 \forall t\geq 0,\quad |x(t)|\leq M+|t|.
\end{equation}
In this step we show
\begin{equation}
\label{local_t}
 \lim_{t\to +\infty}\frac{\lf|x(t)-t\ell \vec{e}_1\rg|}{t}=0.
\end{equation}
Fix $\eps>0$.
Let $r(u)$ be as in \eqref{defru}. Let $\delta_{\eps}>0$ be such that
\begin{equation}
\label{bound_ru}
\forall t,\quad \int_{|x-x(t)|\geq\frac{\lambda(t)}{\delta_{\eps}}}r(u)\leq \eps.
\end{equation}
In view of step 1, \eqref{finite_speedx} and the continuity of $x(t)$ and $\lambda(t)$, there exists $t_0\gg 1$ such that for $\tau\geq t_0$,
\begin{equation}
\label{small_lambda_x}
\sup_{t\in[0,\tau]}\lambda(t)\leq \eps \delta_{\eps}\tau,\quad \sup_{t\in [0,\tau]}|x(t)|\leq \frac{3}{2}\tau. 
\end{equation} 
Let $\tau \geq t_0$ and, for $t\in [0,\tau]$,
$$ \Psi_{\tau}(t)=\int x\varphi\left(\frac{x}{\tau}\right)e(u)(t,x)\,dx,$$
where $\varphi(x)=1$ for $|x|\leq 3$, $\varphi(x)=0$ for $|x|\geq 4$. Then by \eqref{identity4} in Claim \ref{C:identities}, 
\begin{equation}
\label{dif_Y}
 \Psi'_{\tau}(t)=-\int\nabla u\partial_t u+\OOO\lf(\int_{|x|\geq 3\tau} r(u)\rg)=\ell E(u_0,u_1)\vec{e}_1+\OOO\lf(\int_{|x|\geq 3\tau} r(u)\rg),
\end{equation}
where $r(u)$ is defined by \eqref{defru}.
If $t\in[0,\tau]$, then by \eqref{small_lambda_x} (and using that $\eps\leq \frac 32$),
\begin{equation*}
 |x|\geq 3\tau\Longrightarrow \frac{|x-x(t)|}{\lambda(t)}\geq \frac{3\tau-|x(t)|}{\lambda(t)}\geq \frac{3\tau-\frac{3}{2}\tau}{\eps\,\delta_{\eps}\,\tau}\geq \frac{1}{\delta_{\eps}},
\end{equation*}
and thus by \eqref{bound_ru},
\begin{equation*}
 t\in [0,\tau]\Longrightarrow \left|\int_{|x|\geq 3\tau} r(u)(t,x)\,dx\right|\leq \eps.
\end{equation*}
Integrating \eqref{dif_Y}, we get
\begin{equation}
\label{Y_tau}
\left|\Psi_{\tau}(\tau)-\Psi_{\tau}(0)-\tau\ell E(u_0,u_1)\vec{e}_1\right|\leq C\tau\eps.
\end{equation}
Furthermore,
\begin{multline}
\label{dev_Y_tau}
 \Psi_{\tau}(\tau)-x(\tau)E(u_0,u_1)=\int\lf(x\varphi\lf(\frac{x}{\tau}\rg)-x(\tau)\rg)e(u)\\
=\int_{\left|x-x(\tau)\right|\leq \frac{\lambda(\tau)}{\delta_{\eps}}} \left(x\varphi\lf(\frac{x}{\tau}\rg)-x(\tau)\right)e(u)- x(\tau)\int_{\left|x-x(\tau)\right|\geq \frac{\lambda(\tau)}{\delta_{\eps}}}e(u)+\int_{\left|x-x(\tau)\right|\geq \frac{\lambda(\tau)}{\delta_{\eps}} }x\varphi\lf(\frac{x}{\tau}\rg)e(u).
\end{multline}
Notice that $\lf|x\varphi\lf(x/\tau\rg)\rg|\leq 4\tau$. By \eqref{bound_ru}, we bound the third integral in the second line of \eqref{dev_Y_tau} as follows
\begin{equation*}
\lf|\int_{\left|x-x(\tau)\right|\geq \frac{\lambda(\tau)}{\delta_{\eps}} }x\varphi\lf(\frac{x}{\tau}\rg)e(u)\rg|\leq C\eps\tau.
\end{equation*}
By \eqref{bound_ru} and \eqref{small_lambda_x}, the second integral can be estimated by $C\eps\tau$.  To bound the first integral in the second line of \eqref{dev_Y_tau}, write
$$ \left|x-x(\tau)\right|\leq \frac{\lambda(\tau)}{\delta_{\eps}}\Longrightarrow |x|\leq |x(\tau)|+\frac{\lambda(\tau)}{\delta_{\eps}}\leq \frac{5}{2}\tau,$$
and thus on the support of the first integral, $\varphi(x/\tau)=1$. As a consequence, by \eqref{small_lambda_x},
$$ \lf|\int_{\left|x-x(\tau)\right|\leq \frac{\lambda(\tau)}{\delta_{\eps}}} \left(x\varphi\lf(\frac{x}{\tau}\rg)-x(\tau)\right)e(u)\rg|\leq C\frac{\lambda(\tau)}{\delta_{\eps}}\leq C\,\eps\,\tau.$$
Combining the estimates, we get, in view of \eqref{Y_tau},
$$ \frac{1}{\tau}\left|x(\tau)-\tau\ell \vec{e}_1\right|E(u_0,u_1)\leq C\eps+\frac{1}{\tau}|\Psi_{\tau}(0)|,$$
and \eqref{local_t} follows, using that by dominated convergence, 
\begin{equation*}
 \lim_{\tau\to +\infty}\frac{1}{\tau}|\Psi_{\tau}(0)|=0.
\end{equation*} 
\EMPH{Step 3} In this step we show
\begin{equation}
\label{CV_in_mean}
 \lim_{T\to \infty}\frac{1}{T}\int_0^T \int (\partial_t u+\ell\partial_{x_1}u)^2\,dx\,dt=0.
\end{equation} 
Let $R>0$ be a parameter and define 
\begin{multline}
\label{def_ZR}
Z_R(t)=(\ell^2-1)\int (x-t\ell e_1)\cdot \nabla u\partial_tu\,\varphi\lf(\frac{x-t\ell \vec{e}_1}{R}\rg)\\
+\frac{N-2}{2}(\ell^2-1)\int u\partial_t u\,\varphi\lf(\frac{x-t\ell \vec{e}_1}{R}\rg)-\ell^2\int (x_1-t\ell)\cdot \partial_{x_1}u\partial_tu\,\varphi\lf(\frac{x-t\ell \vec{e}_1}{R}\rg),
\end{multline}
where $\varphi\in C_0^{\infty}$, $\varphi(x)=1$ for $|x|\leq 3$, $\varphi(x)=0$ for $|x|\geq 4$.
From \eqref{identity1}, \eqref{identity2} and \eqref{identity3} in Claim \ref{C:identities}, and using that $\int \nabla u\partial_tu=-\ell E(u_0,u_1)\vec{e}_1$, we get
\begin{equation}
\label{dif_ZR}
\left|Z'_R(t)-\int (\partial_tu+\ell\partial_{x_1}u)^2\rg|\leq C\int_{\left|x-t\ell \vec{e}_1\right|\geq 3R} r(u).
\end{equation}
Let $\eps>0$. As in the preceding step, choose $\delta_{\eps}$ such that \eqref{bound_ru} holds. 
In view of steps 1 and 2, and the continuity of $\lambda$ and $x$, there exists $t_0=t_{0}(\eps)\gg 1$ such that for $T\geq t_0$, 
\begin{equation}
\label{local_lambda_x}
\sup_{t\in [0,T]} \lambda(t)\leq \eps\delta_{\eps}T,\quad \sup_{t\in [0,T]}\lf|x(t)-t\ell \vec{e}_1\rg|\leq \eps\,T.
\end{equation} 
Take
\begin{equation*}
 T\geq t_0(\eps),\quad R=\eps \,T.
\end{equation*}
Then
\begin{equation*}
\frac{\left|x- t\ell \vec{e}_1\right|}{R}\leq \frac{|x-x(t)|}{\eps T}+\frac{|t\ell \vec{e}_1-x(t)|}{\eps T}\leq 1+\frac{\delta_{\eps}|x-x(t)|}{\lambda(t)}.
\end{equation*}
In particular $\frac{\left|x- t\ell \vec{e}_1\right|}{R}\geq 3\Longrightarrow \frac{|x-x(t)|}{\lambda(t)}\geq \frac{2}{\delta_{\eps}}$,
and thus
\begin{equation}
 \label{other_bound_r}
\int_{|x-t\ell \vec{e}_1|\geq 3R}r(u)\leq \eps.
\end{equation} 
Integrating \eqref{dif_ZR} between $t=0$ and $t=T$, we get, for $T\geq t_0$, $R=\eps T$,
\begin{equation*}
 \frac{1}{T}\int_0^T \int (\partial_t u+\ell\partial_{x_1}u)^2\,dx\,dt\leq \frac{1}{T}\left(|Z_R(T)|+|Z_{R}(0)|\right)+C\eps.
\end{equation*} 
Using that $|Z_R(t)|\leq C R$ for all $t$, we get the bound $|Z_R(0)|+|Z_R(T)|\leq C\,T\,\eps$, hence
$$ \limsup_{T\to +\infty}\frac{1}{T}\int_0^T \int (\partial_t u+\ell\partial_{x_1}u)^2\,dx\,dt\leq C\eps,$$
which gives \eqref{CV_in_mean}.

\EMPH{Step 4. End of the proof}
As in \cite[Proof of Corollary 5.3]{DuKeMe09P}, we deduce from Step 3 that there exists a sequence $\{t_n\}$ such that  $t_n\to +\infty$ and
\begin{align*}
\lim_{n\to\infty}\lambda(t_n)^{\frac{N-2}{2}}u\left(t_n,\lambda(t_n)x+x(t_n)\right)&= U_0\text{ in }\hdot\\
\lim_{n\to\infty}\lambda(t_n)^{\frac{N}{2}}\partial_t u\left(t_n,\lambda(t_n)x+x(t_n)\right)&= U_1\text{ in }L^2,
\end{align*}
where the solution $U$ with initial condition $(U_0,U_1)$ satisfies, for some small $\tau_0\in (0,T_+(U))$,
$$ \partial_t U+\ell \partial_{x_1}U=0\text{ for }t\in [0,\tau_0].$$
As in the proof of Lemma \ref{L:ene_mom}, we deduce from Lemma \ref{L:asym_ell} that $\ell^2<1$ and $(U_0,U_1)=\pm(W_{\ell}(0),\partial_t W_{\ell}(0)$ up to space rotation, space translation and scaling.
\end{proof}

\subsection{End of the proof}
\label{SS:endofproof}
Let $u$ be as in Theorem \ref{T:compact}. By a standard argument, we can assume that the parameters $\lambda(t)$ and $x(t)$ defining $K$ as in \eqref{def_K} are, for $|d_{\ell}(t)|<\delta_0$ ($\delta_0$ given by Lemma \ref{L:modulation}), the modulation parameters given by Lemma \ref{L:modulation}, and that $x(t)$ and $\lambda(t)$ are continuous functions of $t$.

By Lemma \ref{L:subseq} applied to $u$ and $t\mapsto u(-t)$, there exist sequences $t_n\to +\infty$, $t_n'\to -\infty$ such that
\begin{equation}
\label{dln0}
\lim_{n\to \infty} \lf|d_{\ell}(t_n)\rg|+\lf|d_{\ell}(t_n')\rg|=0,
\end{equation}
where $d_{\ell}$ is defined by \eqref{def_dl}. We start by rescaling the solution between $t_n'$ and $t_n$. Let
\begin{equation*}
\label{def_lambdan}
\lambda_n=\max_{t\in [t'_n,t_n]}\lambda(t).
\end{equation*}
Let $T_n=\frac{t_n-t_n'}{\lambda_n}$, and for $\tau\in [0,T_n]$, $y\in \RR^N$, define $u_n(\tau,y)$ by 
\begin{equation*}
 u(t,x)=\frac{1}{\lambda_n^{\frac{N-2}{2}}}u_n\left(\frac{t-t'_n}{\lambda_n},\frac{x-x(t_n')}{\lambda_n}\right),\quad t\in [t_n',t_n].
\end{equation*}
Then
\begin{equation}
\label{inK}
 \forall \tau \in [0,T_n],\quad \left(\big(\mu_n(\tau)\big)^{\frac{N-2}{2}}u_n\lf(\tau,\mu_n(\tau)y+y_n(\tau)\rg),\big(\mu_n(\tau)\big)^{\frac{N}{2}}\partial_{\tau}u_n\lf(\tau,\mu_n(\tau)y+y_n(\tau)\rg)\right)\in K 
\end{equation}
where by definition, for $\tau\in [0,T_n]$,
\begin{equation*}
 \mu_n(\tau)=\frac{\lambda(\lambda_n\tau+t_n')}{\lambda_n},\quad y_n(\tau)=\frac{x\lf(\lambda_n\tau+t_n'\rg)-x\lf(t_n'\rg)}{\lambda_n}.
\end{equation*}
Indeed, \eqref{inK} follows from 
$$ \lambda(t)^{\frac{N-2}{2}}u\left(t,\lambda(t)x+x(t)\rg)=\left(\frac{\lambda(t)}{\lambda_n}\rg)^{\frac{N-2}{2}}u_n\lf(\frac{t-t_n'}{\lambda_n},\frac{\lambda(t)x+x(t)-x(t_n')}{\lambda_n}\rg),$$
and the analoguous equality for the time derivative of $u$. Note that by the choice of $u_n$,
\begin{equation*}
y_n(0)=0\text{ and }\forall \tau\in [0,T_n],\quad 0<\mu_n(\tau)\leq 1.
\end{equation*}
Define
\begin{gather*}
 Y_n(\tau)=y_n(\tau)-\ell \tau\vec{e}_1,\\ d_n(\tau)=d_{\ell}(\lambda_n\tau+t_n')=\int \lf|\nabla u_n(\tau)\rg|^2+\int \lf(\partial_tu_n(\tau)\rg)^2-\int \lf|\nabla W_{\ell}(0)\rg|^2-\int \lf|\partial_t W_{\ell}(0)\rg|^2.
\end{gather*}
We claim:
\begin{lemma}[Parameter control]
 \label{L:parameter_control}
There exists a constant $C>0$ such that for all $n$, if $0\leq \sigma<\tau\leq T_n$,  then
\begin{enumerate}
 \item \label{I:close} If $|\tau-\sigma|\leq 2\mu_n(\tau)$, then
\begin{equation*}
\frac{1}{C}\leq \lf|\frac{\mu_n(\tau)}{\mu_n(\sigma)}\rg|\leq C,\quad \left|Y_n(\tau)-Y_n(\sigma)\rg|\leq C\mu_n(\tau).
\end{equation*}
\item \label{I:far}If $|\tau-\sigma|\geq \mu_n(\tau)$, then
\begin{equation*}
 \lf|\mu_n(\sigma)-\mu_n(\tau)\rg|+\lf|Y_n(\sigma)-Y_n(\tau)\rg|\leq C\int_{\sigma}^{\tau} |d_n(s)|ds.
\end{equation*}
\end{enumerate} 
\end{lemma}
\begin{lemma}[Virial-type estimate]
\label{L:virial}
 For all $n$,
\begin{equation*}
 \int_0^{T_n}|d_n(s)|\,ds\leq C\left(1+\max_{\tau\in [0,T_n]}|Y_n(\tau)|\right)\lf(|d_n(0)|+|d_n(T_n)|\rg)+C|Y_n(T_n)|.
\end{equation*}
\end{lemma}
\begin{lemma}[Large time control of the space translation]
 \label{L:space_control}
Let $\eps>0$. Then there exists a constant $C_{\eps}>0$ such that for all $n$,
\begin{equation*}
 |Y_n(T_n)|\leq \eps\int_0^{T_n} |d_n(\tau)|\,d\tau+C_{\eps}\left(1+\max_{\tau\in [0,T_n]} |Y_n(\tau)|\rg)\Big(\lf|d_n\lf(T_n\rg)\rg|+\lf|d_n(0)\rg|\Big).
\end{equation*}
\end{lemma}
\begin{proof}[Proof of Theorem \ref{T:compact}]
 Let us prove Theorem \ref{T:compact} assuming Lemmas \ref{L:parameter_control}, \ref{L:virial} and \ref{L:space_control}. We will use that by the choice of the sequences $\{t_n\}$ and $\{t_n'\}_n$,
\begin{equation}
\label{dn0}
 \lim_{n\to \infty}\big(|d_n(0)|+|d_n(T_n)|\big)=0.
\end{equation} 
Combining Lemma \ref{L:virial} and \ref{L:space_control} (with a small $\eps$), we get
\begin{equation}
 \label{new_virial}
\int_0^{T_n} |d_n(s)|\,ds\leq C\left(1+\max_{\tau\in [0,T_n]}|Y_n(\tau)|\right)\lf(|d_n(0)|+|d_n(T_n)|\rg).
\end{equation} 

\EMPH{Step 1. Uniform bound on the modulation parameters}
We first show that there exists a constant $C>0$ such that for all $n$,
\begin{equation*}
\max_{\tau\in [0,T_n]}|Y_n(\tau)|\leq C,\quad \min_{\tau\in [0,T_n]}\mu_n(\tau) \geq \frac{1}{C}. 
\end{equation*} 
By continuity of $Y_n$, there exist $\theta_n\in [0,T_n]$ such that $|Y_n(\theta_n)|=\max_{\tau\in [0,T_n]} |Y_n(\tau)|$. If $\theta_n\leq \mu_n(0)$, then by \eqref{I:close} in Lemma \ref{L:parameter_control}, 
$$|Y_n(\theta_n)|=|Y_n(\theta_n)-Y_n(0)|\leq C\mu_n(0)\leq C.$$
If $\theta_n\geq \mu_n(0)$ then 
combining \eqref{new_virial} with Lemma \ref{L:parameter_control} \eqref{I:far}, we get
$$ |Y_n(\theta_n)|=|Y_n(\theta_n)-Y_n(0)|\leq C\int_0^{\theta_n}|d_n(s)|ds\leq C(1+|Y_n(\theta_n)|)\lf(|d_n(0)|+|d_n(T_n)|\rg),$$
and the boundedness of $|Y_n(\theta_n)|$ follows from \eqref{dn0}.

Similarly, let $\theta_n',\theta_n''\in [0,T_n]$ be such that 
$$\mu_n(\theta_n')=\min_{\tau\in [0,T_n]} \mu_n(\tau),\quad \mu_n(\theta_n'')=\max_{\tau\in [0,T_n]} \mu_n(\tau)=1.$$
Then if $|\theta_n'-\theta_n''|\leq \mu(\theta_n'')=1$ we get immediately by Lemma \ref{L:parameter_control}, \eqref{I:close} that $\mu(\theta_n')\geq\frac{1}{C}$. On the other hand, if $|\theta_n'-\theta_n''|\geq 1$, we obtain, combining \eqref{new_virial} with Lemma \ref{L:parameter_control}, \eqref{I:far} and the uniform boundedness of $Y_n$,
$$ \lf|\mu_n(\theta_n')-1\rg|\leq C\lf(|d_n(0)|+|d_n(T_n)|\rg),$$
and the fact that $\mu_n(\theta_n')$ is bounded from below by a positive constant follows again from \eqref{dn0}.

\EMPH{Step 2. End of the proof}
From Step 1 and \eqref{new_virial}, 
\begin{equation}
 \label{new_new_virial}
\int_0^{T_n} |d_n(s)|\,ds\leq C\lf(|d_n(0)|+|d_n(T_n)|\rg).
\end{equation} 
To conclude the proof, we will show that
\begin{equation}
 \label{dn00}
\lim_{n\to \infty} \max_{\tau\in [0,T_n]} |d_n(\tau)|=0.
\end{equation} 
This would imply that
$$ d_{\ell}(0)=d_n\left(\frac{-t_n'}{\lambda_n}\rg)\underset{n\to \infty}{\longrightarrow} 0,$$
and thus that $d_{\ell}(0)=0$, and Theorem \ref{T:compact} would follow from the first point of Claim \ref{C:trapping}.

To show \eqref{dn00}, we argue by contradiction. By the continuity of the flow of \eqref{CP} in $\hdot\times L^2$, $d_n(\tau)$ is a continuous function of $\tau$. If \eqref{dn00} does not hold, there exists $\eps_0\in [0,\delta_0/2]$ and, for large $n$, $\tau_n\in [0,T_n]$ such that
\begin{equation}
\label{absurd_taun}
\tau\in [0,\tau_n)\Rightarrow |d_n(\tau)|<\eps_0,\text{ and }|d_n(\tau_n)|=\eps_0.
\end{equation}
Recall the modulation parameter $\alpha$ defined in Lemma \ref{L:modulation}. Let
$$ \alpha_n(\tau)=\alpha(\lambda_n\tau+t_n')$$
be the corresponding parameter for the solution $u_n$. Using the modulation estimate of Lemma \ref{L:modulation} and Step 1, we get
$$ \forall \tau\in [0,\tau_n],\quad |\alpha_n'(\tau)|\leq C\frac{|d_n(\tau)|}{\mu_n(\tau)}\leq C|d_n(\tau)|.$$
Integrating between $0$ and $\tau_n$, we get
$$ |\alpha_n(0)-\alpha_n(\tau_n)|\leq C\int_0^{\tau_n}|d_n(\tau)|\,d\tau\leq C\int_0^{T_n}|d_n(\tau)|\,d\tau\leq C\left(|d_n(0)|+|d_n(T_n)|\rg).$$
By \eqref{dn0},
$$ \lim_{n\to \infty} \lf|\alpha_n(0)-\alpha_n(\tau_n)\rg|=0,$$
contradicting \eqref{absurd_taun} since by Lemma \ref{L:modulation} $|\alpha_n(\tau)|\approx |d_n(\tau)|$, and $d_{n}(0)\to 0$ as $n\to \infty$. The proof of \eqref{dn00} is complete, concluding the proof of Theorem \ref{T:compact}.
\end{proof}
It remains to show Lemmas \ref{L:parameter_control}, \ref{L:virial}, \ref{L:space_control}.

\begin{proof}[Proof of Lemma \ref{L:space_control}]
Fix $\eps>0$, and let
\begin{equation}
\label{def_Rn}
 R_n=C_{\eps}\left(1+\max_{\tau\in [0,T_n]} |Y_n(\tau)|\right),
\end{equation}
for some $C_{\eps}>0$ to be chosen.
Let
\begin{equation}
\label{def_Psin}
\Psi_n(\tau)=\Psi\Big[R_n,u_n(\tau),\partial_\tau u_n(\tau),\tau\Big]=\int_{\RR^N} (y-\tau \ell \vec{e}_{1})e(u_n)(\tau)\,\varphi\left(\frac{y-\tau\ell \vec{e}_1}{R_n}\right),
\end{equation} 
where the smooth function $\varphi$ satisfies $\varphi(x)=1$ if $|x|\leq 1$ and $\varphi(x)=0$ if $|x|\geq 2$.

\EMPH{Step 1}
Let $v$ be any solution of \eqref{CP} such that 
$$E(v_0,v_1)=E(W_{\ell}(0),\partial_t W_{\ell}(0)\text{ and } \int \nabla v_0\,v_1=\int \nabla W_{\ell}(0)\,\partial_tW_{\ell}(0).$$
To simplify notations, denote $\partial_0=\partial_t$, and $\partial_j=\partial_{x_j}$ if $j=1\ldots N$.
Then, fixing $R>0$, we have
\begin{equation}
\label{deriv_Phi}
 \frac{d}{dt} \Psi\Big[R,u(t),\partial_t v(t),t\Big]=A\Big[R,v(t),\partial_t v(t),t\Big],
\end{equation} 
where $A\big[R,v(t),\partial_t v(t),t\big]$ is of the form
\begin{equation}
\label{form_A}
A\Big[R,v(t),\partial_t v(t),t\Big]=\sum_{0\leq i,j\leq N} \int \partial_iv\partial_jv\,\psi_{ij}\left(\frac{x-t\ell\vec{e}_1}{R}\right)\,dx+\sum_{0\leq i\leq N} \int \frac{1}{|x|}v\partial_jv\,\psi_{j}\left(\frac{x-t\ell\vec{e}_1}{R}\right)\,dx, 
\end{equation} 
and the smooth functions $\psi_{ij}$ and $\psi_j$ are supported in $|x|\geq 1$. The equality \eqref{deriv_Phi} follows from explicit computation and \eqref{identity4} in Claim \ref{C:identities}.

\EMPH{Step 2} We fix $R>0$, $\Lambda>0$ and $X\in \RR^N$. Then 
$$  \Psi\left[R,\frac{1}{\Lambda^{\frac{N-2}{2}}}W_{\ell}\left(\frac{\tau}{\Lambda},\frac{y-X}{\Lambda}\right),\frac{1}{\Lambda^{\frac{N}{2}}}\partial_t W_{\ell}\left(\frac{\tau}{\Lambda},\frac{y-X}{\Lambda}\right),\tau\right]$$
is independent of $\tau$. Indeed
$$\frac{1}{\Lambda^{\frac{N-2}{2}}}W_{\ell}\left(\frac{\tau}{\Lambda},\frac{y-X}{\Lambda}\right)=\frac{1}{\Lambda^{\frac{N-2}{2}}}W_{\ell}\left(0,\frac{y-X-\tau\ell\vec{e}_1}{\Lambda}\right),$$
and the statement follows from the definition of $\Psi$. For example, the gradient term in the definition of $\Psi$ gives:
\begin{multline*}
 \frac{1}{2\Lambda^N}\int_{\RR^N} \left(y-\tau\ell \vec{e}_1\right)\left|\nabla W_{\ell}\left(0,\frac{y-X-\tau\ell \vec{e}_1}{\Lambda}\right)\right|^2\varphi\left(\frac{y-\tau\ell \vec{e}_1}{R}\right)\\
=
 \frac{1}{2\Lambda^N}\int_{\RR^N} z\left|\nabla W_{\ell}\left(0,\frac{z-X}{\Lambda}\right)\right|^2\varphi\left(\frac{z}{R}\right),
\end{multline*}
which is independent of $\tau$.

Combining this with Step 1, we get
\begin{equation*}
\forall R>0,\; \forall \Lambda>0,\; \forall X\in \RR^N,\; \forall \tau,\quad A\left[R,\frac{1}{\Lambda^{\frac{N-2}{2}}}W_{\ell}\left(\frac{\tau}{\Lambda},\frac{y-X}{\Lambda}\right),\frac{1}{\Lambda^{\frac{N}{2}}}\partial_{\tau} W_{\ell}\left(\frac{\tau}{\Lambda},\frac{y-X}{\Lambda}\right),\tau\right]=0.
\end{equation*} 
As a consequence, replacing $X$ by$X-\tau\ell \vec{e}_1$ in the preceding equality, we get by the definition of $W_{\ell}$:
\begin{multline}
 \label{cestnul}
\forall R>0,\; \forall \Lambda>0,\; \forall X\in \RR^N,\; \forall \tau,\\ A\left[R,\frac{1}{\Lambda^{\frac{N-2}{2}}}W_{\ell}\left(0,\frac{y-X}{\Lambda}\right),\frac{1}{\Lambda^{\frac{N}{2}}}\partial_{\tau} W_{\ell}\left(0,\frac{y-X}{\Lambda}\right),\tau\right]=0.
\end{multline} 
\EMPH{Step 3. Bounds on $\Psi_n(0)$ and $\Psi_n(T_n)$}
In this step we show that if $C_{\eps}$ is chosen large, then for large $n$,
\begin{align}
 \label{bound_Psin_0}
|\Psi_{n}(0)|&\leq C R_n|d_n(0)|\\
\label{bound_Psin_Tn}
|\Psi_{n}(T_n)-Y_n(T_n)E(u_0,u_1)|&\leq C R_n|d_n(T_n)|+\eps|Y_n(T_n)|.
\end{align} 
Fix  $\tau\in \{0,T_n\}$. Then if $n$ is large, $|d_{n}(\tau)|<\delta_0$. By Lemma \ref{L:modulation}, one can write (for some sign $\pm$),
\begin{align}
\label{u=W+eps1'}
\pm u_n(\tau,y)&=\frac{1}{\mu_n(\tau)^{\frac{N-2}{2}}}W_{\ell}\left(0,\frac{y-Y_n(\tau)-\tau\ell \vec{e}_1)}{\mu_n(\tau)}\right)+\frac{1}{\mu_n(\tau)^{\frac{N-2}{2}}}\eps_n\left(\tau,\frac{y-Y_n(\tau)-\tau\ell\vec{e}_1}{\mu_n(\tau)}\right)\\
\label{u=W+eps2'}
\pm \partial_t u_n(\tau,y)&=\frac{1}{\mu_n(\tau)^{\frac{N}{2}}}\partial_t W_{\ell}\left(0,\frac{y-Y_n(\tau)-\tau\ell\vec{e}_1}{\mu_n(\tau)}\right)+\frac{1}{\mu_n(\tau)^{\frac{N}{2}}}\eps_{1,n}\left(t,\frac{y-Y_n(\tau)-\tau\ell\vec{e}_1}{\mu_n(\tau)}\right),
\end{align}
where
\begin{equation}
\label{bound_eps'}
\|\eps_n(\tau)\|_{\hdot}+\|\eps_{1,n}(\tau)\|_{L^2}\leq C|d_n(\tau)|.
\end{equation} 
Expanding the expression \eqref{def_Psin} of $\Psi_n(\tau)$, we get  by \eqref{bound_eps'}, the facts that $|y-\tau\ell\vec{e}_1|\leq R_n$ on the domain of integration and that by the definition of $R_n$, $|Y_n(\tau)|\leq R_n$, 
\begin{multline}
\label{bound_Psin}
\left|
 \Psi_n(\tau)-\Psi\left[R_n,\frac{1}{\mu_n^{\frac{N-2}{2}}}W_{\ell}\lf(0,\frac{y-Y_n(\tau)-\tau\ell \vec{e}_1}{\mu_n}\rg),\frac{1}{\mu_n^{\frac{N}{2}}}\partial_t W_{\ell}\lf(0,\frac{y-Y_n(\tau)-\tau\ell \vec{e}_1}{\mu_n}\rg),\tau\right]\right|\\
 \leq C\left(R_n|d_n(\tau)|+R_n|d_n(\tau)|^2\right).
\end{multline} 
Recall that $Y_n(0)=0$. By the definition of $\Psi_n$ and the parity of $W_{\ell}$ we obtain
$$\Psi\left[R_n,\frac{1}{\mu_n(0)^{\frac{N-2}{2}}}W_{\ell}\lf(0,\frac{y}{\mu_n(0)}\rg),\frac{1}{\mu_n(0)^{\frac{N}{2}}}\partial_t W_{\ell}\lf(0,\frac{y}{\mu_n(0)}\rg),0\right]=0.$$
Hence \eqref{bound_Psin_0} follows. To show \eqref{bound_Psin_Tn}, we must estimate
\begin{multline}
\label{big_ineg}
\Psi\left[R_n,\frac{1}{\mu_n(T_n)^{\frac{N-2}{2}}}W_{\ell}\lf(0,\frac{y-Y_n(T_n)-T_n\ell \vec{e}_1}{\mu_n(T_n)}\rg),\frac{1}{\mu_n(T_n)^{\frac{N}{2}}}\partial_t W_{\ell}\lf(0,\frac{y-Y_n(T_n)-T_n\ell \vec{e}_1}{\mu_n(T_n)}\rg),T_n\right]\\=
 \int (z+Y_n(T_n))e\left(\frac{1}{\mu_n^{\frac{N-2}{2}}}W_{\ell}\left(0,\frac{z}{\mu_n}\right)\rg)\varphi\left(\frac{z+Y_n(T_n)}{R_n}\right)dz\\
=Y_n(T_n)E(u_0,u_1)+Y_n(T_n)\int e\left(\frac{1}{\mu_n^{\frac{N-2}{2}}}W_{\ell}\left(0,\frac{z}{\mu_n}\right)\rg)\left(\varphi\left(\frac{z+Y_n(T_n)}{R_n}\right)-1\right)dz\\
+
\int z\,e\left(\frac{1}{\mu_n^{\frac{N-2}{2}}}W_{\ell}\left(0,\frac{z}{\mu_n}\right)\rg)\left(\varphi\left(\frac{z+Y_n(T_n)}{R_n}\right)-\varphi\left(\frac{z}{R_n}\right)\right)dz,\\
=Y_n(T_n)E(u_0,u_1)+(I)+(II),
\end{multline}
where in the last line we have used that by the parity of $W_{\ell}$, 
$$\int z\,e\left(\frac{1}{\mu_n^{\frac{N-2}{2}}}W_{\ell}\left(0,\frac{z}{\mu_n}\right)\rg)\varphi\left(\frac{z}{R_n}\right)dz=0.$$
By the definition of $R_n$ (taking $C_{\eps}\geq 2$), $|z+Y_n(T_n)|\geq R_n\Longrightarrow |z|\geq R_n/2\geq C_{\eps}/2$. Chosing $C_{\eps}$ large so that for a large constant $C>0$,
\begin{equation}
\label{choice_Ceps}
 \int_{|x|\geq C_{\eps}} r(W_{\ell})(0,x)\,dx\leq \frac{\eps}{C},
\end{equation} 
(where $r$ is defined in \eqref{defru}), we get (using that $\mu_n(T_n)\leq 1$) that the term $(I)$ in \eqref{big_ineg} satisfies:
$$ |(I)|\leq \eps|Y_n(T_n)|.$$
By the mean value theorem, there exists $c\in [0,1]$ such that
$$(II)=\int z\,e\left(\frac{1}{\mu_n^{\frac{N-2}{2}}}W_{\ell}\left(0,\frac{z}{\mu_n}\right)\rg)\, \frac{Y_n(T_n)}{R_n}\cdot \nabla \varphi\left(\frac{z+c Y_n(T_n)}{R_n}\right)dz,$$
and we get, again by \eqref{choice_Ceps},
$$|(II)|\leq \eps|Y_n(T_n)|,$$
which concludes the proof of \eqref{bound_Psin_Tn}.

\EMPH{Step 4. Bound on the derivative of $\Psi_n$}
We show that for an appropriate choice of $C_{\eps}$,
\begin{equation}
\label{bound_Psin'}
 \forall\tau\in [0,T_n],\quad \left|\Psi_n'(\tau)\right|\leq \eps |d_n(\tau)|.
\end{equation} 
First assume $|d_n(\tau)|\geq \delta_0$. Then by the compactness of $K$ and the fact that $\mu_n\leq 1$, we get, if $C_{\eps}$ is large,
$$ \int_{|y-\tau\ell e_1|\geq R_n} r(u_n)\leq \frac{\eps}{C}.$$
Indeed, 
$$|y-\tau\ell \vec{e}_1|\geq R_n\Longrightarrow |y-y_n(\tau)|\geq \frac{R_n}{2}\geq\frac{C_{\eps}}{2}\Longrightarrow \frac{|y-y_n(\tau)|}{\mu_n(\tau)}\geq \frac{C_{\eps}}{2}.$$
The bound \eqref{bound_Psin'} follows, in this case, by the expression of the derivative of $\Psi$ obtained in Step 1.

We next assume $|d_n(\tau)|<\delta_0$. Write $u_n$ as in \eqref{u=W+eps1'}, \eqref{u=W+eps2'}. Expanding the expression \eqref{form_A} of $A(R_n,u,\partial_t u,\tau)$, we must bound, in view of \eqref{cestnul}, the following terms
\begin{gather}
\label{first_term}
\int_{|y-\tau\ell\vec{e}_1|\geq R_n} \frac{1}{\mu_n^N}  \sqrt{|\nabla\eps_n|^2+|\eps_{1,n}|^2}\lf(\tau,\frac{y-Y_n(\tau)-\tau\ell\vec{e}_1}{\mu_n}\right)\,\left|\nabla_{\tau,x}W_{\ell}\lf(0,\frac{y-Y_n(\tau)-\tau\ell\vec{e}_1}{\mu_n}\right)\right|\,dy\\
\label{second_term}
 \int_{|y-\tau\ell\vec{e}_1|\geq R_n} \frac{1}{\mu_n^N} \big(\left|\nabla\eps_n\right|^2+\eps_{1,n}^2\big)\lf(\tau,\frac{y-Y_n(\tau)-\tau\ell\vec{e}_1}{\mu_n}\right)\,dy.
\end{gather}
One can choose $C_{\eps}$ large so that (for a large constant $C>0$),
\begin{equation}
\label{bnd_cpct}
 \int_{|y|\geq C_{\eps}/2} |\nabla\eps_n(\tau)|^2+|\eps_{1,n}(\tau)|^2+\lf|\nabla_{\tau,y}W_{\ell}(0)\rg|^2\,dy\leq \frac{\eps^2}{C}.
\end{equation}
Indeed the set of all $\lf(\eps_n(\tau),\eps_{1,n}(\tau)\right)$ where $n\in \NN$ and $\tau\in [0,T_n]$ stays in a compact subset of $\hdot\times L^2$ as can be deduced from \eqref{inK}, \eqref{u=W+eps1'} and \eqref{u=W+eps2'}.

Using again that $|y-\tau\ell\vec{e}_1|\geq R_n\Longrightarrow \frac{|y-y_n(\tau)|}{\mu_n}\geq \frac{C_{\eps}}{2}$, we bound the terms \eqref{first_term} and \eqref{second_term} by $\eps |d_{n}(\tau)|$ by Cauchy-Schwarz inequality, the bound \eqref{bound_eps'} on $(\eps_n,\eps_{1,n})$ and \eqref{bnd_cpct}. Hence \eqref{bound_Psin'} follows.

\EMPH{Step 5. End of the proof}
By Step 4,
\begin{equation*}
 \left|\Psi_n(T_n)-\Psi_n(0)\right|\leq \eps\int_0^{T_n}|d_n(\tau)|\,d\tau.
\end{equation*}
Combining with Step 3, we get
\begin{equation*}
 |Y_n(T_n)|E(u_0,u_1)\leq CR_n\left(|d_n(0)|+|d_n(T_n)|\rg)+\eps|Y_n(T_n)|+\eps\int_0^{T_n}|d_n(\tau)|\,d\tau.
\end{equation*}
Using that $\eps$ is small and that $E(u_0,u_1)=E(W_{\ell}(0),\partial_tW_{\ell}(0))>0$ we get, by the definition of $R_n$,
\begin{equation*}
 |Y_n(T_n)|\frac{E(u_0,u_1)}{2}\leq C_{\eps}\left(1+\max_{\tau\in [0,T_n]}|Y_n(\tau)|\right)\left(|d_n(0)|+|d_n(T_n)|\rg)+\eps\int_0^{T_n}|d_n(\tau)|\,d\tau,
\end{equation*}
which concludes the proof of Lemma \ref{L:space_control}.
\end{proof}
\begin{proof}[Proof of Lemma \ref{L:virial}]
The proof is very close to the one of Lemma \ref{L:space_control}, and is also a variant of the proof of Lemma 3.8 of \cite{DuMe08}, and we only sketch it. We divide it in the same 5 steps as the proof of \ref{L:space_control}.
Let
$$ R_n=C_0\left(1+\max_{\tau\in [0,T_n]}|Y_n(\tau)|\right),$$
where the large constant $C_0>0$ is to be specified later. Define
\begin{multline*}
 \Phi_n(\tau)=\Phi\Big[R_n,u_n(\tau),\partial_{\tau}u_n(\tau),\tau\Big]\\
=(N-2)\int (y-\tau\ell\vec{e}_1)\nabla u_n\partial_{\tau}u_n\,\varphi\left(\frac{y-\tau\ell\vec{e}_1}{R_n}\right)dy+\frac{(N-2)(N-1)}{2}\int u_n\partial_{\tau}u_n\,\varphi\left(\frac{y-\tau\ell\vec{e}_1}{R_n}\right)dy.
\end{multline*}
\EMPH{Step 1}
By explicit computation (see \eqref{identity1} and \eqref{identity3} in Claim \ref{C:identities}), for any solution $v$ of \eqref{CP} such that $E(v_0,v_1)=E(W_{\ell}(0),\partial_tW_{\ell}(0))$ and $\int \nabla v_0\,v_1=\int \nabla W_{\ell}(0)\,\partial_tW_{\ell}(0)$ and for any $R$,
\begin{multline}
\label{Phi'}
 \frac{d}{dt}\Phi\Big[R,v(t),\partial_{t}v(t),t\Big]\\
=\int |\nabla v|^2+\int (\partial_tv)^2-\int |\nabla W_{\ell}(0)|^2-\int |\partial_t W_{\ell}(0)|^2+B\Big[R,v(t),\partial_{t}v(t),t\Big],
\end{multline}
where $B$ is of the same type \eqref{form_A} as the $A$ of the proof of Lemma \ref{L:space_control}.

\EMPH{Step 2}
As in step 2 of the proof of Lemma \ref{L:space_control}, we notice that for any $R>0$, $\Lambda>0$, $X\in \RR^N$, 
$$ \frac{d}{d\tau}\left( \Phi\left[R,\frac{1}{\Lambda^{\frac{N-2}{2}}}W_{\ell}\left(\frac{\tau}{\Lambda},\frac{y-X}{\Lambda}\right),\frac{1}{\Lambda^{\frac{N}{2}}}\partial_t W_{\ell}\left(\frac{\tau}{\Lambda},\frac{y-X}{\Lambda}\right),\tau\right]\right)=0$$
and deduce that
\begin{equation*}
 B\left[R,\frac{1}{\Lambda^{\frac{N-2}{2}}}W_{\ell}\left(0,\frac{y-X}{\Lambda}\right),\frac{1}{\Lambda^{\frac{N}{2}}}\partial_t W_{\ell}\left(0,\frac{y-X}{\Lambda}\right),\tau\right]=0.
\end{equation*}
\EMPH{Step 3. Bound on $\Phi_n(0)$ and $\Phi_n(T_n)$}
We show
\begin{equation}
\label{step3}
 |\Phi_n(0)|\leq CR_n|d_n(0)|,\quad |\Phi_n(T_n)|\leq CR_n|d_n(T_n)|+C|Y_n(T_n)|.
\end{equation}
Let $\tau\in \{0,T_n\}$. For large $n$, $|d_n(\tau)|<\delta_0$. By \eqref{u=W+eps1'}, \eqref{u=W+eps2'} and \eqref{bound_eps'}.
\begin{multline}
\label{interm_virial}
 \bigg|\int (y-\tau\ell \vec{e}_1)\nabla u_n\partial_{\tau}u_n\varphi\lf(\frac{y-\tau\ell \vec{e}_1}{R_n}\rg)\\
-\int \frac{y-\tau\ell \vec{e}_1}{\mu_n^N}\nabla W_{\ell}\lf(0,\frac{y-\tau\ell \vec{e}_1-Y_n(\tau)}{\mu_n}\rg) \partial_{\tau}W_{\ell}\lf(0,\frac{y-\tau\ell \vec{e}_1-Y_n(\tau)}{\mu_n}\rg)\varphi\lf(\frac{y-\tau\ell \vec{e}_1}{R_n}\rg)\,dy\bigg|\\
\leq CR_n|d_n(\tau)|.
\end{multline}
By the change of variable $z=y-\tau\ell\vec{e}_1-Y_n(\tau)$, we write the term in the second line of \eqref{interm_virial} as
\begin{multline*}
\int \frac{z}{\mu_n^N}\nabla W_{\ell}\lf(0,\frac{z}{\mu_n}\rg) \partial_{\tau}W_{\ell}\lf(0,\frac{z}{\mu_n}\rg)\varphi\lf(\frac{z+Y_n(\tau)}{R_n}\rg)\,dz\\
+
\int \frac{Y_n(\tau)}{\mu_n^N}\nabla W_{\ell}\lf(0,\frac{z}{\mu_n}\rg) \partial_{\tau}W_{\ell}\lf(0,\frac{z}{\mu_n}\rg)\varphi\lf(\frac{z+Y_n(\tau)}{R_n}\rg)\,dz=(I)+(II).
\end{multline*}
Clearly $|(II)|\leq |Y_n(\tau)|$ (in particular $(II)=0$ if $\tau=0$). Furthermore, using the parity of $W_{\ell}$ the mean value theorem, and the bound $|Y_n(\tau)|\leq R_n$, we obtain
\begin{multline*}
|(I)|=\lf|\int \frac{z}{\mu_n^N}\nabla W_{\ell}\lf(0,\frac{z}{\mu_n}\rg) \partial_{\tau}W_{\ell}\lf(0,\frac{z}{\mu_n}\rg)\lf(\varphi\lf(\frac{z+Y_n(\tau)}{R_n}\rg)-\varphi\lf(\frac{z}{R_n}\rg)\rg)\,dz\rg|\\
\leq \lf|\frac{Y_n(\tau)}{R_n}\rg|\int_{|z|\leq 4R_n}\frac{|z|}{\mu_n^N}\left|\nabla_{t,x}W_{\ell}\lf(0,\frac{z}{\mu_n}\rg)\rg|^2\,dz\leq C|Y_n(\tau)|,
\end{multline*}
which yields the estimates \eqref{step3} (recalling again that $Y_n(0)=0$).

\EMPH{Step 4. Bound on $\Phi_n'(\tau)$}
Let us show that if $C_0$ in the definition of $R_n$ is large,
\begin{equation}
 \label{step4}
\forall \tau\in [0,T_n],\quad \left|\Phi_n'(\tau)-d_n(\tau)\rg|\leq \frac{1}{4}|d_n(\tau)|.
\end{equation} 
It is sufficient to show
\begin{equation}
\label{step4bis}
\forall \tau\in [0,T_n],\quad\lf|B\Big[R_n,u_n(\tau),\partial_{\tau}u_n(\tau),\tau\Big]\right|\leq \frac{1}{4}|d_n(\tau)|.
\end{equation} 
Let $\tau\in [0,T_n]$. First assume that $|d_n(\tau)|\geq \delta_0$. Then by definition of $B$,
\begin{multline*}
 \lf|B\Big[R_n,u_n(\tau),\partial_{\tau}u_n(\tau),\tau\Big]\right|\\
\leq \int_{|y-\tau\ell\vec{e}_1|\geq R_n} \left|\nabla_{t,x}u_n\left(\tau,y\right)\right|^2\leq \int_{|y-\tau\ell \vec{e}_1-Y_n(\tau)|\geq \frac{C_0}{2}\mu_n(\tau)} \left|\nabla_{t,x}u_n\left(\tau,y\right)\right|^2,
\end{multline*}
where we used the inequalities $\mu_n(\tau)\leq 1$, $|Y_n(\tau)|\leq \frac{R_n}{2}$ and $C_0\leq R_n$.
From \eqref{inK} and the compactness of $\overline{K}$, we get that for $C_0$ large,
$$\lf|B\Big[R_n,u_n(\tau),\partial_{\tau}u_n(\tau),\tau\Big]\right|\leq \frac{\delta_0}{4}\leq \frac{|d_n(\tau)|}{4}.$$
We next treat the case $|d_n(\tau)|<\delta_0$. By \eqref{u=W+eps1'}, \eqref{u=W+eps2'}, \eqref{bound_eps'} and Step 2, we get that
 $\lf|B\big[R_n,u_n,\partial_{\tau}u_n,\tau\big]\right|$ is bounded (up to a multiplicative constant) by \eqref{first_term} and \eqref{second_term}, and the same argument as in Step 4 of the proof of Lemma \ref{L:space_control} gives \eqref{step4bis} if the constant $C_0$ in the definition of $R_n$ is large enough.

\EMPH{Step 5. End of the proof}
By Step 3 and 4,
$$ \int_0^{T_n}|d_n(\tau)|d\tau\leq CR_n\left(|d_n(0)|+|d_n(T_n)|\right)+C|Y_n(T_n)|,$$
which concludes the proof of Lemma \ref{L:virial} in view of the definition of $R_n$.
\end{proof}
\begin{proof}[Sketch of the proof of Lemma \ref{L:parameter_control}]
 The proof is very close to the proof of Lemma 3.10 in \cite{DuMe08}.

We first notice that the point \eqref{I:close} follows from \eqref{inK} and the compactness of $K$ (see Step 1 of the proof of 
\cite[Lemma 3.10]{DuMe08}).

We next show that there exists $\delta_1>0$ such that 
$$\forall n,\; \forall \tau\in [0,T_n], \;\forall \theta,\sigma\in [\tau-2\mu_n(\tau),\tau+2\mu_n(\tau)]\cap [0,T_n],\quad |d_n(\theta)|\geq \delta_0\Longrightarrow |d_n(\sigma)|\geq \delta_1.$$
If not, there exists a sequence $n_k$ of indexes (which might be stationary), and for each $k$, $\tau_k\in [0,T_{n_k}]$, $\theta_k,\sigma_k\in [\tau_k-2\mu_{n_k}(\tau_k),\tau_k+2\mu_{n_k}(\tau_k)]\cap [0,T_{n_k}]$ such that
\begin{equation}
\label{dnk}
 \lf|d_{n_k}(\theta_k)\rg|\geq \delta_0,\quad \lf|d_{n_k}(\sigma_k)\rg|\leq \frac 1k.
\end{equation}
After extraction of a subsequence, we can find $(U_0,U_1)\in \overline{K}$ such that in $\hdot\times L^2$,
\begin{multline*}
 \lim_{k\to \infty} \lf(\mu_{n_k}(\sigma_k)^{\frac{N-2}{2}} u_{n_k}\lf(\sigma_k,\mu_{n_k}(\sigma_k)y+y_{n_k}(\sigma_k)\rg),\mu_{n_k}(\sigma_k)^{\frac{N}{2}} \partial_{\tau} u_{n_k}\lf(\sigma_k,\mu_{n_k}(\sigma_k)y+y_{n_k}(\sigma_k)\rg)\rg)\\
=\lf(U_0,U_1\rg).
\end{multline*} 
By \eqref{dnk} and Claim \ref{C:trapping}, $(U_0,U_1)=(\pm W_{\ell}(0),\pm \partial_t W_{\ell}(0))$ up to scaling, space translation and rotation. Furthermore
$$ \theta_k=\sigma_k+\frac{\theta_k-\sigma_k}{\mu_{n_k}(\sigma_k)}\mu_{n_k}(\sigma_k).$$
As $\frac{\theta_k-\sigma_k}{\mu_{n_k}(\sigma_k)}$ is bounded by \eqref{I:close} we get by continuity of the flow
$$ \lim_{k\to \infty} |d_{n_k}(\theta_k)|=0,$$
a contradiction with \eqref{dnk}.

We next prove \eqref{I:far} if $\mu_n(\tau)\leq |\tau-\sigma|\leq 2\mu_n(\tau)$. We distinguish two cases. If for all $\theta$ in $[\tau,\sigma]$, $|d_{n}(\theta)|<\delta_0$, then \eqref{I:far} follows from the modulation estimate \eqref{modul2}. On the other hand, if there exists $\theta'\in[\tau,\sigma]$ such that $|d_n(\theta')|\geq \delta_0$, then for all $\theta\in [\tau,\sigma]$, $|d_n(\theta)|\geq \delta_1$. By \eqref{I:close},
$$ |Y_n(\tau)-Y_n(\sigma)|\leq C\mu_n(\tau)\leq C\leq \frac{C}{\delta_1}\int_{\tau}^{\sigma} d_n(s)ds,$$
and
$$ |\mu_n(\tau)-\mu_n(\sigma)|=\left|1-\frac{\mu_n(\sigma)}{\mu_n(\tau)}\right|\mu_n(\tau)\leq C\leq  \frac{C}{\delta_1}\int_{\tau}^{\sigma}d_n(s)ds.$$
The proof of the general case for \eqref{I:far} then follows by subdividing the interval.
\end{proof}

\subsection{Bound of Strichartz norms below the threshold}
\label{SS:bound}
As a consequence of Theorem \ref{T:compact}, we get the following:
\begin{corol}
\label{C:below_threshold}
 Let $M$ such that $0<M<\int |\nabla W|^2$. Then there exists a constant $C_M>0$ such that for any solution $u$ of \eqref{CP} defined on an interval $I$,
$$ \sup_{t\in I}\|\nabla u(t)\|_{L^2}^2+\frac{N-2}{2}\|\partial_t u(t)\|_{L^2}^2\leq M\Longrightarrow \|u\|_{S(I)}\leq C_M.$$
Furthermore, for all $\eps>0$, there exists a constant $C_{M,\eps}>0$ such that for any \emph{radial} solution $u$ of \eqref{CP} defined on an interval $I$,
$$ \sup_{t\in I}\|\nabla u(t)\|_{L^2}^2+\eps\|\partial_t u(t)\|_{L^2}^2\leq M\Longrightarrow \|u\|_{S(I)}\leq C_{M,\eps}.$$
\end{corol}
\begin{remark}
 In the lemma, $I$ does not have to be the maximal interval of existence $I_{\max}$ of $u$. The case $I=I_{\max}$ under stronger hypothesis (see Remark \ref{R:generalization}) is the object of \cite[Corollary 7.3]{KeMe08}. Corollary \ref{C:below_threshold} is a generalization of this result.
\end{remark}
\begin{remark}
\label{R:remark_corol}
 Corollary \ref{C:global} immediately follows from Corollary \ref{C:below_threshold} and the blow-up criterion \eqref{FBUC}. Note that in the radial case, \eqref{bound_nabla_corol_radial} implies by conservation of the energy that $\|\partial_t u(t)\|_{L^2}$ remains bounded as $t\to T_+$, and thus (using again \eqref{bound_nabla_corol_radial}), that there exists $\eps>0$ such that
\begin{equation*}
\limsup_{t\to T_+(u)} \|\nabla u(t)\|_{L^2}^2+\eps\|\partial_tu(t)\|_{L^2}^2<\|\nabla W\|_{L^2}^2, 
\end{equation*} 
which shows that the second part of Corollary \ref{C:below_threshold} applies.
\end{remark}

\begin{proof}[Sketch of proof]
\EMPH{Step 1. Contradiction argument}
We follow the scheme of the proof of \cite{KeMe08}.
For $M>0$, denote by $(\PPP_M)$ the property of the corollary. By the small data well-posedness theory, $(\PPP_M)$ holds for small positive $M$. Let $M_C=\sup\left\{M>0\;:\; (\PPP_M) \text{ holds}\right\}.$ Because of the solution $W$, $M_C\leq\int |\nabla W|^2$. We must show that 
$M_C=\int |\nabla W|^2$.

We argue by contradiction, assuming 
\begin{equation}
 \label{absurd_below}
M_C<\int |\nabla W|^2.
\end{equation} 
Let $\{u_n\}$ be a sequence of solutions to \eqref{CP}, $\{I_n\}_n$ a sequence of intervals such that $u_n$ is defined on $I_n$ and
\begin{equation*}
 \sup_{t\in I_n} \|\nabla u_n(t)\|^2_{L^2}+\frac{N-2}{2}\|\partial_t u_n(t)\|^2_{L^2}\leq M_C+\frac 1n,\quad \lim_{n\to +\infty} \|u_n\|_{S(I_n)}=+\infty.
\end{equation*} 
Taking a smaller $I_n$ if necessary, rescaling and translating in time we can assume that $I_n$ is a finite length interval $(a_n,b_n)$ with $[a_n,b_n]\subset I_{\max}(u_n)$, $a_n<0<b_n$ and 
\begin{gather}
\label{absurd_sequence1}
 \sup_{t\in [a_n,b_n]} \|\nabla u_n(t)\|^2_{L^2}+\frac{N-2}{2}\|\partial_t u_n(t)\|^2_{L^2}\leq M_C+\frac 1n,\\ 
\label{absurd_sequence2}
\lim_{n\to +\infty} \|u_n\|_{S((a_n,0))}=\lim_{n\to +\infty} \|u_n\|_{S((0,b_n))}=+\infty.
 \end{gather} 

\EMPH{Step 2. Existence of a critical element}
In this step we show that for \emph{any} sequence $\{u_n\}_n$ satisfying \eqref{absurd_sequence1} and \eqref{absurd_sequence2}, there exists a subsequence of $\{u_n\}_n$, parameters $\lambda_n>0$ and $x_n\in \RR^N$, and $(v_0,v_1)\in \hdot\times L^2$ such that, in $\hdot\times L^2$,
$$ \lim_{n\to\infty}\left(\lambda_n^{\frac{N-2}{2}}u_n\lf(0,\lambda_n x+x_n\right),\lambda_n^{\frac{N}{2}}\partial_t u_n\lf(0,\lambda_n x+x_n\right)\right)=(v_0,v_1).$$
Consider a profile decomposition $\lf\{U_{\lin}^j\rg\}_{j\geq 1}$, $\lf\{\lambda_{j,n};x_{j,n};t_{j,n}\rg\}_{j,n}$ for the sequence $(u_n(0),\partial_tu_n(0))$. Let $\{U^j\}_{j\geq 1}$ be the corresponding nonlinear profiles. 

At least one of the profiles is nonzero: otherwise this would contradict the fact that $\|u_n\|_{S(0,b_n)}$ tends to infinity.
We must show that there is only one nonzero profile. If not, we may assume, reordering the profiles, that for a small $\eps_0$,
$$ \lf\|\nabla U^{j}_{\lin}\lf(\frac{-t_{j,n}}{\lambda_{j,n}}\rg)\rg\|_{L^2}^2+\lf\|\partial_t U^{j}_{\lin}\lf(\frac{-t_{j,n}}{\lambda_{j,n}}\rg)\rg\|_{L^2}^2\geq 10\eps_0,\quad j=1,2.$$
By the small data theory, we get that for $j=1,2$ and $t$ in the domain of definition of $U^j$,  
\begin{equation}
\label{lower_bnd}
\lf\|\nabla U^{j}\lf(t\rg)\rg\|_{L^2}^2+\frac{N-2}{2}\lf\|\partial_t U^{j}\lf(t\rg)\rg\|_{L^2}^2\geq 2\eps_0.
\end{equation} 
Let $C_0$ be a large constant to be specified later, depending only on $\eps_0$ and $C_{M_C-\eps_0}$. For $n$ large, chose $T_n\in (0,b_n)$ such that 
\begin{equation}
\label{u_S_C0}
 \|u_n\|_{S(0,T_n)}=C_0.
\end{equation} 
Using \eqref{u_S_C0}, one can show with Proposition \ref{P:lin_NL} that for all $j$ such that $T_+(U^j)<\infty$,  for all large $n$,  $T_n<T_+(U^j)\lambda_{j,n}+t_{j,n}$. Taking into account that there is a finite number of such $j$, we have that for all large $n$:
\begin{equation}
\label{T_n_defined}
T_n<\inf_{j\geq 1}\left(T_+(U^j)\lambda_{j,n}+t_{j,n}\right).
\end{equation} 
(with the convention that the righthand side is infinite if $T_+(U^j)=+\infty$).
Define
$$ \SSS_n=\sup_{j}\int_0^{T_n} \int_{\RR^N} \left|U^j\left(\frac{T_n-t_{j,n}}{\lambda_{j,n}},\frac{x-x_{j,n}}{\lambda_{j,n}}\right)\right|^{\frac{2(N+1)}{N-2}}\frac{dx\,dt}{\lambda_{j,n}^{N+1}}.$$
By \eqref{u_S_C0} and Proposition \ref{P:lin_NL}, the sequence $\left\{\SSS_n\right\}_n$ is bounded.
We will show
\begin{equation}
\label{limsup_F_n} 
\limsup_{n\to \infty} \SSS_n\leq C_{M_C-\eps_0},
\end{equation} 
where the constant $C_{M_C-{\eps_0}}$ is given by the property $(\PPP_{M_C-\eps_0})$. Indeed using \eqref{u_S_C0}, Proposition \ref{P:lin_NL} and the orthogonality of the parameters $\left\{\lambda_{j,n};x_{j,n};t_{j,n}\right\}$, we get that any sequence of times $\{\sigma_n\}_n$ such that $0<\sigma_n<T_n$ satisfies the following Pythagorean expansion:
\begin{multline*}
\lim_{n\to \infty} \bigg(\|\nabla u_n(\sigma_n)\|^2_{L^2}+\frac{N-2}{2}\|\partial_t u_n(\sigma_n)\|^2_{L^2}
-
\sum_{j=1}^J\lf\|\nabla U^{j}\lf(\frac{\sigma_n-t_{j,n}}{\lambda_{j,n}}\rg)\rg\|_{L^2}^2\\
-\frac{N-2}{2}\sum_{j=1}^J\lf\|\partial_t U^{j}\lf(\frac{\sigma_n-t_{j,n}}{\lambda_{j,n}}\rg)\rg\|_{L^2}^2-\lf\|\nabla w_n^{J}\lf(\sigma_n\rg)\rg\|_{L^2}^2-\frac{N-2}{2}\lf\|\partial_t w_n^J\lf(\sigma_n\rg)\rg\|_{L^2}^2\bigg)=0.
\end{multline*} 
Combining with \eqref{absurd_sequence1} and \eqref{lower_bnd}, we get that the bound 
$$ \forall j,\quad \sup_{t\in [0,T_n]} \lf\|\nabla U^{j}\lf(\frac{t-t_{j,n}}{\lambda_{j,n}}\rg)\rg\|_{L^2}^2+\frac{N-2}{2}\lf\|\partial_t U^{j}\lf(\frac{t-t_{j,n}}{\lambda_{j,n}}\rg)\rg\|_{L^2}^2\leq M_C-\eps_0$$
holds for large $n$. Thus \eqref{limsup_F_n} follows from $\left(\PPP_{M_C-\eps_0}\rg)$.

By the argument in the proof of Lemma 4.9 in \cite{KeMe06}, using again the orthogonality of the parameters, we can show that \eqref{absurd_sequence1} and \eqref{limsup_F_n} imply that there exists a constant $C_1$, depending only on $\eps_0$, $M_C$ and $C_{M_C-\eps_0}$ such that
$$ \limsup_{n\to \infty}\|u_n\|_{S(0,T_n)}\leq C_1.$$
Chosing the constant $C_0$ in \eqref{u_S_C0} strictly greater than $C_1$ yields a contradiction, which shows that there is only one nonzero profile, say $U_1$, in the profile decomposition of $(u_n(0),\partial_t u_n(0))$. Similarly, we can show that the dispersive part $(w_{0n}^1,w_{1n}^1)$ tends to $0$ in $\hdot\times L^2$. It remains to show that $-t_{1,n}/\lambda_{1,n}$ is bounded, which follows from the conditions $\|u_n\|_{S(0,b_n)}\to +\infty$ (which implies that $-t_{1,n}/\lambda_{1,n}$ is bounded from above) and $\|u_n\|_{S(a_n,0)}\to +\infty$ (which implies that it is bounded from below).

\EMPH{Step 3. Compactness of the critical element and end of the proof}
Let $v$ be the solution to \eqref{CP}  with initial condition $(v_0,v_1)$ and $(T_-(v),T_+(v))$ its maximal interval of existence. Then $v$ inherits the following properties from $u_n$:
\begin{gather}
\label{bound_nabla_v}
\sup_{T_-(v)<t<T_+(v)} \|\nabla v(t)\|_{L^2}^2+\frac{N-2}{2}\|\partial_t v(t)\|_{L^2}^2\leq M_C\\
\label{v_non_scat}
\|v\|_{S(T_-(v),0)}=\|v\|_{S(0,T_+(v))}=+\infty.
\end{gather}
Indeed, if $t\in (T_-(v),T_+(v))$, \eqref{absurd_sequence2} shows that for large $n$, $\lambda_n t\in (a_n,b_n)$. Using that by perturbation theory,
$$ \left(\lambda_n^{\frac{N}{2}-1}u_n(\lambda_n t,\lambda_nx+x_n),\lambda_n^{\frac{N}{2}}\partial_t u_n(\lambda_n t,\lambda_nx+x_n)\rg)\underset{n\to\infty}{\longrightarrow} (v(t),\partial_t v(t))\text{ in }\hdot\times L^2,$$
we get that \eqref{bound_nabla_v} follows from \eqref{absurd_sequence1}.

If \eqref{v_non_scat} does not hold, say $\|v\|_{S(0,T_+(v))}<\infty$, then $T_+(v)=+\infty$, and for large $n$, $T_+(u_n)=+\infty$, and $\|u_n\|_{S(0,+\infty)}\leq 2\|v\|_{S(0,+\infty)}<+\infty$, contradicting \eqref{absurd_sequence2}. 

Let us show that $v$ is compact up to modulation. By a standard lifting argument, it is sufficient to show that for any sequence $\{t_n\}_n$, $t_n\in I_{\max}(v)$, there exist sequences $\{\lambda_n\}_n$ and $\{x_n\}_n$ and a subsequence of $\left\{\left(\frac{1}{\lambda_n^{\frac{N-2}{2}}}v\left(t_n,\frac{\cdot-x_n}{\lambda_n}\right),\frac{1}{\lambda_n^{\frac{N}{2}}}\partial_t v\left(t_n,\frac{\cdot-x_n}{\lambda_n}\right)\right)\right\}_n$ that converges in $\hdot\times L^2$. By \eqref{bound_nabla_v} and \eqref{v_non_scat}, the sequences of solutions $\{u_n\}$ with initial data $(v(t_n),\partial_t v(t_n))$ satisfies \eqref{absurd_sequence1} and \eqref{absurd_sequence2} for a suitable choice of $a_n$ and $b_n$. By Step 2, we get the desired result.

By Claim \ref{C:value_W} and Theorem \ref{T:compact}, the only solution compact up to modulation satisfying \eqref{bound_nabla_v} with $M_C<\int|\nabla W|^2$ is $0$, which concludes the proof.

In the radial case, the proof above works as well, replacing all constants $\frac{N-2}{2}$ by $\eps$, and Theorem \ref{T:compact} by \cite[Theorem 2]{DuKeMe09P} which states that the only radial solutions of \eqref{CP} that are compact up to modulations are (up to scaling and sign change) $0$ and $W$.
\end{proof}

\appendix

\section{Modulation theory}
\label{A:modulation}
In this appendix we show Claim \ref{C:trapping} and Lemma \ref{L:modulation}. Consider a solution $u$ of \eqref{CP} which satisfies \eqref{cond_energy}.

If $x=(x_1,\ldots,x_N)\in \RR^N$, denote by $\ovx=(x_2,\ldots,x_N)\in \RR^{N-1}$. Let 
$$\tilde{u}(t)=u\lf(t,\sqrt{1-\ell^2}\,x_1,\ovx\rg),\quad \tilde{u}_1(t)=\lf(\partial_t u\rg)\lf(t,\sqrt{1-\ell^2}\,x_1,\ovx\rg)+\lf(\ell\partial_{x_1}u\rg)\lf(t,\sqrt{1-\ell^2}\,x_1,\ovx\rg).$$
By Claim \ref{C:value_W}, we get, in view of \eqref{cond_energy}, 
\begin{equation}
\label{eq_tilde}
E(\tilde{u}_0(t),\tilde{u}_1(t))=E(W,0),\quad d_{\ell}(t)=\sqrt{1-\ell^2}\left(\int |\nabla \tu(t)|^2+\int \lf(\tu_1(t)\rg)^2-\int |\nabla W|^2\right),
 \end{equation} 
where $d_{\ell}$ is defined by \eqref{def_dl}. Thus if $d_{\ell}(0)=0$, we get
$$ \int |\tilde{u}(0)|^{\frac{2N}{N-2}}=\int |W|^{\frac{2N}{N-2}},\quad \int |\nabla \tu(0)|^2=\int |\nabla W|^2-\int |\tilde{u}_1|^2,$$
and the fact that $W$ is a minimizer for the Sobolev inequality shows, as usual, that there exist $x_0$, $\lambda_0$ and a sign $\pm$ such that 
$$\tilde{u}(0)=\pm\frac{1}{\lambda_0^{\frac{N-2}{2}}}W\lf(\frac{x-x_0}{\lambda_0}\rg),\quad \tilde{u}_1(0)=0.$$
Coming back to the solution $u$, we get
\begin{equation*}
u_0=\pm \frac{1}{\lambda_0^{\frac{N-2}{2}}}W_{\ell}\lf(0,\frac{x-x_0}{\lambda_0}\rg),\quad u_1=\pm \frac{1}{\lambda_0^{\frac{N}{2}}}\partial_t W_{\ell}\lf(0,\frac{x-x_0}{\lambda_0}\rg).
 \end{equation*}
Thus 
\begin{equation*}
 u(t,x)=\pm \frac{1}{\lambda_0^{\frac{N-2}{2}}}W_{\ell}\lf(\frac{t}{\lambda_0},\frac{x-x_0}{\lambda_0}\rg),
\end{equation*}
which shows the first point of Claim \ref{C:trapping}. The two other points follow by continuity of $d_{\ell}(t)$ with respect to $t$ and the intermediate value theorem.

Let us show Lemma \ref{L:modulation}. Assume that for a small $\delta_0$, $|d_{\ell}(t)|<\delta_0$. 
Then by \eqref{eq_tilde}, $\int |\nabla \tilde{u}(t)|^2$ is close to $\int |\nabla W|^2$, $\int |\tu(t)|^{\frac{2N}{N-2}}$ is close to $\int |W|^{\frac{2N}{N-2}}$  and $\int |\tilde{u}_1(t)|^2$ is small. In particular, by the characterization of $W$ (\cite{Au76,Ta76}), $\tilde{u}$ is close to $W$ or $-W$ after a space translation and a scaling. To fix ideas, we assume that $\tilde{u}$ is close to $W$ after space translation and scaling. As stated in \cite[Claim 3.5]{DuMe08}, by a standard argument using the implicit function theorem (see \cite[Claim 3.5]{DuMe09a} for a proof in a very similar case), one can show that there exists $\lambda(t)$, $\tilde{x}(t)$ such that
\begin{equation*}
 \lambda(t)^{\frac{N-2}{2}}\tilde{u}(t,\lambda(t)x+\tilde{x}(t))\in\lf\{\partial_{x_1}W,\ldots \partial_{x_N}W, x\cdot \nabla W+\frac{N-2}{2}W\rg\}^{\bot},
\end{equation*}
where the orthogonality has to be understood in $\hdot(\RR^N)$. Letting 
\begin{equation*}
 \alpha(t)=\frac{1}{\int |\nabla W|^2}\left(\int \lambda(t)^{\frac{N}{2}}\nabla \tilde{u}\big(t,\lambda(t)x+\tilde{x}(t)\big)\cdot\nabla W(x)\,dx\right)-1,
\end{equation*}
we obtain 
\begin{equation*}
 \lambda(t)^{\frac{N-2}{2}}\tilde{u}(t,\lambda(t)x+\tx(t))=(1+\alpha(t))W(x)+\tilde{f}(t,x).
\end{equation*}
Furthermore:
\begin{equation}
\label{ortho_tildef}
\tilde{f}(t)\perp \vect\left\{W,\,\partial_{x_1}W,\ldots, \partial_{x_N}W, \,x\cdot \nabla W+\frac{N-2}{2}W\right\}. 
 \end{equation}
By the proof of (3.19) in \cite[Lemma 3.7]{DuMe08}, we get the estimates
\begin{equation}
\label{estimates}
|\alpha(t)|\approx \lf\|\nabla\lf(\alpha W+\tilde{f}\rg)\rg\|_{L^2}\approx \lf\|\nabla \tf(t)\rg\|_{L^2}+\lf\|\tu_1(t)\rg\|_{L^2}\approx |d_{\ell}(t)|.
\end{equation}
In \cite[(3.19)]{DuMe08}, $(\tilde{u}(t),\tilde{u}_1(t))$ is replaced by a couple $(u(t),\partial_tu(t))$, where $u$ is a solution to \eqref{CP} such that 
$$E(u_0,u_1)=E(W,0)\text{ and }\left|\int |\nabla u(t)|^2\,dx+\int (\partial_t u(t))^2\,dx-\int |\nabla W|^2\,dx\right|<\delta_0.$$ However, the fact that $u$ is a solution is not used in the proof of estimates \eqref{estimates}, where the time variable is only a parameter. Indeed \eqref{estimates} follows from the fact that $E(\tilde{u}(t),\tilde{u}_1(t))=E(W,0)$, $d_{\ell}(t)$ is small and $\tilde{f}(t)$ satisfies the orthogonality conditions \eqref{ortho_tildef}. It remains to show the estimates \eqref{modul2} on the derivatives of the parameters. The proof is very similar to the one of (3.20) in \cite[Lemma 3.7]{DuMe08}\footnote{in the cited paper, the function $\mu(t)$ is the analogue of our parameter $1/\lambda(t)$}. We sketch it for the sake of completness.

Write $\tilde{u}(t,x)=\frac{1}{\lambda(t)^{\frac{N-2}{2}}}U\lf(t,\frac{x-\tilde{x}(t)}{\lambda(t)}\rg)$, where 
\begin{equation}
\label{def_U}
U(t,x)=\lf(1+\alpha(t)\rg)W+\tilde{f}.
\end{equation} 
By \eqref{estimates},
\begin{equation}
\label{estimates2}
\lf\|\tilde{u}_1(t)\rg\|_{L^2}\leq C\lf|d_{\ell}(t)\rg|.
\end{equation}
Furthermore
\begin{multline*}
\tilde{u}_1(t)=
\partial_t \tilde{u}(t)+\frac{\ell}{\sqrt{1-\ell^2}}\partial_{x_1}\tilde{u}(t)=\\
-\frac{N-2}{2}\,\frac{\lambda'}{\lambda^{\frac{N}{2}}}U\lf(t,\frac{x-\tilde{x}(t)}{\lambda}\rg)+\frac{1}{\lambda^{\frac{N-2}{2}}}\partial_t U\lf(t,\frac{x-\tilde{x}(t)}{\lambda}\rg)
-\frac{\lambda'}{\lambda^{\frac{N+2}{2}}}(x-\tilde{x}(t))\cdot\nabla U\lf(t,\frac{x-\tilde{x}(t)}{\lambda}\rg)\\ -\frac{1}{\lambda^{\frac{N}{2}}}\tilde{x}'(t)\cdot\nabla U\lf(t,\frac{x-\tilde{x}(t)}{\lambda}\rg)+\frac{\ell}{\sqrt{1-\ell^2}\lambda^{\frac{N}{2}}}\partial_{x_1}U\lf(t,\frac{x-\tilde{x}(t)}{\lambda}\rg).
\end{multline*}
By \eqref{def_U},
\begin{multline}
\label{scaled_tu1}
 \lambda^{\frac{N}{2}}\tilde{u}_1\lf(t,\lambda x+\tilde{x}(t)\rg)=-\lambda'\left(\frac{N-2}{2}U+x\cdot\nabla U\rg)+\lambda\partial_t U-\tilde{x}'(t)\cdot\nabla U+\frac{\ell}{\sqrt{1-\ell^2}}\partial_{x_1}U\\
=-\lambda'\left(\frac{N-2}{2}W+x\cdot\nabla W\rg)+\lambda\alpha' W-\tilde{x}'(t)\cdot\nabla W+\frac{\ell}{\sqrt{1-\ell^2}}\partial_{x_1}W+\lambda\partial_t \tilde{f}+g,
\end{multline}
where by definition
$$ g=\left[-\lambda'\left(\frac{N-2}{2}+x\cdot\nabla \rg)-\tilde{x}'(t)\cdot\nabla +\frac{\ell}{\sqrt{1-\ell^2}}\partial_{x_1}\right]\lf(\alpha W+\tilde{f}\rg).$$
Notice that
\begin{multline*}
\left\|\frac{1}{1+|x|}g\right\|_{L^2}\leq C\left(\lf|\lambda'\rg|+\lf|\tilde{x}'-\frac{\ell}{\sqrt{1-\ell^2}}\vec{e}_1\rg|\rg)\lf(|\alpha|+\lf\|\tilde{f}\rg\|_{\hdot}\rg)\\
\leq C\left(\lf|\lambda'(t)\rg|+\lf|\tilde{x}'-\frac{\ell}{\sqrt{1-\ell^2}}\vec{e}_1\rg|\rg)d_{\ell}(t).
\end{multline*}
Taking the scalar product of \eqref{scaled_tu1} in $L^2$ with $\Delta\partial_{x_1}W$,\ldots,$\Delta\partial_{x_n} W$, $\Delta\lf(\lf(\frac{N-2}{2}+x\cdot\nabla\rg)W\rg)$, $\Delta W$ and using that $\partial_t \tilde{f}$ is orthogonal in $L^2(\RR^N)$ with all these functions, we obtain, in view of \eqref{estimates2},
\begin{equation*}
 \lf|\lambda'(t)\rg|+\lf|\tilde{x}'(t)-\frac{\ell}{\sqrt{1-\ell^2}}\vec{e}_1\rg|+\lambda(t)|\alpha'(t)|\leq Cd_{\ell}(t)+ C\left(\lf|\lambda'(t)\rg|+\lf|\tilde{x}'(t)-\frac{\ell}{\sqrt{1-\ell^2}}\vec{e}_1\rg|\rg)d_{\ell}(t).
\end{equation*}
Assuming that $d_{\ell}(t)$ is small enough, which may be obtained by taking a smaller $\delta_0$, we obtain
\begin{equation*} 
 \lf|\lambda'(t)\rg|+\lf|\tilde{x}'(t)-\frac{\ell}{\sqrt{1-\ell^2}}\vec{e}_1\rg|+\lambda|\alpha'(t)|\leq Cd_{\ell}(t),
\end{equation*}
which yields estimates \eqref{modul2} taking $\tilde{x}(t)=\lf(\frac{1}{\sqrt{1-\ell^2}}x_1(t),x_2(t),\ldots,x_n(t)\rg)$. The proof of Lemma \ref{L:modulation} is complete.

\bibliographystyle{alpha} 
\bibliography{blowup}

\end{document}